\documentclass[11pt]{amsart}
\usepackage{amsthm,amssymb,amsmath}

\numberwithin{equation}{section}

\newtheorem{theorem}{Theorem}[section]
\newtheorem{lemma}[theorem]{Lemma}

\newtheorem{corollary}[theorem]{Corollary}
\newtheorem{proposition}[theorem]{Proposition}

\theoremstyle{definition}

\newtheorem{definition}[theorem]{\sc Definition}
\newtheorem{example}[theorem]{\bf Example}
\newtheorem{remark}[theorem]{\bf Remark}

\newcommand{\supp}{{\rm supp}}

\newcommand{\Real}{{\rm Re}}
\newcommand{\Imag}{{\rm Im}}
\newcommand{\dist}{\mbox{dist}}

\newcommand{\spana}{{\mbox{span}}}

\newcommand{\Id}{{\rm Id}}

\newcommand{\trig}{{\rm 0}}

\newcommand{\Div}{{\rm Div}}

\newcommand{\Hom}{{\rm Hom_{\mathcal O}}}

\newcommand{\Rea}{{\rm \scriptscriptstyle Re}}

\DeclareRobustCommand{\gobblefour}[4]{}
\newcommand*{\SkipTocEntry}{\addtocontents{toc}{\gobblefour}}

\include{diagxy}

\begin{document}

%\dedicatory{To the memory of Hans Grauert} 

%\title[Algebras of holomorphic functions]{Towards Oka-Cartan theory for algebras of fibrewise bounded holomorphic functions on \\ coverings of Stein manifolds}
\title[Algebras of holomorphic functions]{Towards Oka-Cartan theory for algebras of holomorphic functions on coverings of Stein manifolds II}



\author{A.~Brudnyi}

\address{Department of Mathematics and Statistics, University of Calgary, 
Calgary, Canada}

\email{albru@math.ucalgary.ca}

\author{D.~Kinzebulatov}

\address{The Fields Institute, Toronto, Canada}

\email{dkinzebu@fields.utoronto.ca}

\thanks{Research of the first and the second authors were partially supported  by NSERC}

\begin{abstract}
We establish basic results of complex function theory within certain algebras of holomorphic functions
on coverings of Stein manifolds (such as algebras of Bohr's holomorphic almost
periodic functions on tube domains or algebras of all fibrewise bounded holomorphic functions
arising, e.g., in the corona problem for $H^\infty$). 
In particular, in this context we obtain results on
holomorphic extension from complex submanifolds, properties of divisors, corona type theorems, 
holomorphic analogues of the Peter-Weyl approximation theorem,
Hartogs type theorems, characterizations of uniqueness sets, etc.
%The model examples of these subalgebras are: (1) algebra of Bohr's holomorphic almost
%periodic functions on tube domains; (2) algebra of all fibrewise bounded holomorphic functions
%(e.g., arising in the corona problem for $H^\infty$).  
Our proofs are based on analogues of Cartan theorems
A and
B for coherent type sheaves on maximal ideal spaces of these algebras proved in Part I.
\end{abstract}

\maketitle

\section{Introduction}
\label{introsect}

In 1930--1950 methods of sheaf theory radically transformed the theory of holomorphic functions of several variables which led to solution of a number of fundamental and long standing problems including problems of holomorphic interpolation, Cousin problems, the Levi problem on characterization of domains of holomorphy, etc. 
%Their work culminated in celebrated Cartan theorems A and B, and effectively created the modern function theory of several complex variables. 
Since then the theory started to play a foundational and unifying role in modern mathematics, with implications for analytic geometry, automorphic forms, Banach algebras, etc.
Further development of the theory was motivated, in part, by the problems requiring to study 
properties of holomorphic functions satisfying additional restrictions (such as uniform boundedness along certain subsets of their domains or certain growth `at infinity'). In particular, the principal question arose of whether the fundamental problems of the function theory of several complex variables can be solved within a proper subclass of the algebra $\mathcal O(X)$ of holomorphic functions on a Stein manifold $X$.
In the present paper we address this question for subalgebras of $\mathcal O(X)$ subject to the following definition.

\begin{definition}
\label{defafunc}
A holomorphic function $f$ defined on a regular covering $p:X \rightarrow X_0$ of a connected complex manifold $X_0$ with a deck transformation group $G$ is called a \textit{holomorphic $\mathfrak a$-function} if 

(1) $f$ is bounded on subsets $p^{-1}(U_0)$, $U_0 \Subset X_0$, and

(2) for each $x \in X$ the function $G\ni g\mapsto f(g \cdot x)$ belongs to
a fixed closed unital 
%self-adjoint (i.e.~closed with respect to complex conjugation) 
subalgebra $\mathfrak a:=\mathfrak a(G)$ of the algebra $\ell_\infty(G)$ of bounded complex functions on $G$ (with pointwise multiplication and $\sup$-norm) that is invariant with respect to the action of $G$ on $\mathfrak a$ by right translations: 
$$
u \in \mathfrak a,~~g \in G \quad \Rightarrow \quad R_g u \in \mathfrak a,
$$
where $R_g(u)(h):=u(h g)$, $h\in G$.

%$R_g(u)(h):=u(h g)$, $u\in \mathfrak a$, $g,h\in G$.

We endow the subalgebra $\mathcal O_{\mathfrak a}(X) \subset \mathcal O(X)$ of holomorphic $\mathfrak a$-functions with the Fr\'{e}chet topology of uniform convergence on subsets $p^{-1}(U_0)$,  $U_0\Subset X_0$.

\end{definition}

The model examples of algebras $\mathcal O_{\mathfrak a}(X)$ are: 

\smallskip

(1) Algebra of Bohr's holomorphic almost periodic functions on a tube domain $T \subset \mathbb C^n$,
%, i.e.~the uniform limits of exponential polynomials
%\begin{equation}
%\label{expoly}
%$T \ni z\rightarrow\sum_{k=1}^m c_ke^{i \langle z,\lambda_k\rangle}$, $c_k \in \mathbb C $, $\lambda_k \in \mathbb R^n$,
%\end{equation}
%where $\langle\cdot,\lambda_k\rangle$ stands for the Hermitian inner product on $\mathbb C^n$ 
see Example \ref{holap} below;
\smallskip

(2) Algebra of all fibrewise bounded holomorphic functions on $X$, see Example \ref{firstex}(1) below. (If $X_0$ is a compact complex manifold, then this algebra coincides with algebra $H^\infty(X)$ of bounded holomorphic functions on $X$). 

\smallskip

See Section \ref{examples} for other examples.

\smallskip

In \cite{BK8}, using some constructions of \cite{Lemp}, we derived analogues of Cartan's theorems A and B for coherent-type sheaves on the fibrewise compactification $c_{\mathfrak a}X$ of the covering $X$ of a Stein manifold $X_0$, a topological space having certain features of a complex manifold (see Section \ref{mainsect} for details). This allows one to transfer in a systematic manner most of the significant results of
the classical theory of holomorphic functions on Stein manifolds to holomorphic functions in algebras $\mathcal O_{\mathfrak a}(X)$.
In particular, in the present paper we establish

\medskip

-- results on holomorphic interpolation on complex a-submanifolds (i.e.~complex submanifolds of $X$ determined by holomorphic $\mathfrak a$-functions) (subsection \ref{sectionext}), 

\medskip

-- tubular neighbourhood theorem for complex $\mathfrak a$-submanifoilds (subsection \ref{sectionext}),

\medskip

-- properties of holomorphic line $\mathfrak a$-bundles and their divisors (subsection \ref{sectdivisors}),
%, and of the divisors determined by functions in $\mathcal O_{\mathfrak a}(X)$, including second Cousin problem,

\medskip

-- characterization of uniqueness sets for functions in $\mathcal O_{\mathfrak a}(X)$ (subsection \ref{sectbohr}),

\medskip

-- Leray integral representation formulas and Hartogs theorems for functions in $\mathcal O_{\mathfrak a}(X)$ (subsection \ref{leray}), 

\medskip

-- holomorphic Peter-Weyl theorems for $\mathcal O_{\mathfrak a}(X)$ (subsection \ref{approx}),

\medskip

-- Cartan's theorems A and B for coherent-type sheaves on complex $\mathfrak a$-submanifolds (subsection \ref{cartansubmsect}),

\medskip

-- Dolbeault isomorphisms on complex $\mathfrak a$-submanifolds (subsection \ref{derhamsect}),

\medskip

-- corona theorems for algebras $\mathcal O_{\mathfrak a}(X)$ (subsection \ref{sectionmaxid}).

%(2) Tubular neighbourhood theorem and a corona-type theorem for $\mathcal O_{\mathfrak a}(X)$. Also, we obtain integral representation formulas, Hartogs-type theorems, holomorphic analogues of the Peter-Weyl approximation theorem, describe uniqueness sets of holomorphic $\mathfrak a$-functions, etc. 

%One of the main results of this paper is the presentation of Bohr's holomorphic almost periodic functions in the form of Definition \ref{defafunc}. This gives one a way to systematically apply to these functions the methods of the modern multidimensional complex analysis (see Example \ref{holap}).

%In Section \ref{cartansubmsect} we transfer our Cartan-type theorems as well as some other methods of the Oka-Cartan theory on $c_{\mathfrak a}X$ developed in \cite{BK8} %(and relying, in part, on some ideas of J.~Leiterer \cite{Lt} and L.~Lempert \cite{Lemp}) %to transfer the Cartan-type theorems  on $c_{\mathfrak a}X$ 
%to certain complex submanifolds of $X$ determined by holomorphic $\mathfrak a$-functions. As a consequence, most of our results for algebra $\mathcal O_{\mathfrak a}(X)$ will be transferred to holomorphic $\mathfrak a$-functions on such submanifolds. 

%(The requirement that the complex analytic $\mathfrak a$-subsets of $X$ that we consider (i.e.~the complex analytic subsets determined by holomorphic $\mathfrak a$-functions) are \textit{regular} is essential, as for the \textit{singular} sets there are simple counterexamples, e.g.~to interpolation within $\mathcal O_{\mathfrak a}(X)$, and hence to Cartan theorems A and B on $c_{\mathfrak a}X$.)

\begin{example}[\textit{Holomorphic almost periodic functions}]
\label{holap}
The theory of almost periodic functions was created in the 1920s by H.~Bohr and nowadays is widely used in
various areas of mathematics including number theory, harmonic analysis, differential equations (e.g.,~KdV equation), etc. 

Let us recall the S.~Bochner (equivalent) definition of almost periodicity: a function $f \in \mathcal O(T)$ on a tube domain $T=\mathbb R^n+i\Omega \subset \mathbb C^n$, $\Omega \subset \mathbb R^n$ is open and convex, is called {\em holomorphic almost periodic} if the family of its translates $\{z \mapsto f(z+s),~z \in T\}_{s \in \mathbb R^n}$ is relatively compact in the topology of uniform convergence on tube subdomains $T'=\mathbb R^n+i\Omega'$, $\Omega' \Subset \Omega$. 
The principal result of Bohr's theory (see \cite{Bo}) is the approximation theorem which states that every holomorphic almost periodic function
is the uniform limit (on tube subdomains $T'$ of $T$) of exponential polynomials
\begin{equation}
\label{expoly}
z\mapsto\sum_{k=1}^m c_ke^{i \langle z,\lambda_k\rangle}, \quad z\in T,\quad c_k \in \mathbb C, \quad \lambda_k \in \mathbb R^n,
\end{equation}
where $\langle\cdot,\lambda_k\rangle$ is the Hermitian inner product on $\mathbb C^n$.

The classical approach to study of holomorphic almost periodic functions exploits the fact that $T$ is the trivial bundle with base $\Omega$ and fibre $\mathbb R^n$ (e.g.,~as in the characterization of almost periodic functions in terms of their Jessen functions defined on $\Omega$, see \cite{Sh,Lv,JT,Rn,FR,Torn} and references therein). 
In our approach, we consider $T$ as a regular covering $p:T \rightarrow T_0\, (:=p(T) \subset \mathbb C^n)$ with the deck transformation group $\mathbb Z^n$, where
%\begin{equation*}
$p(z):=\bigl(e^{i z_1}, \dots, e^{i z_n}\bigr)$, $z=(z_1,\dots,z_n) \in T$
%\end{equation*}
(if $n=1$, then this is a complex strip covering an annulus in $\mathbb C$), and obtain 
%(Theorem \ref{equivapthm}) that
%holomorphic almost periodic functions on $T$ are precisely the functions $f\in\mathcal O(T)$ bounded on subsets $p^{-1}(U_0)$, $U_0 \Subset T_0$, and such that for each $x\in T$ the functions $f_x(t):=f(x+t)$, $t\in\mathbb Z^n$, belong to algebra $AP(\mathbb Z^n)$ which is generated by exponential polynomials
%$t \rightarrow \sum_{k=1}^m c_ke^{i\langle t,\lambda_k \rangle}$, $t \in \mathbb Z^n$, $\lambda \in \mathbb R^n$:

\begin{theorem}
\label{equivapthm} 
A function $f \in \mathcal O(T)$ is almost periodic
if and only if $f \in \mathcal O_{AP}(T)$.
\end{theorem}

\noindent (Here $AP=AP(\mathbb Z^n)$ is the algebra of von Neumann's almost periodic functions on group $\mathbb Z^n$, see definition in Example \ref{exm}(2) below.)
This result enables us to regard holomorphic almost periodic functions on $T$ as:

\begin{itemize}
\item[(a)] holomorphic sections of a certain holomorphic Banach vector bundle on $T_0$;

\item[(b)] holomorphic-type functions on the fibrewise Bohr compactification of the covering $p:T\rightarrow T_0$.% a topological space having some properties of a complex manifold.
\end{itemize}
As a result, we can apply the methods of multidimensional complex function theory (in particular, analytic sheaf theory and Banach-valued complex analysis) to study holomorphic almost periodic functions.
% (in addition to the variety of techniques already used in the theory of holomorphic almost periodic functions, such as Fourier analysis-type arguments, cf. \cite{Bo}, \cite{Ln}, \cite{Lv}, or arguments based on the properties of Monge-Amper\'{e} currents, cf. \cite{F2R}, etc).
In particular, even in this classical setting, we obtain new results on holomorphic almost periodic
extension, recovery of almost periodicity of a holomorphic function 
from that for its trace to a real periodic hypersurface, etc.
We also show that some results known for holomorphic almost periodic functions are, in fact, valid for a general algebra $\mathcal O_{\mathfrak a}(X)$.

It is interesting to note that already in his monograph \cite{Bo} H.~Bohr uses equally often the aforementioned ``trivial fibre bundle'' and ``regular covering'' points of view on a complex strip. We mention also that the Bohr compactification  of a tube domain $\mathbb R^n +i\Omega$ in the form $b\mathbb R^n +i\Omega$, where $b\mathbb R^n$ is the Bohr compactification of group $\mathbb R^n$, was used earlier in \cite{Fav, Fav2,Gri}.
\end{example}

\begin{example}
\label{firstex}
(1) By definition, every $\mathcal O_{\mathfrak a}(X) \subset \mathcal O_{\ell_\infty(G)}(X)$; here $G$ is the deck transformation group of covering $p:X \rightarrow X_0$. 

Algebra $\mathcal O_{\ell_\infty(G)}(X)$ arises, e.g., in study of holomorphic $L^2$-functions on coverings of pseudoconvex manifolds \cite{GHS,Br2, Br5,La}, Caratheodory hyperbolicity (the Liouville property) of $X$ \cite{LS,Lin}, corona-type problems for bounded holomorphic functions on $X$ \cite{Br}. 
Earlier, some methods similar to those developed in the present paper were elaborated for algebra $\mathcal O_{\ell_\infty(G)}(X)$ 
in \cite{Br}-\cite{Br4}, \cite{Br8} in connection with corona-type problems for some subalgebras of bounded holomorphic functions on coverings of bordered Riemann surfaces, Hartogs-type theorems, integral representation of holomorphic functions of slow growth on coverings of Stein manifolds, etc.

A recent confirmation of potential productivity of the sheaf-theoretic approach to corona problem for $H^\infty$ comes from the recent paper \cite{BrBanach} on Banach-valued holomorphic functions on the unit disk $\mathbb D \subset \mathbb C$ having relatively compact images.

\smallskip

(2) Let $\mathfrak a:=c(G)\subset\ell_\infty(G)$ (with ${\rm card~} G=\infty$) be the subalgebra of bounded complex functions on $G$ that admit continuous extensions to the one-point compactification of $G$. 
Then  $\mathcal O_{c}(X)$ consists of holomorphic functions that have fibrewise limits at `infinity'.
\end{example}

For other examples of algebras $\mathfrak a$ and $\mathcal O_{\mathfrak a}(X)$ see subsections \ref{exm} and \ref{holap3}.

%\subsection*{Notation and definitions}

\medskip

In the formulation of our main results we use the following definitions.

Assume that $X_0$ is equipped with a path metric $d_0$ determined by a (smooth) hermitian metric. We can lift $d_0$ uniquely to a path metric $d$ on $X$.
%Let $d$ be a semi-metric on $X$ defined by $$d(x_1,x_2):=d_0\bigl(p(x_1),p(x_2)\bigr), \quad x_1,x_2 \in X.$$ 

A function $f \in C(X)$ is called a {\em continuous $\mathfrak a$-function} if it is bounded and uniformly continuous with respect to metric $d$ on subsets $p^{-1}(U_0)$, $U_0 \Subset X_0$, and is
such that for each $x\in X$ the function $G\ni g\mapsto f(g \cdot x)$ belongs to  $\mathfrak a$.

We denote by $C_{\mathfrak a}(X)$ the algebra of continuous $\mathfrak a$-functions on $X$. It is easily seen that $C_{\mathfrak a}(X)$ does not depend on the choice of the hermitian metric on $X_0$ and $C_{\mathfrak a}(X) \cap \mathcal O(X)=\mathcal O_{\mathfrak a}(X).$

If $D_0 \Subset X_0$ is a subdomain, we set $D:=p^{-1}(D_0) \subset X$ and define $C_{\mathfrak a}(\bar{D})$ to be the subalgebra of complex functions $f$ on $\bar{D}$ (the closure of $D$) that are bounded and uniformly continuous with respect to semi-metric $d$ and
such that for each $x\in \bar{D}_0$ functions $G\ni g\mapsto f(g \cdot x)$ belong to  $\mathfrak a$.

%Let $\Lambda_{c}^{t,s} (X)$ denote the space of smooth $(t,s)$-forms with compact supports in $X$, endowed with the standard topology (see, e.g., \cite{Dem}); continuous linear functionals on $\Lambda_{c}^{t,s} (X)$ are called $(n-t,n-s)$-currents.

%
%Over each simply connected open subset $U_0 \subset X_0$ there exists a biholomorphic trivialization
%$\psi:p^{-1}(U_0) \rightarrow U_0 \times G$ of covering $p:X \rightarrow X_0$ which is a morphism of fibre bundles with fibres $G$. We fix some system of biholomorphic trivializations, and denote for a given subset $S \subset G$
%\begin{equation}
%\label{pi0}
%\Pi(U_0,S):=\psi^{-1}(U_0 \times S).
%\end{equation}
%(see Section \ref{charts} for details). 

\medskip

\noindent\textbf{Acknowledgement.~}We thank Professors T. Bloom, L. Lempert, T. Ohsawa, R. Shafikov and Y.-T. Siu for their interest to this work.

\tableofcontents

\section{Elements of function theory within $\mathcal O_{\mathfrak a}(X)$} 
\label{mainsect}

\SkipTocEntry\subsection{}
\label{constrsect}
Our approach is based on analogues of the Cartan theorems A and B for coherent-type sheaves on the fibrewise compactification $c_{\mathfrak a}X$ of covering $p:X \rightarrow X_0$. %This object, denoted by $c_{\mathfrak a}X$, is a fibre bundle over $X_0$ containing $p:X \rightarrow X_0$ as a dense subbundle.  Moreover, if $\mathfrak a$ is self-adjoint (i.e., closed with respect to complex conjugation) and $X_0$ is a Stein manifold, $c_{\mathfrak a}X$ can be naturally identified with the maximal ideal space of algebra $\mathcal O_{\mathfrak a}(X)$, cf.~Theorem \ref{corona}.  
We briefly describe its construction postponing details till Section \ref{compsect} (see also \cite{BK8}).

Let $M_{\mathfrak a}$ denote the maximal ideal space of algebra $\mathfrak a$, i.e., the space of non-zero continuous homomorphisms $\mathfrak a \rightarrow \mathbb C$ endowed with weak* topology (of $\mathfrak a^*$). 
The space $M_{\mathfrak a}$ is compact and
every element $f$ of $\mathfrak a$ determines a function $\hat{f} \in C(M_{\mathfrak a})$ by the formula
$$\hat{f}(\eta):=\eta(f), \quad \eta \in M_{\mathfrak a}.$$
Since algebra $\mathfrak a$ is uniform (i.e., $\|f^2\|=\|f\|^2$) and hence is semi-simple, the homomorphism $\hat\, :\mathfrak a \rightarrow C(M_{\mathfrak a})$ (called the {\em Gelfand transform}) is an isometric embedding (see, e.g., \cite{Gam}).
We have a continuous map $j=j_{\mathfrak a}:G \rightarrow M_{\mathfrak a}$ defined by 
%$$j(g)(f):=f(g), \quad \bigl(f \in \mathfrak a\bigr),$$
associating to each point in $G$ its point evaluation homomorphism in $M_{\mathfrak a}$. This map is an injection if and only if algebra $\mathfrak a$ separates points of $G$.

Let $\hat{G}_{\mathfrak a}$ denote the closure of $j(G)$ in $M_{\mathfrak a}$ (also a compact Hausdorff space).
If algebra $\mathfrak a$ is self-adjoint, then $\mathfrak a \cong C(M_{\mathfrak a})$ and hence $\hat{G}_{\mathfrak a}=M_{\mathfrak a}$.
In a standard way the action of group $G$ on itself by right multiplication determines the right action of $G$ on $M_{\mathfrak a}$, so that $\hat{G}_{\mathfrak a}$ is invariant with respect to this action.

\begin{definition}
\label{compdefin}
{\em The fibrewise compactification} $\bar{p}:c_{\mathfrak a}X \rightarrow X_0$ is defined to be the fibre bundle with fibre $\hat{G}_{\mathfrak a}$ associated to the regular covering $p:X\rightarrow X_0$ (regarded as a principal bundle with fibre $G$).
\end{definition}

There exists a continuous map 
\begin{equation}
\label{iotamap}
\iota=\iota_{\mathfrak a}:X \rightarrow c_{\mathfrak a}X
\end{equation}
induced by the equivariant map $j$. Clearly, $\iota(X)$ is dense in $c_{\mathfrak a}X$. If $\mathfrak a$ separates points of $G$, then $\iota$ is an injection. 

\begin{definition}\label{holomo}
A function $f \in C(c_{\mathfrak a}X)$ is called {\em holomorphic} if its pullback $\iota^*f$ is holomorphic on $X$. 
The algebra of functions holomorphic on $c_{\mathfrak a}X$ is denoted by $\mathcal O(c_{\mathfrak a}X)$.
\end{definition}

For a subalgebra $\mathfrak a\subset\ell_\infty(G)$ 
we have a monomorphism $\mathcal O_{\mathfrak a}(X) \hookrightarrow \mathcal O(c_{\mathfrak a}X)$  (see Proposition \ref{basicpropthm} below) which is an isomorphism if $\mathfrak a$ is self-adjoint (in this case we can work with algebra $\mathcal O(c_{\mathfrak a}X)$ instead of $\mathcal O_{\mathfrak a}(X)$). See \cite{BK8} for other properties of $c_{\mathfrak a}X$.

Analogously to Definition \ref{holomo} we define holomorphic functions on open subsets of $c_{\mathfrak a}X$ and thus obtain the structure sheaf $\mathcal O:=\mathcal O_{c_{\mathfrak a}X}$ of germs of holomorphic functions on $c_{\mathfrak a}X$. 
Now, a {\em coherent sheaf} $\mathcal A$ on $c_{\mathfrak a}X$ is a sheaf of modules over $\mathcal O$ such that every point in $c_{\mathfrak a}X$ has a neighbourhood $U$ over which, for any $N \geq 1$, there is a free resolution of $\mathcal A$ of length $N$, i.e., an exact sequence of sheaves of modules of the form
\begin{equation}
\label{coh0}
\mathcal O^{m_{N}}|_U \overset{\varphi_{N-1}}{\to} \dots \overset{\varphi_2}{\to} \mathcal O^{m_{2}}|_U \overset{\varphi_1}{\to} \mathcal O^{m_{1}}|_U \overset{\varphi_0}{\to} \mathcal A|_U \to 0
\end{equation}
(here $\varphi_i$, $0 \leq i \leq N-1$, are homomorphisms of sheaves of modules).
If $X=X_0$ and $p=\Id$, then  this definition gives the classical definition of a coherent sheaf on a complex manifold $X_0$.

\medskip

Let $X_0$ be a Stein manifold, $\mathcal A$ a coherent sheaf on $c_{\mathfrak a}X$.

\begin{theorem}[\cite{BK8}]
\label{thmA}

Each stalk $\phantom{}_{x}\mathcal A$ ($x \in c_{\mathfrak a}X$) is generated by  sections $\Gamma(c_{\mathfrak a}X,\mathcal A)$ as an $\phantom{}_{x}\mathcal O$-module (``Cartan-type Theorem A''). 
\end{theorem}

\begin{theorem}[\cite{BK8}]
\label{thmB}
Sheaf cohomology groups $H^i(c_{\mathfrak a}X,\mathcal A)=0$ for all $i \geq 1$ (``Cartan-type Theorem B'').
\end{theorem}

%In the next two sections, devoted to applications of Theorems \ref{thmA} and \ref{thmB}, we will use the following definitions.
%Suppose that algebra $\mathfrak a$ is self-adjoint. 

The collection of open subsets of $X$ of the form $V=\iota^{-1}(U)$, where $U \subset c_{\mathfrak a}X$ is open, determines a topology on $X$,
denoted by $\mathcal T_{\mathfrak a}$, Hausdorff if and only if $\mathfrak a$ separates points of $G$. (The base of $\mathcal T_{\mathfrak a}$ consists of interiors of sublevel sets of functions in $C_{\mathfrak a}(X)$.)
We define spaces of continuous and holomorphic $\mathfrak a$-functions on $V=\iota^{-1}(U)\in \mathcal T_{\mathfrak a}$ by
$$C_{\mathfrak a}(V):=\iota^* C(U), \quad \mathcal O_{\mathfrak a}(V):=\iota^*\mathcal O(U).$$ 
For subsets $V,W \in \mathcal T_{\mathfrak a}$ we denote $$\mathcal O_{\mathfrak a}(V,W):=\{f \in C(V,W): f^*h \in \mathcal O_{\mathfrak a}(V) \text{ for all } h \in \mathcal O_{\mathfrak a}(W)\}.$$

\medskip

\textit{In subsections \ref{sectionext}, \ref{sectdivisors} and \ref{sectbohr} we assume that algebra $\mathfrak a$ is self-adjoint.}

\SkipTocEntry\subsection{}
\label{sectionext}
In this subsection we describe our results on complex $\mathfrak a$-submanifolds, their tubular neighbourhoods and on related interpolation problems that can be derived by means of Theorem \ref{thmB}.
%Using Theorem \ref{thmB}, we obtain the result on extension within subalgebra $\mathcal O_{\mathfrak a}(X)$ (Theorem \ref{extapthm} below). To formulate it, 
To formulate the results we will need

%\medskip

%In the following definitions we
%assume that algebra $\mathfrak a$ is self-adjoint. 

\begin{definition}
\label{tfinedef}
An open cover $\mathcal V$ of $X$ is said to be of class ($\mathcal T_{\mathfrak a}$) if it is the pullback by $\iota$ of an open cover of $c_{\mathfrak a}X$ (e.g., $\mathcal V=p^{-1}(\mathcal V_0)$, where $\mathcal V_0$ is an open cover of $X_0$ is of class ($\mathcal T_{\mathfrak a}$)).
%and $\bar{p}(V) \Subset X_0$ for every $V \in \mathcal V$ (e.g., $\mathcal V=p^{-1}(\mathcal V_0)$, where $\mathcal V_0$ is a cover of $X_0$ by relatively compact open subsets is of class ($\mathcal T_{\mathfrak a}$)).
\end{definition}

It is easy to see that any open cover of $X$  by sets in $\mathcal T_{\mathfrak a}$ is a subcover of an open cover of $X$ of class ($\mathcal T_{\mathfrak a}$).

\begin{definition}
\label{manifolddef}
A closed subset $Z \subset X$ is called a complex $\mathfrak a$-submanifold of codimension $k \leq n:=\dim_{\mathbb C}X_0$ if there exists an open cover $\mathcal V$ of $X$ of class ($\mathcal T_{\mathfrak a}$) such that for each $V\in\mathcal V$ the closure of $p(V)$ is compact and contained in a coordinate chart on $X_0$ and either $V\cap Z=\emptyset$ or there are functions $h_1,\dots,h_k \in \mathcal O_{\mathfrak a}(V)$ that satisfy:
\begin{itemize}
\item[(1)] $Z \cap V=\{x \in V\, :\, h_1(x)=\dots=h_k(x)=0\}$;
\item[(2)] maximum of moduli of determinants of all $k\times k$ submatrices of the Jacobian matrix 
of the map $x\mapsto \bigl(h_1(x),\dots,h_k(x) \bigr)$ with respect to local coordinates on $V$ pulled back from a coordinate chart on $X_0$ containing the closure of $p(V)$
is uniformly bounded away from zero on $Z \cap V$.
\end{itemize}
\end{definition}

Some examples of complex $\mathfrak a$-submanifolds are given in subsection \ref{submex} below.

We have analogues of Cartan type theorems \ref{thmA} and \ref{thmB} on complex $\mathfrak a$-submanifolds, see~subsection \ref{submanifoldsect} and Theorems \ref{thmA_}, \ref{thmB_} below.

\begin{theorem}[Characterization of complex $\mathfrak a$-submanifolds]
\label{countthm}
Suppose that $X_0$ is a Stein manifold.
Then a closed subset $Z \subset X$ is a complex $\mathfrak a$-submanifold of codimension $k \leq n$ if and only if there exist an at most countable collection of functions $f_i \in \mathcal O_{\mathfrak a}(X)$, $i \in I$, and an open cover $\mathcal V$ of $X$ of class ($\mathcal T_{\mathfrak a}$) such that 

(\textit{i}) $Z=\{x \in X: f_i(x)=0 \text{ for all } i \in I\}$,

(\textit{ii}) for each $V\in\mathcal V$ the closure $p(V)$ is compact and contained in a coordinate chart on $X_0$, and either $V\cap Z=\emptyset$ or there are functions
$f_{i_1}, \dots f_{i_k}$ such that $Z \cap V=\{x \in V: f_{i_1}=\dots=f_{i_k}=0\}$ and the maximum of moduli of determinants of all $k\times k$ submatrices of the Jacobian matrix 
of the map
$x \mapsto \bigl(f_{i_1}(x),\dots,f_{i_k}(x) \bigr)$ with respect to local coordinates on $V$ pulled back from a coordinate chart on $X_0$ containing the closure of $p(V)$
is uniformly bounded away from zero on $Z \cap V$.
\end{theorem}

We prove Theorem \ref{countthm} in Section \ref{countthmsect}.

\begin{definition}
\label{holdef2}
A function $f \in \mathcal O(Z)$ on a complex $\mathfrak a$-submanifold $Z \subset X$ is called a holomorphic $\mathfrak a$-function if it admits an extension to a function in $C_{\mathfrak a}(X)$.
\end{definition}

The subalgebra of holomorphic $\mathfrak a$-functions on $Z$ is denoted by $\mathcal O_{\mathfrak a}(Z)$.
Alternatively, subalgebra $\mathcal O_{\mathfrak a}(Z)$ can be defined in terms of $\mathfrak a$-currents, see~subsection \ref{equivcur}.

\medskip

We have the following analogue of the classical tubular neighbourhood theorem:

\begin{theorem}
\label{tubularnbd}
Let $X_0$ be a Stein manifold, $Z \subset X$ be a complex $\mathfrak a$-submanifold. Then there exist an open in topology $\mathcal T_{\mathfrak a}$ neighbourhood $\Omega \subset X$ of $Z$ and a family of maps $h_t \in \mathcal O_{\mathfrak a}(\Omega,\Omega)$ continuously depending on $t \in [0,1]$ such that 
\[
h_t|_{Z}=\Id_{Z}\quad\text{ for all}\quad t\in [0,1],\qquad h_0=\Id_{\Omega}\quad\text{and}\quad h_1(\Omega)=Z.
\]
\end{theorem}

Theorem \ref{tubularnbd} is proven in Section \ref{countthmsect}.

Theorem \ref{tubularnbd} gives a linear extension operator $h_1^*:\mathcal O_{\mathfrak a}(Z)\rightarrow \mathcal O_{\mathfrak a}(\Omega)$, $f\mapsto h_1^*f$. 

\medskip

Using Theorem \ref{thmB} we prove the following interpolation result.

%In many cases such linear extension reduces problems for algebra $\mathcal O_{\mathfrak a}(Z)$ to similar problems for $\mathcal O_{\mathfrak a}(X)$. For instance, combining this method with Theorem \ref{thmB} we obtain the following.

\begin{theorem}
\label{extapthm}
Suppose $X_0$ is a Stein manifold, $Z \subset X$ is a complex $\mathfrak a$-submanifold, and $f \in \mathcal O_{\mathfrak a}(Z)$. Then there is $F \in \mathcal O_{\mathfrak a}(X)$ such that $F|_Z=f$.
\end{theorem}

We prove Theorem \ref{extapthm} in Section \ref{countthmsect}.

\begin{example}
Suppose that $Z_1$, $Z_2 \subset T:=\mathbb R^n+i\Omega \subset \mathbb C^n$ (where $\Omega \subset \mathbb R^n$ is convex) are non-intersecting smooth complex hypersurfaces that are periodic, possibly with different periods, with respect to the usual action of $\mathbb R^n$ on $T$ by translations. Suppose also that the Euclidean distrance $\dist(Z_1,Z_2)>0$. Let $f_1 \in \mathcal O(Z_1)$, $f_2 \in \mathcal O(Z_2)$ be holomorphic functions periodic with respect to these periods. The union $Z_1 \cup Z_2$ is a complex almost periodic submanifold of $T$ in the sense of Definition \ref{manifolddef} (cf.~Example \ref{holap}), and so  by Theorem \ref{extapthm} there is a holomorphic almost periodic function $F \in \mathcal O_{AP}(T)$ such that $F|_{Z_i}=f_i$, $i=1,2$. 
\end{example}

%\begin{remark}
%\label{extrem}

%It is not difficult to construct an example of a subset $Z$ satisfying (1) but not (2) in Definition \ref{manifolddef},
%such that $\mathcal O_{\mathfrak a}(X)|_{Z} \subsetneq C_{\mathfrak a}(X) \cap \mathcal O(Z)$.

\SkipTocEntry\subsection{} 
\label{sectdivisors}
This subsection describes our results on $\mathfrak a$-divisors.

\medskip

Let $Z \subset X$ be a complex submanifold.
Recall that a continuous line bundle $L$ on $Z$ is given by an open cover $\{U_\alpha\}$ of $Z$ and nowhere zero functions $d_{\alpha\beta} \in C(U_\alpha \cap U_\beta)$ (where $d_{\alpha\beta}:=1$ if $U_\alpha \cap U_\beta=\varnothing$) satisfying the 1-cocycle conditions:
\begin{equation}
\label{linbun1}
\forall \alpha, \beta \quad d_{\alpha\beta}=d_{\beta\alpha}^{-1} \text{ on } U_\alpha \cap U_\beta,
\end{equation}
\begin{equation}
\label{linbun2}
\forall \alpha,\beta, \gamma \quad d_{\alpha\beta}d_{\beta\gamma}d_{\gamma\alpha}=1 \text{ on } U_\alpha \cap U_\beta \cap U_\gamma \neq \varnothing .
\end{equation}
If all $d_{\alpha\beta} \in \mathcal O(U_\alpha \cap U_\beta)$, then $L$ is called a holomorphic line bundle.

In a standard way one defines continuous and holomorphic line bundles morphisms (see, e.g.,~\cite{Hirz}). 
The categories of continuous and holomorphic line bundles on $Z$ are denoted by $\mathcal L^c(Z)$ and $\mathcal L(Z)$, respectively.

%\medskip

An {\em effective (Cartier) divisor} $E$ on $Z$ is given by an open cover $\{U_\alpha\}$ of $Z$ and not identically  zero on open subsets of $U_\alpha$ functions $f_\alpha \in \mathcal O(U_\alpha)$ such that 
\begin{equation}
\label{divformula}
\forall \alpha,\beta\quad f_\alpha=d_{\alpha\beta} f_\beta \ \text{ on } U_\alpha \cap U_\beta  \text{ for some } d_{\alpha\beta}\in\mathcal O(U_\alpha \cap U_\beta ,\mathbb C\setminus\{0\}).
\end{equation}
Clearly, holomorphic 1-cocycle $\{d_{\alpha\beta}\}$ determines a holomorphic line bundle denoted by $L_E$.

The collection of effective divisors on $Z$ is denoted by $\Div(Z)$.

%\begin{definition}
%\label{diveqdef}
Divisors $E=\{(U_\alpha,f_\alpha)\}$ and $E'=\{(V_\beta,g_\beta)\}$ in $\Div(Z)$ are said to be {\em equivalent} (in $\Div(Z)$) if there exists a refinement $\{W_\gamma\}$ of both covers $\{U_\alpha\}$ and $\{V_{\beta}\}$ and nowhere zero functions $p_\gamma\in\mathcal O(W_\gamma)$ such that
\begin{equation}
\label{equivdivrel}
f_\alpha|_{W_\gamma}=p_\gamma\cdot g_{\beta}|_{W_\gamma} \quad \text{ for } W_\gamma\subset U_\alpha\cap V_\beta.
\end{equation}
%\end{definition}
If divisors $E$, $E'$ are equivalent, then their line bundles $L_E$, $L_{E'}$ are isomorphic.

\medskip

Now, let $Z\subset X$ be either a complex $\mathfrak a$-submanifold or $X$ itself. 

\begin{definition}
\label{tfinedef2}
An open cover of $Z$ is said to be of class ($\mathcal T_{\mathfrak a}$) if it is the pullback by $\iota$ of an open cover of the closure of $\iota(Z)$ in $c_{\mathfrak a}X$ (cf. Definition \ref{tfinedef}).
\end{definition}

\begin{definition}
\label{lineap}
A continuous line bundle $L$ on $Z$ is called an {\em $\mathfrak a$-bundle} if in its definition (see (\ref{linbun1}), (\ref{linbun2}))

(1) $\{U_\alpha\}$ is of class ($\mathcal T_{\mathfrak a}$), 

(2) $\forall \alpha, \beta\ $  $d_{\alpha\beta}\in C_{\mathfrak a}(U_\alpha \cap U_\beta)$. 

\noindent If all $d_{\alpha\beta}\in \mathcal O_{\mathfrak a}(U_\alpha \cap U_\beta)$, then $L$ is called a holomorphic line $\mathfrak a$-bundle.
\end{definition}

The categories of continuous and holomorphic line $\mathfrak a$-bundles on $Z$ are denoted by $\mathcal L_{\mathfrak a}^c(Z)$ and $\mathcal L_{\mathfrak a}(Z)$, respectively.

\begin{definition}
\label{defdivisor0}
A divisor $E \in \Div(Z)$ is called an {\em effective $\mathfrak a$-divisor} if in its definition (see \eqref{divformula}) 

(1) $\{U_\alpha\}$ is of class ($\mathcal T_{\mathfrak a}$), 

(2) $\forall\alpha\ $ $f_\alpha \in \mathcal O_{\mathfrak a}(U_\alpha)$,

(3) $\forall\alpha,\beta\ $ $f_\alpha=d_{\alpha\beta}f_\beta$ on $U_\alpha \cap U_\beta$  for some $d_{\alpha\beta} \in \mathcal O_{\mathfrak a}(U_\alpha \cap U_\beta)$. %whose modulus is uniformly bounded away from zero on every open subset $V \Subset U_\alpha \cap U_\beta$, $V \in \mathcal T_{\mathfrak a}$,
 %the property that the closure of $\iota(V)$ is contained in an open subset $W \subset c_{\mathfrak a}X$ such that $U_\alpha \cap U_\beta=\iota^{-1}(W)$, for all $\alpha,\beta$.

\end{definition}
The collection of $\mathfrak a$-divisors is denoted by $\Div_{\mathfrak a}(Z)$. 

By the definition the line bundle $L_E$ of an $\mathfrak a$-divisor $E$ is a holomorphic line $\mathfrak a$-bundle.

\begin{definition}
\label{diveqdef}
$\mathfrak a$-divisors $E=\{(U_\alpha,f_\alpha)\}$ and $E'=\{(V_\beta,g_\beta)\}$ are said to be \textit{$\mathfrak a$-equivalent} if in the above definition of equivalence in $\Div(Z)$ (see~(\ref{equivdivrel}))

(1) $\{W_\gamma\}$ is of class ($\mathcal T_{\mathfrak a}$), 

(2) $\forall \gamma $ $p_\gamma$, $p_\gamma^{-1} \in \mathcal O_{\mathfrak a}(W_\gamma)$.
\end{definition}

If divisors $E$, $E'$ are $\mathfrak a$-equivalent, then their line bundles $L_E$, $L_{E'}$ are isomorphic in $\mathcal L_{\mathfrak a}(Z)$.

\medskip

For some algebras $\mathfrak a$ (e.g., algebras of holomorphic almost periodic functions, see~Example \ref{holap} and subsection \ref{holap3})  $\mathfrak a$-divisors can be equivalently defined in terms of their currents of integration, see~subsection \ref{zeroapsect}.

The basic example of an $\mathfrak a$-divisor is divisor $E_f$ of a function  $f \in \mathcal O_{\mathfrak a}(Z)$, called an {\em $\mathfrak a$-principal divisor}.  
There are, however, divisors in $\Div_{\mathfrak a}(Z)$ that are not $\mathfrak a$-principal (see~subsection \ref{submex}(4)); because
% explain what c_aX
the \v{C}ech cohomology group $H^2(\overline{\iota(Z)},\mathbb Z)$, where $\overline{\iota(Z)}$ is the closure of $\iota(Z)$ in $c_{\mathfrak a}X$, whose elements measure  deviations of $\mathfrak a$-divisors on $X$ from being $\mathfrak a$-principal is in general non-trivial (see the proof of Theorem \ref{chernthm} for details).
% (e.g.~in the setting of Example \ref{holap} the fibrewise compactification $c_{AP}T$ is an inverse limit of smooth principal fibre bundles with compact Abelian groups $(S^1)^m \times \oplus_{k=1}^l \mathbb Z/n_k \mathbb Z$ as their fibres, cf.~Example 4.2 in \cite{BK8}).
%bohrex%
This naturally leads to the following question in a special case first considered in \cite{FRR}:

{\em Suppose that $X_0$ is Stein and $H^2(Z,\mathbb Z)=0$. Does there exist a class of functions $\mathfrak C_{\mathfrak a} \subset \mathcal O_{\ell^\infty}(Z)$, $\ell_\infty:=\ell_\infty(G)$, such that for each function from $\mathfrak C_{\mathfrak a}$ its divisor is equivalent to a divisor in $\Div_{\mathfrak a}(Z)$, and conversely, every divisor in $\Div_{\mathfrak a}(Z)$ is equivalent to a principal divisor determined by a function in $\mathfrak C_{\mathfrak a}$? }

If $Z=X=\{z \in \mathbb C: a<\Imag(z)<b\}$ and $\mathfrak a=AP(\mathbb Z)$ (see~Example \ref{holap}), then it was established in \cite{FRR} that the class $$\mathfrak C_{AP}=\{f \in \mathcal O(Z): |f| \in C_{AP}(Z)\}$$
satisfies this property.
(The proof in \cite{FRR} uses some properties of almost periodic currents.)  
%(see \cite{Fav2} for the extension of this result to the almost periodic functions of several variables). 
In Proposition \ref{zeroapcor} we extend this result to $\mathfrak a$-divisors defined on certain one-dimensional complex manifolds $Z$. % (e.g.~if $Z=X$ is a covering of a non-compact Riemann surface) (Proposition \ref{zeroapcor} below).
 In turn, 
the results in \cite{Fav2} show that for the algebra of holomorphic almost periodic functions $\mathcal O_{AP}(Z)$ on a tube domain $Z \subset \mathbb C^n$ with $n>1$ (see~Example \ref{holap}) the functions in $\mathfrak C_{AP}$ determine only a proper subclass of almost periodic (i.e.,~$AP$-) divisors and provide a complete description of this subclass.
Using our sheaf-theoretic approach we extend this result in Theorem \ref{zeroapthm} and Proposition \ref{zeroapcor} below. 

To formulate the results we require

%
%Let $Z \subset X$ be a complex $\mathfrak a$-submanifold.
%We say that divisors $E$, $E' \in \Div(Z)$ are $C_{\mathfrak a}$-equivalent if in Definition \ref{diveqdef} the refinement $\{W_\gamma\}$ is of class ($\mathcal T_{\mathfrak a}$) and functions $p_\alpha \in \mathcal O(W_\gamma)$ are such that $|p_\gamma|,|p_\gamma|^{-1} \in C_{\mathfrak a}(W_\gamma)$.

\begin{definition}
\label{semidef}
A line bundle $L \in \mathcal L_{\mathfrak a}(Z)$ is called $\mathfrak a$-{\em semi-trivial} if there exists an isomorphism $\psi$ in category $\mathcal L_{\ell_\infty}(Z)$ of $L$ onto the trivial line bundle $L_0 \in \mathcal L_{\ell_\infty}(Z)$ such that $|\psi|^2:=\psi \otimes \bar{\psi}$ is an isomorphism in category $\mathcal L_{\mathfrak a}^c(Z)$ of $L \otimes \bar{L}$ onto $L_0 \otimes \bar{L}_0$.
\end{definition}
(Here $\bar{L}$ is the bundle defined by complex conjugation of fibres of $L$.)

This definition is related to the question raised above via the following result (where we do not assume that $X_0$ is Stein).

The proofs of the Theorems \ref{divrelprop}, \ref{zeroapthm}, \ref{chernthm} and Proposition \ref{zeroapcor} below are given in Section \ref{needs}.

\begin{theorem}
\label{divrelprop}
If the line bundle $L_E$ of an $\mathfrak a$-divisor $E$ is $\mathfrak a$-semi-trivial, then $E$ is $\ell_\infty$-equivalent to divisor $E_f \in \Div(Z)$ of a function $f \in \mathcal O(Z)$ with $|f| \in C_{\mathfrak a}(Z)$. 

Suppose that $\mathfrak a$ is such that $\hat{G}_{\mathfrak a}$ is a compact topological group and $j(G)\subset\hat{G}_{\mathfrak a}$ is a dense subgroup. Then for $Z=X$ the converse holds:
%For $\mathfrak a=AP$ (the algebra of von Neumann almost periodic functions on the deck transformation group $G$, see~Example \ref{exm}(2) below) and $Z=X$ the converse holds:

If $E\in {\rm Div}_{\mathfrak a}(X)$ is $\ell_\infty$-equivalent to $E_f \in \Div(X)$ with  $|f|\in C_{\mathfrak a}(X)$, then $L_E$ is $\mathfrak a$-semi-trivial.
\end{theorem}

The second statement of the theorem is valid, e.g., for $\mathfrak a=AP(G)$, the algebra of von Neumann almost periodic functions on the deck transformation group $G$, see~Example \ref{exm}(2) below.  In this case $\hat{G}_{\mathfrak a}:=bG$, the Bohr compactification of $G$.

Now, we characterize the class of $\mathfrak a$-semi-trivial holomorphic line $\mathfrak a$-bundles.

\begin{theorem}
\label{zeroapthm}

Suppose $X_0$ is a Stein manifold. A line bundle $L \in \mathcal L_{\mathfrak a}(Z)$ is
$\mathfrak a$-semi-trivial
if and only if

(1) $L$ is isomorphic in category $\mathcal L_{\mathfrak a}(Z)$ to a discrete line $\mathfrak a$-bundle $L'$ (i.e., a bundle determined by a locally constant cocycle) and

(2) $L'$ is trivial in the category of discrete line bundles on $Z$.

\end{theorem}

The argument in the proof of  Theorem \ref{zeroapthm} implies that if the line bundle $L_E$ of an $\mathfrak a$-divisor $E$ satisfies condition (1) only, then the current of integration associated with $E$ (see subsection \ref{zeroapsect}) coincides with $\frac{i}{\pi}\partial\bar{\partial} \log h$, where $h$ is a nonnegative continuous plurisubharmonic $\mathfrak a$-function on $Z$.

%As a consequence of Theorem \ref{zeroapthm}, we obtain the following generalization of the result in \cite{FRR}.

\begin{proposition}
\label{zeroapcor}
Suppose $X_0$ is a Stein manifold and $Z \subset X$ is one-dimensional. Then for a line bundle $L \in \mathcal L_{\mathfrak a}(Z)$ condition (1) of Theorem \ref{zeroapthm} is satisfied.
If also $H^1(Z,\mathbb C)=0$, then condition (2) is satisfied as well.
\end{proposition}

In particular, conditions (1) and (2) of Theorem \ref{zeroapthm} are satisfied if $Z=X$ is the universal covering of a non-compact Riemann surface
$X_0$ and $\mathfrak a\subset\ell_\infty\bigr(\pi_1(X_0)\bigr)$ is a self-adjoint closed subalgebra.

%The assertion of Theorem \ref{zeroapthm} can be refined for holomorphic almost periodic functions on coverings of complex manifolds, cf.~Section \ref{holap3}. 

\medskip

The second Cousin problem for algebra $\mathcal O_{\mathfrak a}(X)$ asks about conditions for a divisor in $\Div_{\mathfrak a}(X)$ to be $\mathfrak a$-principal. Our next result provides some sufficient conditions for its solvability.

\begin{theorem}
\label{chernthm}
Let $X_0$ be a Stein manifold and $E \in \Div_{\mathfrak a}(X)$.
If $X_0$ is homotopy equivalent to an open subset $Y_0 \subset X_0$ such that the
restriction of $E$ to $Y:=p^{-1}(Y_0)$ is $\mathfrak a$-equivalent to an $\mathfrak a$-principal divisor, then $E$ is $\mathfrak a$-equivalent to an $\mathfrak a$-principal divisor as well. 

In particular, this is true if $\mathfrak a$ is such that $\hat{G}_{\mathfrak a}$ is a compact topological group and $j(G)\subset\hat{G}_{\mathfrak a}$ is a dense subgroup, and $\supp(E) \cap Y=\varnothing$; here $\supp(E)$ is the union of zero loci of holomorphic functions determining $E$. 
\end{theorem}

%If the covering dimension of $\hat{G}_{\mathfrak a}$ is zero (e.g., $\mathfrak a=\ell_\infty$ or $AP_{\mathbb Q}(\mathbb Z^n)$, cf.~Example  \ref{exm}(3)), then Theorem \ref{chernthm} follows trivially (cf.~\cite{Br8}).

For the algebra of Bohr's holomorphic almost periodic functions with $X$ and $Y$ being tube domains and $\mathfrak a=AP(\mathbb Z^n)$ (see~Example \ref{holap}) this theorem is due to \cite{FRR} ($n=1$) and \cite{Fav} ($n \geq 1$). The proof in \cite{FRR} uses Arakelyan's theorem and gives an explicit construction of a holomorphic almost periodic function that determines the principal divisor. Similarly to \cite{Fav} our proof of Theorem \ref{chernthm} is sheaf-theoretic.

\SkipTocEntry\subsection{}
\label{sectbohr}
A classical result by H.~Bohr states that if a holomorphic function $f$ on a complex strip $T:=\{z \in \mathbb C: \Imag(z) \in (a,b)\}$, bounded on closed substrips, is continuous almost periodic on a horizontal line $\mathbb R+ic$, $c \in (a,b)$, then $f$ is holomorphic almost periodic on $T$. In this subsection we extend this result to algebras $\mathcal O_{\mathfrak a}(X)$.

The regular covering $p:X \rightarrow X_0$ is a principal fibre bundle over $X_0$ with structure group $G$, hence, for a cover $\{U_{0,\gamma}\}$ of $X_0$ by open simply connected subsets there exists a locally constant cocycle $\{c_{\delta\gamma}:U_{0,\gamma} \cap U_{0,\delta} \rightarrow G\}$ such that the covering $p:X \rightarrow X_0$ is obtained from the disjoint union 
$\sqcup_{\gamma}U_{0,\gamma} \times G$ by the identification
\begin{equation}
\label{extrem2}
U_{0,\delta} \times G \ni (x,g) \sim (x,g\cdot c_{\delta\gamma}(x)) \in U_{0,\gamma} \times G \quad \text{ for all } x\in U_{0,\gamma} \cap U_{0,\delta},
\end{equation}
where projection $p$ is induced by the projections $U_{0,\gamma} \times G\rightarrow U_{0,\gamma}$ (see, e.g., \cite{Hirz}). 
Local inverses $\psi_\gamma:p^{-1}(U_{0,\gamma}) \rightarrow U_{0,\gamma} \times G$ to the identification map form a system of biholomorphic trivializations of the covering. 
For a given subset $S \subset G$ denote
\begin{equation*}
\Pi_\gamma(U_{0,\gamma},S):=\psi_\gamma^{-1}(U_{0,\gamma} \times S).
\end{equation*}

Now, let $U_0 \subset X_0$ be an open simply connected set contained in some $U_{0,\gamma_*}$, $Z_0 \subset U_0$ be a uniqueness set for holomorphic functions in $\mathcal O(X_0)$, and subsets $L \subset K \subset G$ be such that
the closure of $j(L)$ in $\hat{G}_{\mathfrak a}$ is contained in the interior of the closure of $j(K)$ in $\hat{G}_{\mathfrak a}$ (see~subsection \ref{constrsect} for notation) and
$\cup_{i=1}^m\, L \cdot g_i=G$ for some $g_1,\dots,g_m \in G$. 

Consider $Z\subset X$ such that
$$
p^{-1}(Z_0) \cap \Pi_{\gamma_*}(U_0,K)\subset Z.
$$
(In particular, we can take $L=K:=G$ and $Z:=p^{-1}(Z_0)$.)

We define $C_{\mathfrak a}(Z):=C_{\mathfrak a}(X)|_Z$.

\begin{theorem}
\label{hypthm}
If $f \in \mathcal O_{\ell_\infty(G)}(X)$ and $f|_Z \in C_{\mathfrak a}(Z)$, then $f \in \mathcal O_{\mathfrak a}(X)$.
\end{theorem}

We prove Theorem \ref{hypthm} in Section \ref{hypsect}.

\begin{remark}\label{rem2.22}
(1) In the settings of the classical Bohr theorem the choice of the objects in Theorem \ref{hypthm} can be specified:
\begin{proposition}\label{prop2.23}
Suppose that $\mathfrak a$ is such that $\hat{G}_{\mathfrak a}$ is a compact topological group and $j(G)\subset\hat{G}_{\mathfrak a}$ is a dense subgroup. 
Given $K\subset G$ the following conditions are equivalent:
\begin{itemize}
\item[(a)]
There exist $g_1,\dots, g_m\in G$ such that $\cup_{i=1}^m\, K\cdot g_i=G$;
\item[(b)]
The closure of $j(K)$ in  $\hat{G}_{\mathfrak a}$ has a nonempty interior;
\item[(c)]
There exists a subset $L\subset K$ satisfying conditions of Theorem \ref{hypthm}.
\end{itemize}
\end{proposition}

The proof of Proposition \ref{prop2.23} is given in Section \ref{hypsect}.

Thus, for such algebras $\mathfrak a$ one can take as the set $K$ in Theorem \ref{hypthm}, e.g., any nonempty subset of the form
$\{g\in G\, :\, |f(g)|<1,\, f\in \mathfrak a\}$.
For instance, in Bohr's result the line $\mathbb R+ic$ can be replaced with a set $S+K$, where $S\Subset T$ is an infinite set (hence, it is a uniqueness set for $\mathcal O(T)$)
and $K:=\{n\in\mathbb Z\, :\, |E(n)|<1\}\ne\emptyset$, where $E$ is a univariate exponential polynomial of form \eqref{expoly}.

(2) As an example of the uniqueness set $Z_0$ in Theorem \ref{hypthm} we can take any real hypersurface in $X_0$ or, more generally,
a set of the form $\{x \in X_0: \rho_1(x)=\dots=\rho_d(x)=0\}$, where
$\rho_1,\dots,\rho_d$ are real-valued differentiable functions on $X_0$
and $\partial \rho_1(x_0) \wedge \dots \wedge \partial \rho_d(x_0) \neq 0$ for some $x_0 \in Z_0$ (see, e.g., \cite{Bog}).
\end{remark}

\SkipTocEntry\subsection{}
\label{nonselfadj}
\label{approx}
\label{leray}

\textit{In this subsection we do not assume that algebra $\mathfrak a$ is self-adjoint.}

\medskip

%\begin{remark}
1. The following discussion suggests an alternative approach to study of $\mathcal O_{\mathfrak a}(X)$.
Namely, we have an equivalent presentation of functions in $\mathcal O_{\mathfrak a}(X)$ as holomorphic sections of a holomorphic Banach vector bundle $\tilde{p}:C_{\mathfrak a}X_0 \rightarrow X_0$
 associated to the principal fibre bundle $p:X \rightarrow X_0$ and having fibre $\mathfrak a$ defined as follows. 
 
The regular covering $p:X \rightarrow X_0$ is a principal fibre bundle with structure group $G$ (see~subsection \ref{sectbohr}).  
%The regular covering $p:X \rightarrow X_0$ is a principal fibre bundle with structure group $G$, hence, there exists an open cover $\{U_{0,\gamma}\}$ of $X_0$ and a locally constant cocycle $\{c_{\delta\gamma}:U_{0,\gamma} \cap U_{0,\delta} \rightarrow G\}$ so that the covering $p:X \rightarrow X_0$ can be obtained from the disjoint union 
%$\sqcup_{\gamma}U_{0,\gamma} \times G$ by the identification
%\begin{equation}
%\label{extrem2}
%U_{0,\delta} \times G \ni (x,g) \sim (x,g\cdot c_{\delta\gamma}(x)) \in U_{0,\gamma} \times G \quad \text{ for all } x\in U_{0,\gamma} \cap U_{0,\delta},
%\end{equation}
%where projection $p$ is induced by the projections $U_{0,\gamma} \times G\rightarrow U_{0,\gamma}$ (see, e.g., \cite{Hirz}). 
Then $\tilde p:C_{\mathfrak a}X_0\rightarrow X_0$ is a holomorphic Banach vector bundle associated to $p:X \rightarrow X_0$ and having fibre $\mathfrak a$
obtained from the disjoint union $\sqcup_{\gamma}U_{0,\gamma} \times \mathfrak a$ by the identification 
\begin{equation}
\label{extrem3}
U_{0,\delta} \times \mathfrak a \ni (x,f) \sim (x,R_{c_{\delta\gamma}(x)}(f)) \in U_{0,\gamma} \times \mathfrak a \quad \text{ for all }x\in U_{0,\gamma} \cap U_{0,\delta}.
\end{equation}
The projection $\tilde p$ is induced by projections $U_{0,\gamma}\times \mathfrak a \rightarrow U_{0,\gamma}$.

Let $\mathcal O(C_{\mathfrak a}X_0)$ be the space of holomorphic sections of $C_{\mathfrak a}X_0$. This is a Fr\'{e}chet algebra with respect to the usual pointwise operations and the topology of uniform convergence on compact subsets of $X_0$.

\begin{proposition}
\label{iso1prop}
$\mathcal O_{\mathfrak a}(X) \cong \mathcal O(C_{\mathfrak a}X_0)$.
\end{proposition}

We give proof of Proposution \ref{iso1prop} in Section \ref{isosect}.

Using Proposition \ref{iso1prop} we obtain the following result on extension within the class of holomorphic $\mathfrak a$-functions.  

\begin{proposition}
\label{simpleext}
Let $M_0$ be a closed complex submanifold of a Stein manifold $X_0$, $M:=p^{-1}(M_0)\subset X$, 
$D_0 \Subset X_0$ is Levi strictly pseudoconvex (see, e.g., \cite{GR}), $D:=p^{-1}(D_0)$, and $f \in \mathcal O_{\mathfrak a}(M \cap D)$ is bounded.
Then there exists a bounded function $F\in \mathcal O_{\mathfrak a}(D)$ such that $F|_{M \cap D}=f|_{M \cap D}$.
\end{proposition}

Indeed, subalgebra $\mathcal O_{\mathfrak a}(M)$ is isomorphic to the algebra $\mathcal O(C_{\mathfrak a}X)|_{M_0}$ of holomorphic sections of bundle $C_{\mathfrak a}X$ over $M_0$. 
Since $X_0$ is Stein,
there exist holomorphic Banach vector bundles $p_1:E_1 \rightarrow X_0$ and $p_2:E_2 \rightarrow X_0$ with fibres $B_1$ and $B_2$, respectively, such that $E_2=E_1 \oplus C_{\mathfrak a}X_0$ (the Whitney sum) and $E_2$ is holomorphically trivial, i.e., $E_2 \cong X_0\times B_2$ (see, e.g.,~\cite{Obz}). 
Thus, any holomorphic section of $E_2$ can be naturally identified with a $B_2$-valued holomorphic function on $X_0$. 
By $q:E_2 \rightarrow C_{\mathfrak a}X_0$ and $i:C_{\mathfrak a}X_0 \rightarrow E_2$ we denote the corresponding quotient and embedding homomorphisms of the bundles so that 
$q \circ i=\Id$. (Similar identifications hold for the bundle $C_{\mathfrak a}D_0$.)
Given a function $f \in \mathcal O(C_{\mathfrak a}X_0)|_{M_0}$ consider its image $\tilde{f}:=i(f)$, a $B_2$-valued holomorphic function on $M_0$, and apply to it the integral representation formula from \cite{HL} asserting the existence of a bounded function $\tilde{F} \in \mathcal O(D_0,B_2)$ such that $\tilde{F}|_{M_0 \cap D_0}=\tilde{f}|_{M_0 \cap D_0}$. Finally, we define $F:=q(\tilde{F})$. 

In fact, this method allows to obtain similar extension results for holomorphic functions on $X$ whose restrictions to each fibre belong to some Banach space, and are possibly unbounded, see \cite{Br4}.

In view of Proposition \ref{simpleext} it is natural to ask to what extent Theorems \ref{thmA}, \ref{thmB} and  \ref{extapthm} depend on the assumption that subalgebra $\mathfrak a$ is self-adjoint.
%\end{remark}

\medskip

2. Next, we show that the classical Leray integral representation formula can be extended to work within subalgebra $\mathcal O_{\mathfrak a}(X)$.

For a given $z \in X_0$ by $\mathfrak a_{z}$ we denote the subalgebra of functions $h:p^{-1}(z) \rightarrow \mathbb C$ such that for all $x \in p^{-1}(z)$ functions 
$G \ni g \mapsto h(g \cdot x)$ are in $\mathfrak a$,  endowed with $\sup$-norm. Clearly, $\mathfrak a_{z}$ is isometrically isomorphic to $\mathfrak a$.

Let 
%\begin{equation}
%\label{assumptions}
$$D_0 \Subset X_0 \text{ be a subdomain and\ } D:=p^{-1}(D_0). $$
%\end{equation}
We denote
$
\mathcal A_{\mathfrak a}(D):=C_{\mathfrak a}(\bar{D}) \cap \mathcal O_{\mathfrak a}(D)
$
(see Introduction for definitions). This is a Banach space with respect to $\sup$-norm.

\begin{theorem}
\label{interp}
Let $X_0$ be a Stein manifold and $D_0$, $D$ be as above. There is a family of bounded linear operators $L_z:\mathfrak a_z \rightarrow \mathcal A_{\mathfrak a}(D)$, $z \in D_0$, holomorphic in $z$ and such that

(1) $L_z(h)(x)=h(x)$ for all $h \in \mathfrak a_z$, 
$x \in p^{-1}(z)$;

(2) $\sup_{z \in D_0}\|L_z\|<\infty$.
\end{theorem}

We prove Theorem \ref{interp} in Section \ref{isosect}.

Now, let us recall the classical Leray integral representation formula. 
For $\xi,\eta \in \mathbb C^n$ we define $\langle \eta,\xi \rangle:=\sum_{k=1}^n \eta_j \xi_j$ and 
$$\omega(\xi):=d\xi_1 \wedge \dots \wedge d\xi_n, \quad
\omega'(\eta):=\sum_{k=1}^n (-1)^{k-1} \eta_k d\eta_1 \wedge \dots \wedge d\eta_{k-1} \wedge d\eta_{k+1} \wedge \dots \wedge d\eta_n.$$

For a domain $D_0\Subset\mathbb C^n$ we set $Q:=D_0 \times \mathbb C^n$. Fix $z\in D_0$ and define a hypersurface $P_z \subset Q$ by
\begin{equation*}
P_z:=\{(\eta,\xi) \in Q: \langle \eta,\xi-z\rangle=0\}.
\end{equation*}
Let $h_z$ be a $2n-1$-dimensional cycle in $Q \setminus P_z$ whose projection to $D_0$ is homologous to $\partial D_0$. 

\medskip

\noindent \textbf{Leray integral representation formula} (see, e.g., \cite{HL}). For any function  $f \in  \mathcal O(D_0)$,
\begin{equation}
\label{leray1}
f(z)=\frac{(n-1)!}{(2\pi i)^n} \int_{h_z} f(\xi) \frac{\omega'(\eta) \wedge \omega(\xi)}{\langle \eta,\xi-z\rangle^n}.
\end{equation}

Interpreting $ z \mapsto L_z\bigl(f|_{p^{-1}(z)}\bigr)$, $z\in D_0$, $f\in \mathcal O_{\mathfrak a}(D)$, in Theorem \ref{interp} as an $\mathcal A_{\mathfrak a}(D)$-valued holomorphic function on $D_0$ and using the fact that representation (\ref{leray1}) is valid for Banach-valued holomorphic functions (because the integral kernel in this formula is continuous and bounded on $h_z$) we obtain:

\begin{theorem}[Leray type integral representation formula]
\label{intreprthm}
Let $X_0\subset \mathbb C^n$ be a Stein domain and $D_0\Subset X_0$ be a subdomain. Then for any function $f \in \mathcal O_{\mathfrak a}(D)$,
\begin{equation}
\label{leray2}
f(x)=\frac{(n-1)!}{(2\pi i)^n} \int_{h_z} L_\xi\bigl(f|_{p^{-1}(\xi)}\bigr)(x)\frac{\omega'(\eta) \wedge \omega(\xi)}{\langle \eta,\xi-z\rangle^n}, \quad \text{ for all }\quad x \in p^{-1}(z).
\end{equation}
\end{theorem}

A similar formula for functions in $\mathcal O_{\ell_\infty}(D)$ was first established in \cite{Br4}.

\medskip

3. Similarly to \cite{Br3} we obtain the following Hartogs-type theorem.

%
%Let $\Lambda_{c}^{t,s} (X)$ denote the space of smooth $(t,s)$-forms with compact supports in $X$, endowed with the standard topology (see, e.g., \cite{Dem}); continuous linear functionals on $\Lambda_{c}^{t,s} (X)$ are called $(n-t,n-s)$-currents.

\begin{theorem}
\label{hartogsthm}
Suppose $n:={\rm dim}_{\mathbb C}X_0\ge 2$.  Let
$D_0 \Subset X_0$ be a subdomain with a connected piecewise $C^1$ boundary $\partial D_0$ contained in a Stein open submanifold of $X_0$ and $D:=p^{-1}(D_0)$.
Assume that $f \in C_{\mathfrak a}(\partial D)$ satisfies the tangential Cauchy-Riemann equations on $\partial D$, i.e., for any smooth $(n,n-2)$-form $\omega$ on $X$ having compact support
\begin{equation*}
\int_{\partial D}f\,\bar{\partial}\omega=0.
\end{equation*}
 Then there exists a function $F \in \mathcal O_{\mathfrak a}(D) \cap C(\bar{D})$ such that $F|_{\partial D}=f$.
\end{theorem}

The proof of Theorem \ref{hartogsthm} is given in Section \ref{isosect}.

In particular, Theorem \ref{hartogsthm} implies that if $n\ge 2$, then each continuous almost periodic function on the boundary $\partial T=\mathbb R^n+i\partial\Omega$ of a tube domain $T:=\mathbb R^n+i\Omega\subset\mathbb C^n$, where $\Omega\Subset\mathbb R^n$ is a domain with piecewise-smooth boundary $\partial\Omega$, satisfying the tangential Cauchy-Riemann equations on $\partial T$, admits a continuous extension to a holomorphic almost periodic function
in ${\mathcal O}_{AP}(T) \cap C(\bar{T})$.

\medskip

%\subsection{}
4. Now, we extend Bohr's approximation theorem for holomorphic almost periodic functions (see~Introduction) to an arbitrary subalgebra $\mathcal O_{\mathfrak a}(X)$.

Let $\mathfrak a_\iota$ ($\iota \in I$) be a collection of closed subspaces of $\mathfrak a$ such that

(1) $\mathfrak a_\iota$ are invariant with respect to the action of $G$ on $\mathfrak a$ by right translates (i.e., if $f \in \mathfrak a_\iota$, then $R_g(f) \in \mathfrak a_\iota$ for all $g \in G$),

(2) the family $\{\mathfrak a_\iota: \iota \in I\}$ forms a direct system ordered by inclusion, and

(3) the linear space $\mathfrak a_0:=\bigcup_{\iota \in I} \mathfrak a_\iota$ is dense in $\mathfrak a$.

The model examples of subspaces $\mathfrak a_{\iota}$ are given in subsection \ref{approxex} below.

\medskip

Let $\mathcal O_\iota(X)$ be the space of holomorphic functions $f \in \mathcal O_{\mathfrak a}(X)$ such that for every $x_0 \in X_0$ functions $$g \mapsto f(g \cdot x), \quad g \in G, \quad x \in F_{x_0},$$ 
belong to $\mathfrak a_\iota$.
Let $\mathcal O_0(X)$ be $\mathbb C$-linear hull of spaces $\mathcal O_\iota(X)$ with $\iota$ varying over $I$.

\begin{theorem}
\label{approxthm}
If $X_0$ is a Stein manifold,
then $\mathcal O_0(X)$ is dense in $\mathcal O_{\mathfrak a}(X)$.
\end{theorem}

We prove Theorem \ref{approxthm} in Section \ref{isosect}.

If $\mathfrak a=AP(G)$ (see~subsections \ref{exm}(2) and \ref{holap3}), then this theorem may be viewed as a holomorphic analogue of the Peter-Weyl approximation theorem.

%Note that, together with the example of Section \ref{usualtrigex}(1), this theorem gives another proof of Theorem \ref{equivapthm}.

\section{Examples}
\label{examples}

\subsection{Examples of subalgebras $\mathfrak a$}
\label{exm}
In addition to $\ell_\infty(G)$, $c(G)$ and $AP(\mathbb Z^n)$ we list the following examples of self-adjoint subalgebras of $\ell_\infty(G)$ separating points of $G$ and invariant with respect to the action of $G$ by right translations. 

%The first three examples are from \cite{BK8}.

(1) If group $G$ is residually finite (respectively, residually nilpotent), i.e., for any element $t \in G$, $t \neq e$, there exists a normal subgroup $G_t \not\ni t$ such that $G/G_t$ is finite (respectively, nilpotent), we consider the closed algebra $\hat{\ell}_\infty(G) \subset \ell_{\infty}(G)$ generated by pullbacks to $G$ of algebras $\ell_\infty(G/G_t)$ for all $G_t$ as above.

(2) Recall that a (continuous) bounded function $f$ on a (topological) group $G$ is called {\em almost periodic} if the families of its left and right translates $$\{t \mapsto f(st)\}_{s \in G}, \quad \{t \mapsto f(ts)\}_{s \in G}$$ are relatively compact in $\ell_\infty(G)$ (J.~von Neumann \cite{N}). (It was proved in \cite{Ma}  that
the relative compactness of either the left or the right family of translates already gives the almost periodicity on $G$.) 
The algebra of almost periodic functions on $G$ is denoted by $AP(G)$.

The basic examples of almost periodic functions on $G$ are given by the matrix elements of the finite-dimensional irreducible unitary representations of $G$.

Recall that group $G$ is called \textit{maximally almost periodic} if its finite-dimensional irreducible unitary representations separate points.
Equivalently, $G$ is maximally almost periodic iff it admits a monomorphism into a compact topological group. 

Any residually finite group
%, i.e., a group such that the intersection of all its finite index normal subgroups is trivial, 
belongs to this class. In particular, $\mathbb Z^n$, finite groups, free groups, finitely generated nilpotent groups, pure braid groups, fundamental groups of three dimensional manifolds are maximally almost periodic. 

We denote by $AP_{\trig}(G) \subset AP(G)$ the space of functions
\begin{equation}\label{ap0}
t \mapsto \sum_{k=1}^m c_k \sigma^k_{ij}(t), \quad t\in G,\quad c_k \in \mathbb C, \quad \sigma^k=(\sigma^k_{ij}),
\end{equation}
where $\sigma^k$ ($1 \leq k \leq m$) are finite-dimensional irreducible unitary representations of $G$.
The von~Neumann approximation theorem \cite{N} states that $AP_{\trig}(G)$ is dense in $AP(G)$. 

In particular, the algebra $AP(\mathbb Z^n)$ of almost periodic functions on $\mathbb Z^n$ contains as a dense subset the subalgebra of exponential polynomials
%\begin{equation*}
$t \mapsto \sum_{k=1}^m c_ke^{i\langle \lambda_k,t \rangle}$, $t \in \mathbb Z^n$, $\lambda_k \in \mathbb R^n$, $m\in\mathbb N.$
%\end{equation*}
Here $\langle \lambda_k,\cdot \rangle$ denotes the linear functional defined by $\lambda_k$.

(3) The algebra $AP_{\mathbb Q}(\mathbb Z^n)$ of almost periodic functions on $\mathbb Z^n$ with rational spectra. This is the subalgebra of $AP(\mathbb Z^n)$ generated over $\mathbb C$ by functions $t \mapsto e^{i\langle \lambda,t \rangle}$ with $\lambda \in \mathbb Q^n$.

(4) If group $G$ is finitely generated then, in addition to the subalgebra $c(G) \subset \ell_\infty(G)$ of functions having  limits at `infinity', we can define a subalgebra $c_{E}(G) \subset \ell_\infty(G)$ of functions having limits at `infinity' along each `path'.

To make this definition precise, we will need the notion of the end compactification of a connected and locally connected topological space $T$ that admits an exhaustion by compact subsets $K_i$, $i\in\mathbb N$, whose interiors cover $T$. Recall that the set of ends $E=E_T$ of space $T$ is the inverse limit of an inverse system of discrete spaces $\{\pi_0(T \setminus K_i)\}$, where $\pi_0(T \setminus K_i)$ is the set of connected components of $T \setminus K_i$, and each inclusion $T \setminus K_j \subset T \setminus K_i$, $i \leq j$, induces projection $\pi_0(T\setminus K_j) \rightarrow \pi_0(T \setminus K_i)$. The end compactification $\bar{T}_E$ of $T$ is a compact space defined as the disjoint union $T \sqcup E_T$ endowed with the weakest topology containing all open subsets of $T$ and all open neighbourhoods of the ends: an open neighbourhood of an end $e=\{e_i \in \pi_0(T \setminus K_i)\, ,\, i\in\mathbb N\}$ is a subset $V \subset T \sqcup E_T$ such that $V\cap E_T$ and $V \cap T$ are open in the corresponding topologies and $e_i\subset V\cap T$ for some $i\in\mathbb N$, see \cite{Fre}. 

Now, suppose that group $G$ is finitely generated. By $\bar{G}_E$ we denote the end compactification of the Cayley graph ${\mathcal C}_G$ of $G$. Identifying naturally $G$ with the vertex set of ${\mathcal C}_G$ we define the subalgebra
$c_E(G) \subset \ell_\infty(G)$ of functions admitting continuous extensions to $\bar{G}_E$. 
For example, if $G=\mathbb Z$, then $E=\{\pm \infty\}$, and $c_E(\mathbb Z)$ consists of functions $\mathbb Z \rightarrow \mathbb C$ having limits at $\pm\infty$.
%Note that $c(G)$ is a closed subalgebra of $c_E(G)$.

(5) For a finitely generated group $G$, let $SAP(G) \subset \ell_\infty(G)$ denote the minimal subalgebra containing $AP(G)$ and $c_E(G)$.  Elements of $SAP(G)$ are called \textit{semi-almost periodic functions} (this is a variant of definition in \cite{Sar} for $G=\mathbb R$), see~Example \ref{hsap} below.

(6) Let $N$ be an infinite subgroup of $G$ and $N\backslash G$ be the set of (right) conjugacy classes. For a given class $Nx \in N\backslash G$ endowed with the discrete topology by $c(Nx)$ we denote the subalgebra of bounded functions $Nx \rightarrow \mathbb C$ that admit extensions to the one-point compactification of $Nx$.
Let $c_{N}(G) \subset \ell_\infty(G)$ denote the subalgebra consisting of functions $h$ such that $$h|_{Nx} \in c(Nx) \quad \text{ for each } Nx \in N\backslash G.$$
(Thus, $h$ has limits `at infinity' along each conjugacy class.)

Every function $h \in c_{N}(G)$ can be viewed as a bounded function on $N\backslash G$ with values in Banach algebra $c(N)$, i.e., $$h \in \ell_\infty(N\backslash G, c(N)).$$ Instead of $\ell_\infty(N\backslash G,c(N))$ we may consider other Banach algebras of $c(N)$-valued functions on $N\backslash G$, e.g.,~$c(N\backslash G, c(N))$, thus obtaining other subalgebras of $\ell_\infty(G)$ satisfying assumptions of Section \ref{introsect}.

\subsection{Holomorphic almost periodic functions on coverings of complex manifolds}
\label{holap3}
In \cite{BK8} we defined holomorphic almost periodic functions on a regular covering $X \rightarrow X_0$ as
elements of algebra $\mathcal O_{AP}(X)$ (see~subsection \ref{exm}(2) for the definition of algebra $AP=AP(G)$).
Equivalently, a function $f \in \mathcal O(X)$ is called holomorphic almost periodic if each $G$-orbit in $X$ has a neighbourhood $U$ that is invariant with respect to the (left) action of $G$, such that the family of translates $\{z \mapsto f(g \cdot z), z \in U\}_{g \in G}$ is relatively compact in the topology of uniform convergence on $U$ (see \cite{BK2} for the proof of the equivalence). 

This is a variant of definition in \cite{W}, where $G$ is taken to be the group of all biholomorphic automorphisms of a complex manifold $X$ (see also \cite{Ves}).

%(an interesting result in \cite{Ves} states that on Siegel domains of the second kind there are no non-constant holomorphic almost periodic functions in the sense of \cite{W}, although on Siegel domains of the first kind (i.e., on tube domains in $\mathbb C^n$) the holomorphic almost periodic functions even separate points).

For instance, if $X_0$ is a non-compact Riemann surface and $p:X \rightarrow X_0$ is a regular covering with a maximally almost periodic deck transformation group $G$,
%(for instance, $X_0$ is hyperbolic, then $X=\mathbb D$ is its universal covering, and $G=\pi_1(X_0)$ is a free (not necessarily finitely generated) group);
then functions in $\mathcal O_{AP}(X)$ arise, e.g., as linear combinations over $\mathbb C$ of matrix entries of fundamental solutions of certain linear differential equations on $X$ (see subsection \ref{usualtrigex}(2) for details).

\medskip

We say that the covering $p:X \rightarrow X_0$ has the $\mathcal O_{\mathfrak a}$-{\em Liouville property} if $\mathcal O_{\mathfrak a}(X)$ does not contain non-constant bounded functions.

Recall that a complex manifold $X_0$ is called \textit{ultraliouville} if there are no non-constant bounded continuous plurisubharmonic functions on $X_0$
%(Recall that an upper-semicontinuous function $f:X_0 \rightarrow [-\infty,\infty)$ is called plurisubharmonic if for any holomorphic map $q: \mathbb D \rightarrow Y$, $\mathbb D \subset \mathbb C$ is the open unit disk, the pullback $q^*f$ is subharmonic on $\mathbb D$.) 
(e.g.,~connected compact complex manifolds and their Zariski open subsets are ultraliouville).

According to \cite{Lin},
if $X_0$ is ultraliouville and $G$ is virtually nilpotent (i.e., contains a nilpotent subgroup of finite index), then $X$ has the $\mathcal O_{\ell_\infty}$-Liouville property. For holomorphic almost periodic functions on $X$ this result can be strengthened, see \cite[Th. ~2.3]{BK2}:

{\em Let $p:X\rightarrow X_0$ be a regular covering of an ultraliouville complex manifold $X_0$. Then

(1) $X$ has the $\mathcal O_{AP}$-Liouville property.

(2) Let $n \geq 2$,
$D_0 \Subset X_0$ be a subdomain with a connected piecewise smooth boundary $\partial D_0$ contained in a Stein open submanifold of $X_0$, and $D:=p^{-1}(D_0)$.
Then $X \setminus D$ has $\mathcal O_{AP}$-Liouville property.}
%\end{theorem}

%Assertion (2) of Theorem \ref{liouvillethm} follows from (1) and Theorem \ref{hartogsthm}.

For instance, consider the universal covering $p:\mathbb D\rightarrow \mathbb C \setminus \{0,1\}$ of doubly punctured complex plane (here the deck transformation group is free group with two generators). Although there are plenty of non-constant bounded holomorphic functions on $\mathbb D$, all bounded holomorphic almost periodic functions on $\mathbb D$ corresponding to this covering are constant because $\mathbb C \setminus \{0,1\}$ is ultraliouville.

%\end{example}

For other properties of algebra $\mathcal O_{AP}(X)$ see subsection \ref{zeroapsect} below.

%\medskip
%
%Another interesting property of subalgebra $\mathcal O_{AP}(X)$ is related to the properties of complex almost periodic submanifolds at `infinity', see Proposition \ref{biholprop} below.

%Suppose that $X_0$ is a Stein manifold, and $G$ is a maximally almost periodic group. In Theorem \ref{corona} below we show that the maximal ideal space of algebra $\mathcal O_{AP}(X)$ is homeomorphic to the fibrewise compactification $c_{AP}X$ of covering $X$, introduced in Section \ref{constrsect}.
%As a set, the fibrewise compactification $c_{AP}X$ is a disjoint union of copies $X_H$ ($H \in \Upsilon$) of $X$, with one these copies being $X$ itself
%%; each copy $X_H$ is dense in $c_{\mathfrak a}X$ 
%(see Section \ref{structsect} and Example \ref{bohrex} for details).
%
%Now, let $Y$ be an almost periodic submanifold of $X$. It follows from the results of Section \ref{submanifoldsect} that the closure of $Y$ (identified with $\iota(Y)$) in $c_{AP}(X)$ is a disjoint union of (different) almost periodic submanifolds $Y_H \subset X_H$, $H \in \Upsilon$.
%%; one of these submanifolds is $Y$ itself; the other ones can be viewed as instances of $Y$  at `infinity'.
%
%\begin{proposition}
%Any two $Y_{H_1}$, $Y_{H_2}$ ($H_1,H_2 \in \Upsilon$) are biholomorphic.
%\end{proposition}

\subsection{Holomorphic semi-almost periodic functions}
\label{hsap}
Suppose that group $G$ is finitely generated. Elements of algebra $\mathcal O_{SAP}(X)$ (see~Example \ref{exm}(5)) are called \textit{holomorphic semi-almost periodic functions}.
By Theorem \ref{approxthm} algebra $\mathcal O_{SAP}(X)$ is generated by subalgebras $\mathcal O_{AP}(X)$ (see~Example \ref{holap3}) and $\mathcal O_{c_E}(X)$ (see~Example \ref{exm}(4)). In the case
$T \rightarrow T_0$ is a complex strip covering an annulus $T_0$ (see~Example \ref{holap}), algebra $\mathcal O_{SAP}(T)$ is related
to the subalgebra of Hardy algebra $H^\infty(\mathbb D)$ of bounded holomorphic functions on the unit disk $\mathbb D \subset \mathbb C$ generated by functions whose moduli have only the first-kind boundary discontinuities (see \cite{BrK2}).

\subsection{Examples of complex $\mathfrak a$-submanifolds}
\label{submex}
%The following are the examples of complex $\mathfrak a$-submanifolds.
We assume that subalgebra $\mathfrak a$ is self-adjoint.

(1) If $Z_0 \subset X_0$ is a complex submanifold of codimension $k$, then $Z:=p^{-1}(Z_0) \subset X$ is a complex $\mathfrak a$-submanifold of codimension $k$. 

(2) The disjoint union of a finite collection of complex $\mathfrak a$-submanifolds $Z_i$ of $X$ separated by functions in $C_{\mathfrak a}(X)$ (i.e., for each $i$ there is $f \in C_{\mathfrak a}(X)$ such that $f=1$ on $Z_i$ and $f=0$ on $Z_j$ for  $j \neq i$) is a complex $\mathfrak a$-submanifold.

(3) Let $Z_0=\{x \in X_0: f_1(x)=\dots=f_k(x)=0\}$ for some $f_i \in \mathcal O(X_0)$ $(1 \leq i \leq k)$ be a complex submanifold of $X_0$ of codimension $k$.
Set $Z:=p^{-1}(Z_0)$.
Further, for an open subset $X_0' \Subset X_0$ and functions $h_1,\dots,h_k \in \mathcal O_{\mathfrak a}(X)$ we define $X':=p^{-1}(X_0')$, $\delta:=\sup_{x \in X'}\max_{1 \leq i \leq k}|h_i(x)|,$
and
$$Z_h:=\{x \in X: p^*f_1(x)+h_1(x)=\dots=p^*f_k(x)+h_k(x)=0\}.$$
Using the inverse function theorem with continuous dependence on parameter (Theorem \ref{ift}), it is not difficult to see that $Z_h$ is a complex $\mathfrak a$-submanifold of $X'$ provided that $\delta>0$ is sufficiently small.

% explicit example?

(4) A complex $\mathfrak a$-submanifold of $X$ is called \textit{cylindrical} if each open set $V$ in Definition \ref{manifolddef} has form $V=p^{-1}(V_0)$ for some open $V_0 \subset X_0$ (i.e., it is determined by holomorphic $\mathfrak a$-functions on preimages by $p$ of open subsets of $X_0$).

In \cite{BK2} we constructed a non-cylindrical $\mathfrak a$-hypersurface in $X$ in the case $\mathfrak a=AP(\mathbb Z)$ (see~subsection \ref{holap3}) and $p:X \rightarrow X_0$ is a regular covering of a Riemann surface $X_0$ with deck transformation group $\mathbb Z$. We assumed that $X_0$ has finite type and is a relatively compact subdomain of a larger (non-compact) Riemann surface $\tilde{X}_0$ whose fundamental group satisfies $\pi_1(\tilde{X}_0) \cong \pi_1(X_0)$ (e.g., the covering of Example \ref{holap} with $n=1$, i.e., a complex strip covering an annulus, is a regular covering of this form).

Let us briefly describe this construction.

The covering $X$ of $X_0$ admits an injective holomorphic map into a holomorphic fibre bundle over $X_0$ having fibre $(\mathbb C^*)^2$, $\mathbb C^*:=\mathbb C \setminus \{0\}$, defined as follows. 
First, note that the regular covering $p:X \rightarrow X_0$ admits presentation as a principal fibre bundle with fibre $\mathbb Z$, see (\ref{extrem2}).
We choose two characters $\chi_1, \chi_2:\mathbb Z\rightarrow\mathbb S^1\cong \mathbb R/(2\pi\mathbb Z)$ such that the homomorphism $(\chi_1,\chi_2):\mathbb Z\rightarrow \mathbb T^2=\mathbb S^1\times\mathbb S^1$ is an embedding with dense image. Consider the fibre bundle $b_{\mathbb T^2}X$ over $X_0$ with fibre $\mathbb T^2$ associated with the principal fibre bundle $p:X \rightarrow X_0$ via the homomorphism $(\chi_1,\chi_2)$.
The bundle $b_{\mathbb T^2}X$ is embedded into a holomorphic fibre bundle $b_{(\mathbb C^*)^2}X$ with fibre $(\mathbb C^*)^2$ associated with the composite of the embedding homomorphism $\mathbb T^2\hookrightarrow (\mathbb C^*)^2$ and $(\chi_1,\chi_2)$. 
Now, the covering $X$ of $X_0$ admits an injective $C^\infty$ map into $b_{\mathbb T^2}X$ with dense image and the composite of this 
map with the embedding of $b_{\mathbb T^2}X$ into $b_{(\mathbb C^*)^2}X$ is an injective holomorphic map $X\rightarrow b_{(\mathbb C^*)^2}X$. Further, the bundle $b_{(\mathbb C^*)^2}X$ admits a holomorphic trivialization $\eta: b_{(\mathbb C^*)^2}X \rightarrow X_0 \times (\mathbb C^*)^2$. We choose $\chi_1(1)$ and $\chi_2(1)$ so close to $1\in\mathbb S^1$ that the image $\eta(b_{\mathbb T^2}X)\subset X_0 \times (\mathbb C^*)^2$ is sufficiently close to $X_0 \times \mathbb T^2$. Thus identifying $X$ (by means of holomorphic injection $X\hookrightarrow b_{(\mathbb C^*)^2}X \stackrel{\eta}{\rightarrow} X_0 \times (\mathbb C^*)^2$) with a subset of $X_0 \times (\mathbb C^*)^2$, we obtain that $X$ is sufficiently close to $X_0 \times \mathbb T^2$. 
Next, we construct a smooth complex hypersurface in $X_0 \times (\mathbb C^*)^2$ such that in each cylindrical coordinate chart $U_0 \times (\mathbb C^*)^2$ on $X_0 \times (\mathbb C^*)^2$ for $U_0 \Subset X_0$ simply connected it cannot be determined as the set of zeros of a holomorphic function on $U_0 \times (\mathbb C^*)^2$. Intersecting this hypersurface with $X$ we obtain a non-cylindrical almost periodic hypersurface in $X$.
(To construct such a hypersurface in $X_0 \times (\mathbb C^*)^2$, we determine a smooth divisor in $(\mathbb C^*)^2$ that has a non-zero Chern class (i.e., it cannot be given by a holomorphic function on $(\mathbb C^*)^2$), and whose support intersects the real torus $\mathbb T^2 \subset (\mathbb C^*)^2$ transversely. Then we take the pullback of this divisor with respect to the projection $X_0 \times (\mathbb C^*)^2 \rightarrow  (\mathbb C^*)^2$ to get the desired hypersurface.)

\subsection{Examples of spaces $\mathfrak a_\iota$ in Theorem \ref{approxthm}}
\label{approxex}
%We have the following examples of spaces $\mathfrak a_\iota$:

(1) Let $\mathfrak a=\ell_\infty(G)$, $I$ be the collection of all subsets of $G$ ordered by inclusion. It is easy to verify that given $\iota \in I$ we can define $\mathfrak a_{\iota}$ to be the closed linear subspace spanned by translates $\{R_g(\chi_\iota): g \in G\}$ of the characteristic function $\chi_\iota$ of subset $\iota$.

(2) Let $\mathfrak a=AP(\mathbb Z^n)$ (see~subsection \ref{exm}(2)). We can take $I$ to be the collection of all finite subsets of $\mathbb R^n$ ordered by inclusion and $\mathfrak a_\iota(\mathbb Z^n):=\spana_{\mathbb C}\{t \mapsto e^{i\langle \lambda, t \rangle}, \lambda \in \iota, \iota \in I, t \in \mathbb Z^n \}$.

We can also consider $\mathfrak a=AP_{\mathbb Q}(\mathbb Z^n)$, the algebra of almost periodic functions on $\mathbb Z^n$ having rational spectra (see~subsection \ref{exm}(3)). 
Here we take $I$ to be the collection of all finite subsets of $\mathbb Q^n$ ordered by inclusion and define spaces $\mathfrak a_\iota(\mathbb Z^n)$ similarly to the above.

(3) Let $\mathfrak a=AP(G)$ (see~subsection \ref{exm}(2)) and $I$ consist of finite collections of finite-dimensional irreducible unitary representations of group $G$. We define $\mathfrak a_{\iota}(G)$, where $\iota=\{\sigma_1,\dots,\sigma_m\} \in I$, to be the linear $\mathbb C$-hull of matrix elements $\sigma^{ij}_k \in AP(G)$ of representations $\sigma_k=(\sigma_k^{ij})$, $1 \leq k \leq m$.

\section{Comments}

\subsection{Equivalent definition of holomorphic $\mathfrak a$-functions}% on complex $\mathfrak a$-submanifolds}
\label{equivcur}
Let $\Lambda_{c}^{t,s} (X)$ denote the space of smooth $(t,s)$-forms on $X$ with compact supports endowed with the standard topology (see, e.g., \cite{Dem}). Recall that continuous linear functionals on $\Lambda_{c}^{t,s} (X)$ are called $(n-t,n-s)$-currents.

There is an equivalent definition of holomorphic $\mathfrak a$-functions on a complex $\mathfrak a$-submanifold $Z$ (see~Definition \ref{holdef2}) in terms of currents. 
Namely, let $\mathfrak a$ be self-adjoint, then
a function $f \in \mathcal O(Z)$ on a complex $\mathfrak a$-submanifold $Z \subset X$ is a holomorphic $\mathfrak a$-function if and only if 
it is bounded on subsets $Z \cap p^{-1}(U_0)$, $U_0 \Subset X_0$, and the corresponding current $c_f$, 
\begin{equation}
\label{current}
(c_f,\varphi):=\int_{Z}f\varphi, \quad \varphi \in \Lambda_{c}^{m,m}(X), \quad m:=\dim_{\mathbb C}Z,
\end{equation}
is an $\mathfrak a$-current meaning that for each $\varphi $ the function $G \ni g \mapsto\bigl(c_f,\varphi_g\bigr)$ belongs to algebra $\mathfrak a$; here $\varphi_g(x):=\varphi(g \cdot x)$ ($x \in X$). (The proof follows an argument in \cite[Prop.~2.4]{Fav2}.)

In the setting of Example \ref{holap} (holomorphic almost periodic functions on tube domains)
almost periodic currents were studied, e.g., in \cite{F2R} (see further references therein).

\subsection{Cylindrical $\mathfrak a$-divisors}
\label{divrem}

The class of $\mathfrak a$-principal divisors is contained in a larger class of cylindrical $\mathfrak a$-divisors, i.e., $\mathfrak a$-divisors determined by functions $f_\alpha \in \mathcal O_{\mathfrak a}(U_\alpha)$ with $U_\alpha=p^{-1}(U_{0,\alpha})$ for some open $U_{0,\alpha} \subset X_0$ (see~Definition \ref{defdivisor0}).

If covering dimension of the maximal ideal space $M_{\mathfrak a}$ of  $\mathfrak a$ is zero, then every $\mathfrak a$-divisor is $\mathfrak a$-equivalent to a cylindrical $\mathfrak a$-divisor (the latter follows from an equivalent definition of $\mathfrak a$-divisors as divisors on fibrewise compactification $c_{\mathfrak a}X$, see~\cite{BK2}). 
In particular, all $\ell_\infty$-, $\hat{\ell}_\infty(G)$- (for a residually finite group $G$), $AP_{\mathbb Q}$-divisors (see~(1) and (3) in subsection \ref{exm}) are $\ell_\infty$-, $\hat{\ell}_\infty(G)$-, $AP_{\mathbb Q}$-equivalent to cylindrical divisors (see~Examples  3.3(3) and 3.3(4) in \cite{BK8}). There are, however, non-cylindrical $AP$-divisors, see~subsection 4.4 in \cite{BK2}.

For an $\mathfrak a$-divisor $E$ on $X$ Theorem \ref{chernthm} implies the following:

(a)  If there exists a function $f \in \mathcal O_{\mathfrak a}(U)$, where $U=p^{-1}(U_{0})$, $U_0 \subset X_0$ is open, such that $E|_{U}$ is determined by $f$, then $E$ is $\mathfrak a$-equivalent to a cylindrical divisor (see~the argument in the proof of Theorem \ref{chernthm}).

(b) If $\mathfrak a=AP(G)$ and $E$ not $\mathfrak a$-equivalent to a cylindrical $\mathfrak a$-divisor, then the projection of $\supp(E)$ to $X_0$ is dense (the converse is not true, see subsection \ref{submex}(4)).

\subsection{Almost periodic divisors}
\label{zeroapsect}

%We will use definitions and notation of subsection \ref{holap3} (the algebra $\mathcal O_{AP}(X)$ of holomorphic almost periodic functions on covering $X$). Almost periodic (i.e.,~$AP$-) divisors have a number of additional properties:

%(1) The converse to the assertion of Proposition \ref{zeroapcor} is true: if $f \in \mathcal O(X)$, $|f| \in C_{AP}(X)$, then $E_f$ is equivalent to a divisor in $\Div_{AP}(X)$. 

%(1) Suppose that $X_0$ is an arbitrary complex manifold, $\mathfrak a$ is self-adjoint, and we are given a function $f \in \mathcal O(X)$ with $|f| \in C_{\mathfrak a}(X)$ such that there is an open subset
%$Y \Subset X$ for which no net $\{g_\alpha\} \subset G$ with the property that the net $\bigl\{x \mapsto f\bigl(g_\alpha \cdot x\bigr)\}$ of
%translates converges uniformly to $0$ on $\bar{Y}$ exists. 
%It is not difficult to show that divisor $E_f \in \Div(X)$ is equivalent to a divisor $E \in \Div_{\mathfrak a}(X)$. 

We use notation introduced in subsection \ref{equivcur}. Let $T_E$ be the current of integration of a divisor $E \in \Div(X)$, i.e., 
\begin{equation*}
(T_E,\varphi):=\int_{E}\varphi, \qquad \varphi \in \Lambda_{c}^{n-1,n-1}(X)
\end{equation*}
(see, e.g.,~\cite{Dem}).
%(Here we use the fact that locally, in the usual topology on $X$, divisor $E$ can be presentated as a collection of analytic hypersurfaces with prescribed multiplicities.)
%It is easy to see that $E$, $E' \in \Div(X)$ are equivalent if and only if $T_E=T_{E'}$.
One can prove that if $E \in \Div_{AP}(X)$, then
current $T_E$
is almost periodic. 
Conversely, if the current of integration $T_E$ of a divisor $E \in \Div(X)$ is almost periodic, then $E$ is equivalent to an $AP$-divisor.

\subsection{Approximation of holomorphic almost periodic functions}
\label{usualtrigex}
(1) Let \penalty-10000 $\mathcal O_{\trig}(T)\subset \mathcal O_{AP(\mathbb Z^n)}(T)$ be a subspace determined by the choice of spaces $\mathfrak a_\iota=\mathfrak a_\iota(\mathbb Z^n)$ ($\iota \in I$) as in subsection \ref{approxex}(2).
We show that exponential polynomials, see (\ref{expoly}),
are dense in $\mathcal O_{\trig}(T)$.

We denote $e_\lambda(t):=e^{i\langle \lambda,t\rangle}$ ($\lambda \in \mathbb R^n$, $t \in \mathbb Z^n$).
Clearly, $e_{\lambda} \in \mathcal O_{\{\lambda\}}(T)$. 
Now, let $\iota=\{\lambda_1,\dots,\lambda_m\}$.
Since functions $e_{\lambda_k}$ ($1 \leq k \leq m$) are linearly independent in $\mathfrak a_\iota$, 
there exist linear projections $p_{\iota,\lambda_k}:\mathfrak a_\iota \rightarrow \mathfrak a_{\{\lambda_k\}}$. Since projections $p_{\iota,\lambda_k}$, $1 \leq k \leq m$, are invariant with respect to the action of $G$ on itself by right translates, they determine projections $P_{\iota,\lambda_k}:\mathcal O_{\iota}(T) \rightarrow \mathcal O_{\{\lambda_k\}}(T)$. 
(The latter follows, e.g., from the presentation of functions in $\mathcal O_{AP}(T)$ as sections of holomorphic Banach vector bundle $C_{AP}X_0$, see~(\ref{extrem3}), where projections $P_{\iota,\lambda_k}$ become bundle homomorphisms $C_{\mathfrak a_\iota}X_0 \rightarrow C_{\mathfrak a_{\{\lambda_k\}}}X_0$.)
Therefore, there exist functions $f_{\lambda_k} \in \mathcal O_{\{\lambda_k\}}(T)$, $f_{\lambda_k}:=P_{\iota,\lambda_k}(f)$, $1 \leq k \leq m$, such that 
$f(z)=\sum_{k=1}^m f_{\lambda_k}(z)$, $z \in T$. It is easy to see that for each $f_{\lambda_k}$ there exists a function $h_{\lambda_k} \in \mathcal O(T_0)$ such that $f_{\lambda_k}/e_{\lambda_k}=p^*h_{\lambda_k}$; hence,
\begin{equation}
\label{usualtrigrepr}
f(z)=\sum_{k=1}^m (p^*h_{\lambda_k})(z)e^{i \langle \lambda_k,z\rangle}, \quad z \in T.
\end{equation}
Since the base $T_0$ of the covering is a relatively compact Reinhardt domain, functions $h_{\lambda_k}$ admit expansions into Laurent series (see, e.g., \cite{Shab})
\begin{equation*}
h_{k}(z)=\sum_{|\alpha|=-\infty}^\infty b_\alpha z^{\alpha}, \quad z \in T_0, \quad b_t \in \mathbb C,
\end{equation*}
where $\alpha=(\alpha_1,\dots,\alpha_n)$ is a multiindex, $|\alpha|:=\alpha_1+\dots+\alpha_n$.
Since $p(z)=\bigl(e^{i z_1}, \dots, e^{i z_n}\bigr)$, $z=(z_1,\dots,z_n) \in T$ (see~Example \ref{holap}), each $p^*h_{\lambda_k}$ admits an approximation by finite sums 
\begin{equation}
\label{usualtrigrepr2}
\sum_{|\alpha|=-M}^M b_\alpha e^{i \langle \alpha, z\rangle}, \quad z \in T,
\end{equation}
converging uniformly on subsets $p^{-1}(W_0) \subset T$, $W_0 \Subset T_0$.
Together with (\ref{usualtrigrepr}) this implies that exponential polynomials (\ref{expoly}) are dense in $\mathcal O_0(T)$.

A similar argument shows that the algebra of holomorphic almost periodic functions with rational spectra (whose elements admit approximations by exponential polynomials (\ref{expoly}) with $\lambda_k \in \mathbb Q^n$) coincides with algebra $\mathcal O_{AP_{\mathbb Q}}(T)$ (see~subsection \ref{exm}(3)). 

\medskip

%\subsection{Approximation of holomorphic almost periodic functions on coverings of complex manifolds}
%\label{holap3approx}
(2) Let $X_0$ be a non-compact Riemann surface, $p:X \rightarrow X_0$ be a regular covering with a maximally almost periodic deck transformation group $G$ (for instance, $X_0$ is hyperbolic, $X=\mathbb D$ is its universal covering and $G=\pi_1(X_0)$ is a free (not necessarily finitely generated) group).
Functions in $\mathcal O_{AP}(X)$ (see~subsection \ref{holap3}) arise, e.g., as linear combinations over $\mathbb C$ of matrix entries of fundamental solutions of certain linear differential equations on $X$.

Indeed, let
$\mathcal U_G$ be the set of finite dimensional irreducible unitary representations $\sigma:G \rightarrow U_m$ ($m \geq 1$), $I$ be the collection of finite subsets of $\mathcal U_G$ directed by inclusion, and for each $\iota \in I$  let $AP_{\iota}(G)$ be the (finite-dimensional) subspace generated by matrix elements of the unitary representations $\sigma \in \iota$. Then by Theorem \ref{approxthm} the $\mathbb C$-linear hull $\mathcal O_{\trig}(X)$ of spaces $\mathcal O_{\iota}(X)$ is dense in $\mathcal O_{\mathfrak a}(X)$ (note that for each $\sigma \in \mathcal U_G$ the space $\mathcal O_{\{\sigma\}}(X)$ is the $\mathbb C$-linear hull of coordinates of vector-valued functions $f$ in $\mathcal O(X,\mathbb C^m)$ having the property that $f(g \cdot x)=\sigma(g)f(x)$ for all $g \in G$, $x \in X$).
Now, a unitary representation $\sigma:G\rightarrow U_m$, $m \geq 1$, can be obtained as the monodromy of the system $dF=\omega F$ on $X_0$, where $\omega$ is a holomorphic $1$-form on $X_0$ with values in the space of $m \times m$ complex matrices $M_m(\mathbb C)$ (see, e.g., \cite{For}). In particular, the system $dF=(p^*\omega) F$ on $X$ admits a global solution $F \in \mathcal O\bigl(X,GL_m(\mathbb C)\bigr)$  such that $F \circ g^{-1}=F\sigma(g)$ $(g \in G)$. By definition, a linear combination of matrix entries of $F$ is an element of $\mathcal O_{AP}(X)$.

\subsection{Approximation property}
Recall that a (complex) Banach space $B$ is said to have the approximation property if for every compact set $K\subset B$ and every $\varepsilon>0$ there is a bounded operator $T=T_{\varepsilon,K} \in \mathcal L(B,B)$ of finite rank so that $$\|Tx-x\|_B <\varepsilon \quad \text{ for every } x \in K.$$
For example,
space $AP(G)$ of almost periodic functions on a group $G$ (see~subsection \ref{exm}(2)) has the approximation property with (approximation) operators $T$ in $\mathcal L(AP(G),AP_{\trig}(G))$ (see, e.g., an argument in \cite{Sh}).

In subsection \ref{approx} suppose additionally to conditions (1)--(3) that

(4) spaces $\mathfrak a_\iota$, $\iota \in I$, are finite-dimensional, and

(5) space $\mathfrak a$ has the approximation property with approximation operators
$S \in \mathcal L(\mathfrak a,\mathfrak a_0)$ equivariant with respect to the action of $G$ on $\mathfrak a$ by right translations, i.e., $S\bigl(R_g(f)\bigr)=R_g\bigl(S(f)\bigr)$ for all $f \in \mathfrak a$, $g \in G$.

One can show that if $X_0$ is a Stein manifold and $D_0 \Subset X_0$ is a strictly pseudoconvex domain, then the Banach space $\mathcal A_{\mathfrak a}(D):=\mathcal O_{\mathfrak a}(D) \cap C_{\mathfrak a}(\bar{D})$, $D:=p^{-1}(D_0)$, has the approximation property with approximation operators in $\mathcal L\bigl(\mathcal A_{\mathfrak a}(D),\mathcal A_{\trig}(D)\bigr)$
(here $\mathcal A_{\trig}(D)$ is defined similarly to $\mathcal O_{\trig}(D)$ in Theorem \ref{approxthm}).

\section{Structure of fibrewise compactification $c_{\mathfrak a}X$}

We refer to Section \ref{structproofsect} for the proofs of the results formulated in the present section.

\label{compsect}

\subsection{Complex structure}
\label{holfuncsect}
\label{charts}
\label{notsect}

%
%
%Let $\mathcal O(U)$ denote the algebra of holomorphic functions on $U$, endowed with the topology of uniform convergence on compact subsets of $U$.
%It is immediate that a function $f \in C(c_{\mathfrak a}X)$ is in $\mathcal O(c_{\mathfrak a}X)$ if and only if each point in $c_{\mathfrak a}X$ has a neighbourhood $U$ such that $f|_{U} \in \mathcal O(U)$. 

A function $f \in C(U)$ on an open subset $U \subset c_{\mathfrak a}X$ is called {\em holomorphic}, i.e., belongs to the space $\mathcal O(U)$, if 
$\iota^* f$ is holomorphic on $V:=\iota^{-1}(U) \subset X$ in the usual sense (see subsection \ref{constrsect} for notation).

The proof of the next proposition is similar to the proof of Proposition 2.3 in \cite{BK8}.

\begin{proposition}
\label{basicpropthm}
The following is true:
\begin{itemize}

\item[(1)]
A function $f$ in $C_{\mathfrak a}(V)$ determines a unique function $\hat{f}$ in $C(U)$ such that $\iota^*\hat{f}=f$; moreover,  $f \in \mathcal O_{\mathfrak a}(V)$ if and only if $\hat{f} \in \mathcal O(U)$.
Thus, there are monomorphisms
$C_{\mathfrak a}(V) \hookrightarrow C(U)$, $\mathcal O_{\mathfrak a}(V) \hookrightarrow \mathcal O(U).$

\item[(2)]If $\mathfrak a$ is self-adjoint, then $C_{\mathfrak a}(V) \cong C(U)$ and
$\mathcal O_{\mathfrak a}(V) \cong \mathcal O(U)$.

\end{itemize}
\end{proposition} 

Let $U_0 \subset X_0$ be open. 
A function $f \in C(U)$ on an open subset $U\subset U_0 \times \hat{G}_{\mathfrak a}$ is called {\em holomorphic} if the function $\tilde j^*f$, where $\tilde j:={\rm Id}\times j :U_0\times G\rightarrow  U_0 \times \hat{G}_{\mathfrak a}$, is holomorphic on the open subset $\tilde j^{-1}(U)$ of the complex manifold $U_0\times G$ (see subsection \ref{constrsect} for the definition of the map $j$).

For sets $U$ as above, by $\mathcal O(U)$ we denote the algebra of holomorphic functions on $U$ endowed with the topology of uniform convergence on compact subsets of $U$.
Clearly, $f \in C(c_{\mathfrak a}X)$ belongs to $\mathcal O(c_{\mathfrak a}X)$ if and only if each point in $c_{\mathfrak a}X$ has an open neighbourhood $U$ such that $f|_{U} \in \mathcal O(U)$.

By $\mathcal O_{U}$ we denote the sheaf of germs of holomorphic functions on $U$.

The category $\mathcal M$ of ringed spaces of the form $(U,\mathcal O_{U})$, where $U$ is either an open subset of $c_{\mathfrak a}X$ and $X$ is a regular covering of a complex manifold $X_0$ or an open subset of $U_0 \times \hat{G}_{\mathfrak a}$ with $U_0\subset X_0$ open, contains in particular complex manifolds. %(We expand category $\mathcal M$ in Section \ref{submanifoldsect} below.)

\begin{definition}
\label{defholmap}
A morphism of two objects in $\mathcal M$, that is, a map $F \in C(U_1,U_2)$, where $(U_i,\mathcal O_{U_i})\in\mathcal M$, $i=1,2$, such that $F^*\mathcal O_{U_2} \subset \mathcal O_{U_1}$, is called a {\em holomorphic map}.  
\end{definition}

The collection of holomorphic maps $F: U_1 \rightarrow U_2$, $(U_i,\mathcal O_{U_i})\in\mathcal M$, $i=1,2$, is denoted by $\mathcal O(U_1,U_2)$. If $F\in \mathcal O(U_1,U_2)$ has inverse $F^{-1}\in\mathcal O(U_2,U_1)$, then $F$ is called a {\em biholomorphism}.

%We will need the following consequence of Theorem 4.1 in \cite{BK8}:

%\begin{proposition}
%\label{mapstructlem}
%Let $V_{0,1}$, $V_{0,2} \subset \mathbb C^n$, $K_1$, $K_2 \subset \hat{G}_{\mathfrak a}$ be open. A map $F \in \mathcal O(V_{0,1} \times K_1,V_{0,2} \times K_2)$ admits presentation
%$F(z,\omega)=\bigl(f_\omega(z),h(\omega) \bigr)$, $(z,\omega) \in U_0 \times K_1,$
%where maps $f_\omega \in \mathcal O(V_{0,1},V_{0,2})$ depend continuously on $\omega \in K_1$ (in topology of uniform convergence on compact subsets), $h \in C(K_1,K_2)$.
%\end{proposition}

%\medskip

%\subsection{Coordinate charts}
Further, over each simply connected open subset $U_0 \subset X_0$ there exists a biholomorphic trivialization
$\psi=\psi_{U_0}:p^{-1}(U_0) \rightarrow U_0 \times G$ of covering $p:X \rightarrow X_0$ which is a morphism of fibre bundles with fibre $G$ (see~subsection \ref{sectbohr}). %We fix some system of biholomorphic trivializations of $X$.
%, and denote for a given subset $S \subset G$
%\begin{equation}
%\label{pi0}
%\Pi(U_0,S):=\psi^{-1}(U_0 \times S).
%\end{equation}
Then there exists a biholomorphic trivialization $\bar{\psi}=\bar{\psi}_{U_0}:\bar{p}^{-1}(U_0) \rightarrow U_0 \times \hat{G}_{\mathfrak a}$
of bundle $c_{\mathfrak a}X$ over $U_0$ which is a morphism of fibre bundles with fibre $\hat{G}_{\mathfrak a}$ such that 
the following diagram 
\begin{equation*}
\bfig
\node a1(0,0)[p^{-1}(U_0)]
\node a2(0,-500)[U_0 \times G]
\node b1(700,0)[\bar{p}^{-1}(U_0)]
\node b2(700,-500)[U_0 \times \hat{G}_{\mathfrak a}]
\arrow[a1`a2;\psi]
\arrow[a1`b1;\iota]
\arrow[b1`b2;\bar{\psi}]
\arrow[a2`b2;\Id \times j]
\efig
\end{equation*}
is commutative.

For a given subset $S \subset G$ we denote
\begin{equation}
\label{pi1}
\Pi(U_0,S):=\psi^{-1}(U_0 \times S)
\end{equation}
and identify $\Pi(U_0,S)$ with $U_0 \times S$ where appropriate (here $\Pi(U_0,G)=p^{-1}(U_0)$).

For a subset $K \subset \hat{G}_{\mathfrak a}$ we denote
\begin{equation}
\label{pi2}
\hat{\Pi}(U_0,K)~\bigl(=\hat{\Pi}_{\mathfrak a}(U_0,K)\bigr):=\bar{\psi}^{-1}(U_0 \times K).
\end{equation}
A pair of the form $(\hat{\Pi}(U_0,K),\bar{\psi})$ will be called a \textit{coordinate chart} for $c_{\mathfrak a}X$. Similarly, sometimes we identify $\hat{\Pi}(U_0,K)$ with $U_0 \times K$.
If $K \subset \hat{G}_{\mathfrak a}$ is open, then, by our definitions, $\bar{\psi}^*: \mathcal O(U_0 \times K)\rightarrow\mathcal O(\hat{\Pi}(U_0,K))$ is an isomorphism of (topological) algebras.

\subsection{Basis of topology on $c_{\mathfrak a}X$}
\label{topsect}

By $\mathfrak Q$ we denote the basis of topology of $\hat{G}_{\mathfrak a}$ consisting of sets of the form
\begin{equation}
\label{base1}
\left\{\eta \in \hat{G}_{\mathfrak a}: \max_{1 \leq i \leq m}|h_i(\eta)-h_i(\eta_0)|<\varepsilon\right\}
\end{equation}
for $\eta_0 \in \hat{G}_{\mathfrak a}$, $h_1,\dots,h_m \in C(\hat{G}_{\mathfrak a})$, and $\varepsilon>0$.

The fibrewise compactification $c_{\mathfrak a}X$ is a paracompact Hausdorff space (as a fibre bundle with a paracompact base and a compact fibre); thus, $c_{\mathfrak a}X$ is a normal space.

It is easy to see that the family
\begin{equation}
\label{base2}
\mathfrak B:=\{\hat{\Pi}(V_0,L) \subset c_{\mathfrak a}X: V_0 \text{ is open simply connected in } X_0 \text{ and } L \in \mathfrak Q\}
\end{equation}
forms a basis of topology of $c_{\mathfrak a}X$.

%We denote by $\mathfrak Q$ the basis of topology of $\hat{G}_{\mathfrak a}$ consisting of sets of the form
%\begin{equation}
%\label{base1}
%\left\{\eta \in \hat{G}_{\mathfrak a}: \max_{1 \leq i \leq m}|h_i(\eta)-h_i(\eta_0)|<\varepsilon\right\}
%\end{equation}
%for $\eta_0 \in \hat{G}_{\mathfrak a}$, $h_1,\dots,h_m \in C(\hat{G}_{\mathfrak a})$, and $\varepsilon>0$.
%
%The fibrewise compactification $c_{\mathfrak a}X$ is a paracompact Hausdorff space (as a fibre bundle with a paracompact base and a compact fibre); thus, $c_{\mathfrak a}X$ is a normal space.
%
%It is easy to see that the family
%\begin{equation}
%\label{base2}
%\mathfrak B:=\{\hat{\Pi}(V_0,L) \subset c_{\mathfrak a}X: V_0 \text{ is open simply connected in } X_0 \text{ and } L \in \mathfrak Q\}.
%\end{equation}
%forms a basis of topology of $c_{\mathfrak a}X$. 

\subsection{Complex submanifolds}
\label{submanifoldsect}
To formulate the required definition
note that for every $f \in \mathcal O(U_0 \times K)$, where $U_{0} \subset X_0$, $K \subset  \hat{G}_{\mathfrak a}$ are open, functions $f(\cdot,\omega)$, $\omega \in K$, are in $\mathcal O(U_0)$. Indeed, since $f \in \mathcal O(U_0 \times K)$, functions $U_0 \ni z \mapsto f(z,j(g))$ ($g \in j^{-1}(K)$) are holomorphic. Then since $j(j^{-1}(K))$ is dense in $K$ (see~Section \ref{mainsect}) and $f$ is bounded on each $S \Subset U_0 \times K$, by the Montel theorem $f(\cdot,\omega)\in\mathcal O(U_0)$ for all $\omega \in K$.

\begin{definition}
\label{defap}
A closed subset $Y \subset c_{\mathfrak a}X$ is called a {\em complex submanifold of codimension} $k$ if for every $y \in Y$ there exist its neighbourhood of the form $U=\hat{\Pi}(U_0,K) \subset c_{\mathfrak a}X$, where $U_{0} \subset X_0$ is open and simply connected, $K \subset  \hat{G}_{\mathfrak a}$ is open, and functions $h_1,\dots,h_k \in \mathcal O(U)$ such that

(1) $Y \cap U=\{x \in U: h_1(x)=\dots=h_k(x)=0\}$;

(2) Rank of the map $z \mapsto \bigl(h_1(z,\omega),\dots,h_k(z,\omega)\bigr)$ is $k$ at each point $x=(z,\omega) \in Y \cap U$.
\end{definition}

The next result on the local structure of a complex submanifold of $c_{\mathfrak a}X$ follows straightforwardly from the inverse function theorem with continuous dependence on parameter (Theorem \ref{ift}).

% if one takes into account that every point in $c_{\mathfrak a}X$ has a neighbourhood of the form $\hat{\Pi}(V_0,L)$, where $V_0 \subset X_0$, $L \subset \hat{G}_{\mathfrak a}$ are open, that is biholomorphic to $V_0 \times L$ (cf.~subsection \ref{charts}). 

\begin{proposition}
\label{localstructap}
Let $Y \subset c_{\mathfrak a}X$ be a complex submanifold.
For every $y_0 \in Y$ there exist an open neighbourhood $V \subset c_{\mathfrak a}X$ of $y$, open subsets $V_0\subset X_0$ and $K\subset \hat{G}_{\mathfrak a}$, a (closed) complex submanifold $Z_0$ of $V_0$ (of the same codimension as $Y$), and a biholomorphic map $\Phi \in \mathcal O(V_0 \times K,V)$
such that $\Phi\bigl(V_0\times (K\cap j(G))\bigr)=V\cap\iota(X)$ and $\Phi^{-1} (V \cap Y)=Z_0 \times K.$
\end{proposition}

(See the proof in subsection \ref{iftsect}.)

We apply this proposition to establish the following important fact (see subsection \ref{constrsect} for the definition of a coherent sheaf on $c_{\mathfrak a}X$).

\begin{proposition}
\label{idcohlem}
The ideal sheaf $I_Y$ of germs of holomorphic functions vanishing on a complex submanifold $Y \subset c_{\mathfrak a}X$ is coherent.
\end{proposition}

Now, we list other properties of complex submanifolds of $c_{\mathfrak a}X$.

\begin{proposition}
Any complex submanifold $Y$ of $c_{\mathfrak a}X$ has the following properties:

(\textit{i}) $\iota^{-1}\bigl(Y\bigr) \subset X$ is a complex submanifold of $X$ of codimension $k$.

(\textit{ii}) $Y \cap \iota(X)$ is dense in $Y$. 
\end{proposition}

Assertion (\textit{i}) is immediate from the definition while assertion (\textit{ii}) follows from the fact that  $\iota(X)$ is dense in $X$ combined with Proposition \ref{localstructap}.

\begin{proposition}
\label{closprop}

If $Z \subset X$ is a complex $\mathfrak a$-submanifold (see~Definition \ref{manifolddef}), then 
the closure of $\iota(Z)$ in $c_{\mathfrak a}X$ is a complex submanifold of $c_{\mathfrak a}X$.

Suppose that $\mathfrak a$ is self-adjoint. If $Y$ is a complex submanifold of $c_{\mathfrak a}X$, then $\iota^{-1}(Y) \subset X$ is a complex $\mathfrak a$-submanifold.
\end{proposition}

\begin{definition}
\label{defa}
A function $f\in C(Y)$ is called holomorphic if  $\iota^* f \in \mathcal O\bigl(\iota^{-1}(Y)\bigr)$. 
The algebra of holomorphic functions on $Y$ is denoted by $\mathcal O(Y)$.
\end{definition}

Similarly, we define holomorphic functions $\mathcal O(U)$ on an open subset $U \subset Y$ as those continuous functions whose pullbacks by $\iota$ are holomorphic in the usual sense. 

%We expand category $\mathcal M$ of subsection \ref{holfuncsect} by adding Holomorphic maps between ringed spaces of the form $(U,\mathcal O_{U})$, where $U$ is an open subset of a complex submanifold $Y\subset c_{\mathfrak a}X$ are defined similarly to.

\begin{proposition}
\label{equivprop2}
Suppose that $\mathfrak a$ is self-adjoint,
$Y$ is a complex submanifold of $c_{\mathfrak a}X$. We set $Z:=\iota^{-1}(Y)$. 
Then $\mathcal O_{\mathfrak a}(Z) \cong \iota^*\mathcal O(Y) $ (so every function in $f \in \mathcal O_{\mathfrak a}(Z)$, see~Definition \ref{holdef2}, admits a unique extension to a function $\hat{f} \in \mathcal O(Y)$ such that $f=\iota^*\hat{f}$).
\end{proposition}

\subsection{Cartan theorems A and B on complex submanifolds}
\label{cartansubmsect}

The notion of a coherent sheaf on $c_{\mathfrak a}X$ (see subsection \ref{constrsect})
extends to analytic sheaves on a complex submanifold of $c_{\mathfrak a}X$. It turns out that if $X_0$ is a Stein manifold, then for such coherent sheaves we have analogues of Cartan theorems A and B (Theorems \ref{thmA_} and \ref{thmB_} below).

More precisely, we have the structure sheaf $\mathcal O_Y$ of germs of holomorphic functions on a complex submanifold $Y \subset c_{\mathfrak a}X$ (see~Definition \ref{defa}). 
A {\em coherent sheaf} $\mathcal A$ on $Y$ is a sheaf of modules over $\mathcal O_Y$ such that every point in $Y$ has a neighbourhood $V\subset Y$ over which, for any $N \geq 1$, there is a free resolution of $\mathcal A$ of length $N$, i.e., an exact sequence of sheaves of modules of the form
\begin{equation*}
%\label{coh0_}
\mathcal O_Y^{m_{N}}|_V \overset{\varphi_{N-1}}{\to} \dots \overset{\varphi_2}{\to} \mathcal O_Y^{m_{2}}|_V \overset{\varphi_1}{\to} \mathcal O_Y^{m_{1}}|_V \overset{\varphi_0}{\to} \mathcal A|_V \to 0
\end{equation*}
(here $\varphi_i$, $0 \leq i \leq N-1$, are homomorphisms of sheaves of modules).

\medskip

Given a sheaf of modules $\mathcal A$ over $\mathcal O_Y$, we define a sheaf $\tilde{\mathcal A}$ on $c_{\mathfrak a}X$ (called the \textit{trivial extension} of $\mathcal A$) by 
the formulas $$\tilde{\mathcal A}|_{c_{\mathfrak a}X \setminus Y}:=0, \quad \tilde{\mathcal A}|_Y:=\mathcal A.$$

\begin{theorem}
\label{cartansubm}
If $\mathcal A$ is a coherent sheaf on a complex submanifold $Y \subset c_{\mathfrak a}X$, then $\tilde{\mathcal A}$ is a coherent sheaf on $c_{\mathfrak a}X$.
\end{theorem}

It is immediate that $H^k(Y,\mathcal A) \cong H^k(c_{\mathfrak a}X,\tilde{\mathcal A})$. Therefore, Theorems \ref{thmA}, \ref{thmB} and Theorem \ref{cartansubm} imply the following analogues of Cartan theorems A and B.

\medskip

Let $\mathcal A$ be a coherent sheaf on a complex submanifold $Y \subset c_{\mathfrak a}X$ with $X_0$ Stein. 

\begin{theorem}
\label{thmA_}
Each stalk $\phantom{}_{x}\mathcal A$ ($x \in Y$) is generated by  sections $\Gamma(Y,\mathcal A)$ as an $\phantom{}_{x}\mathcal O_Y$-module (``Cartan type theorem A''). 
\end{theorem}

\begin{theorem}
\label{thmB_}
Sheaf cohomology groups $H^i(Y,\mathcal A)=0$ for all $i \geq 1$ (``Cartan type theorem B'').
\end{theorem}

%(1) The definition of coherence (cf.~Section \ref{constrsect}) extends naturally to the sheaves of modules on a complex submanifold $Y$ of $c_{\mathfrak a}X$ (if $\mathfrak a$ is self-adjoint, then the complex submanifolds of $c_{\mathfrak a}X$ are identifies with the complex $\mathfrak a$-submanifolds of $X$ by map $\iota$, see Section \ref{submanifoldsect} below), with the analogs of Cartan theorems A and B being valid for such sheaves (cf.~Section \ref{cartansubmsect} and Theorems \ref{thmA_}, \ref{thmB_}).
%In particular, since the higher cohomology groups of coherent sheaves on complex submanifolds of $c_{\mathfrak a}X$ vanish, it is natural to call such submanifolds of $c_{\mathfrak a}X$ the \textit{Stein} (closed) submanifolds.

Let $Y$ be either $c_{\mathfrak a}X$ or a complex submanifold of $c_{\mathfrak a}X$.
The definition of coherence on $Y$ extends directly to open subsets of $Y$. It is natural to call such subset $W\subset Y$ a \textit{Stein manifold} if all higher cohomology groups of coherent sheaves on $W$ vanish (i.e.,~Cartan type theorem B holds on $W$). One can ask about characterization of Stein open submanifolds $W\subset Y$ (e.g.,~in terms of appropriately defined plurisubharmonic exhaustion $\mathfrak a$-functions on $W$).

\subsection{Dolbeault-type complex}

\label{derhamsect}

%\subsection{Differential forms}
%\label{dolbeault}

In this part we describe a Dolbeault-type complex and analogues of Dolbeault isomorphisms used
in the proof of Proposition \ref{zeroapcor}.

Let $Y\subset c_{\mathfrak a}X$ be a complex submanifold. We define the {\em holomorphic tangent bundle} $TY$ of $Y$ as a holomorphic bundle on $Y$ whose pullback by $\iota$ to $\iota^{-1}(Y)$ coincides with  the holomorphic tangent bundle of the complex submanifold $\iota^{-1}(Y)\subset X$ (see the proof of Theorem \ref{tubularnbd} in Section \ref{countthmsect} for existence and uniqueness of $TY$). 

The definition of the antiholomorphic tangent bundle $\overline{TY}$ of $Y$ is analogous. 

We define the complexified 
tangent bundle of $Y$ as the Whitney sum
\begin{equation}
\label{tdecomp}
T^{\mathbb C}Y:=TY \oplus \overline{TY}.
\end{equation}

By $\Lambda_c^m(Y):=\Gamma(Y,\wedge^m(T^{\mathbb C}Y)^*)$ we denote the space of continuous sections of the vector bundle $\wedge^m(T^{\mathbb C}Y)^*$ ($0 \leq m \leq n:=\dim_{\mathbb C}X_0$), where $(T^{\mathbb C}Y)^*$ is the dual bundle of $T^{\mathbb C}Y$. Elements of $\Lambda_c^m(Y)$ will be called continuous \textit{$m$-forms}.

\medskip

By Proposition \ref{localstructap}  for every point $x \in Y$ there exist a neighbourhood $U \subset c_{\mathfrak a}X$, a biholomorphism $\varphi:U \rightarrow U_0 \times K$, where $U_0 \subset \mathbb C^n$, $K \subset \hat{G}_{\mathfrak a}$, $K \in \mathfrak Q$ (see~(\ref{base1})) are open, and a complex submanifold $Y_0 \subset U_0$ such that $\varphi(Y \cap U)=Y_0 \times K$ and $\varphi(Y\cap U\cap \iota(X))=Y_0\times( K\cap j(G))$.
By $\wedge^m T^{\mathbb C}(Y_0\times K)^*$ we denote the pullback to $Y_0\times K$ of the bundle  $\wedge^m(T^{\mathbb C}Y_0)^*$ under the natural projection $Y_0\times K\rightarrow Y_0$. 
Since $\varphi^{-1}|_{Y_0\times( K\cap j(G))}: Y_0\times( K\cap j(G))\rightarrow Y\cap U\cap \iota(X)$ is a biholomorphism of usual complex manifolds, 
\[
\bigl(\varphi^{-1}|_{Y_0\times( K\cap j(G))}\bigr)^*\bigl(\wedge^m(T^{\mathbb C}Y)^*\bigr)=\wedge^mT^{\mathbb C}(Y_0\times K)^*|_{Y_0\times( K\cap j(G))}.
\] 
Since $Y_0\times( K\cap j(G))$ is dense in $Y_0\times K$, the latter bundle is dense in the bundle $\wedge^mT^{\mathbb C}(Y_0\times K)^*$. Thus the above identity and the continuity of $\varphi^{-1}$ imply that  $(\varphi^{-1})^*\bigl(\wedge^m(T^{\mathbb C}Y)^*\bigr)=\wedge^mT^{\mathbb C}(Y_0\times K)^*$. In particular, $(\varphi^{-1})^*$  maps $\Lambda_c^m(Y \cap U)$ to $\Lambda_c^m(Y_0 \times K)$, the space of continuous sections of $\wedge^mT^{\mathbb C}(Y_0\times K)^*$. Clearly,
\begin{equation}
\label{formdecomp}
 \Lambda_c^m(Y_0 \times K)=C(Y_0 \times K) \otimes \Lambda_c^m(Y_0),
\end{equation}
where $\Lambda_c^m (Y_0)$ is the space of continuous $m$-forms on $Y_0$ and $C(Y_0 \times K)$ is the space of continuous functions on $Y_0 \times K$ endowed with the Fr\'{e}chet topology of uniform convergence on compact subsets of $Y_0 \times K$.% (cf.~Lemma 5.3.4(1) in \cite{BK8}).

By $\Lambda^m(Y) \subset \Lambda_c^m(Y)$ we denote the subspace of $C^\infty$ $m$-forms, that is, forms $\omega$ such that for each ``coordinate map'' $\varphi:U \rightarrow U_0 \times K$, 
\begin{equation}\label{cinfty}
(\varphi^{-1})^* \omega|_{Y \cap U} \in \Lambda^m(Y_0 \times K):=C^\infty(Y_0 \times K) \otimes \Lambda^m(Y_0),
\end{equation}
where $\Lambda^m(Y_0)$ is the space of $C^\infty$ $m$-forms on $Y_0$ and $C^\infty(Y_0 \times K) \subset C(Y_0 \times K)$ is the subspace of continuous functions that are $C^\infty$ when viewed as functions on $Y_0$ taking values in the Fr\'{e}chet space $C(K)$. 

We denote $C^\infty(Y):=\Lambda^0(Y)$.

\begin{lemma}
\label{doesnotdependlem}
$\Lambda^m(Y)$ is correctly defined by (local) conditions \eqref{cinfty}.
\end{lemma}

Let $a \in C^\infty(Y_0 \times K)$. We define the differential $da \in \Lambda^1(Y_0 \times K)$ as follows. 
(To simplify notation, we may assume without loss of generality that $Y_0$ is an
open subset of $\mathbb C^{n-k}$.)

Define
$
da:=\sum_{i=1}^{n-k} \frac{\partial a}{\partial z_i} dz_i,
$
where $\frac{\partial a}{\partial z_i} \in C^\infty(Y_0 \times K)$ is the derivative of the Fr\'{e}chet-valued map
$z\mapsto a(z,\cdot)\in C(K)$, $z=(z_1,\dots,z_{n-k})\in Y_0$, with respect to $z_i$.

We have the operator of differentiation
$
d:\Lambda^m(Y_0 \times K) \rightarrow \Lambda^{m+1}(Y_0 \times K)$ defined by the formula 
\begin{equation}
\label{dop}
d\left(\sum_{i=}^l a_i  \omega_i \right):=\sum_{i=1}^l a_i  d\omega_i+\sum_{i=1}^l da_i \wedge \omega_i, \quad a_i \in C^\infty(Y_0 \times K), \quad \omega_i \in \Lambda^m(Y_0).
\end{equation}
Now, we define the operator of differentiation $d: \Lambda^m(Y)\rightarrow \Lambda^{m+1}(Y)$:

For each coordinate map $\varphi:U \rightarrow U_0 \times K$ and $\omega \in \Lambda^m(Y)$ the form $d\omega\in \Lambda^{m+1}(Y)$ satisfies  
\begin{equation}\label{differ}
(\varphi^{-1})^*d\omega=d \bigl((\varphi^{-1})^*\omega\bigr),
\end{equation}
where the right-hand side is defined by (\ref{dop}). 

Existence of $d\omega$ satisfying local conditions \eqref{differ} follows from the facts that due to these conditions $\iota^*\circ d|_{\Lambda^m(Y)|_U}=d\circ\iota^*|_{\Lambda^m(Y)|_U}$, where the $d$ on the right denotes the standard differentiation on the space of differential forms defined on the complex submanifold $\iota^{-1}(Y) \subset X$, and that $\iota(\iota^{-1}(Y))$ is dense in $Y$ (see Proposition \ref{localstructap}). By the same reason we have $d\circ d=0$.

Further, (\ref{tdecomp}) induces decomposition 
$(T^{\mathbb C}Y)^*=TY^* \oplus \overline{TY}^*$ and, hence,
\begin{equation}
\label{tdecomp2}
\Lambda^m(Y)=\oplus_{p+k=m}\, \Lambda^{p,k}(Y), 
\end{equation}
where
$$
\Lambda^{p,k}(Y):=\Gamma(Y,\wedge^p TY^* \otimes \wedge^q \overline{TY}^*)\cap\Lambda^m(Y).
$$
Since the pullback by $\iota$ of $TY^*$ coincides with the holomorphic cotangent bundle of complex submanifold $\iota^{-1}(Y) \subset X$, the pullback by $\iota$ of decomposition (\ref{tdecomp2}) agrees with the usual type decomposition of differential forms on $\iota^{-1}(Y)$.
%(alternatively, we can deduce this decomposition from formula (\ref{tdecomp})).

Using the natural projections $\pi^{p,k}:\Lambda^m(Y) \rightarrow \Lambda^{p,k}(Y)$ ($m=p+k$), we define
$$
\partial:=\pi^{p+1,k} \circ d, \quad \bar{\partial}:=\pi^{p,k+1} \circ d.
$$
Since pullbacks by $\iota$ of these operators coincide with their usual counterparts on the complex submanifold $\iota^{-1}(Y) \subset X$ and the image by $\iota$ of the latter is dense in $Y$, we have $\partial\circ\partial=0$, $\bar{\partial}\circ\bar{\partial}=0$ and $d=\partial+\bar{\partial}$.

The above definitions and notation transfer naturally to open subsets of $Y$. In particular, we can define the sheaf $\Lambda^{p,k}$ of germs of $C^\infty$ $(p,q)$-forms on $Y$.
\begin{lemma}
\label{finesheaflem}
For any open cover $\mathcal U$ of $Y$ there is a subordinate $C^\infty$ partition of unity.
\end{lemma}
This lemma implies that $\Lambda^{p,k}$ is a fine sheaf, and therefore 
cohomology groups $H^r(Y,\Lambda^{p,k})=0$ for all $r \geq 1$ (see, e.g.,~\cite{GR}).

\medskip

Let $Z^{p,k} \subset \Lambda^{p,k}$ denote the subsheaf of germs of $\bar{\partial}$-closed forms. We have the following analogue of $\bar{\partial}$-Poincar\'{e} lemma for sections of $Z^{p,k}$.

\begin{proposition}
\label{poincarelem}
Let $Y \subset c_{\mathfrak a}X$ be a complex submanifold. For every point $x \in Y$ there are neighbourhoods $W, V \subset Y$, $W \Subset V$, of $x$ such that restriction to $W$ of any $\bar{\partial}$-closed form in $\Lambda^{p,k+1}(V)$ is $\bar{\partial}$-exact.
\end{proposition}

%\textbf{1.~}We introduce the Dolbeault complex first.

Let $Z^{p,k}(Y) \subset \Lambda^{p,k}(Y)$ denote the subspace of $\bar{\partial}$-closed forms. We define the {\em Dolbeault cohomology groups} of $Y$ as 
$$
H^{p,k}(Y):=Z^{p,k}(Y)/\bar{\partial} \Lambda^{p,k-1}(Y), \quad p \geq 0, \quad k \geq 1, $$ 
$$
H^{p,0}(Y):=Z^{p,0}(Y).
$$

We set $\Omega^p:=Z^{p,0}$. Then $\Omega^p$ is the sheaf of germs of holomorphic $p$-forms on $Y$, i.e., holomorphic sections of the bundle $\wedge^p TY^*$. (Note that $\iota^*\Omega^p$ is the sheaf of germs of usual holomorphic $p$-forms on the complex submanifold $\iota^{-1}(Y)\subset c_{\mathfrak a}X$.)

Since $\Lambda^{p,k}$ is a fine sheaf, from Proposition \ref{poincarelem} and a standard result in \cite[Ch.B \S 1.3]{GrRe} we obtain

\begin{corollary}[Dolbeault-type isomorphism]
\label{dolcor1}
$\forall p,k \geq 0$, $H^{p,k}(Y) \cong H^k(Y,\Omega^p)$.
\end{corollary}

%Indeed, by Proposition \ref{poincarelem} we have the exact sequences of sheaves
%$$
%0 \rightarrow Z^{p,k} \rightarrow \Lambda^{p,k} \overset{\bar{\partial}}{\rightarrow} Z^{p,k+1} \rightarrow 0, \quad k \geq 0.
%$$
%Furthermore, it is easy to show using the partitions of unity ($C^\infty$-smooth in the variable in $X_0$ and continuous in the variable in $\hat{G}_{\mathfrak a}$) that the sheaf cohomology groups $H^r(Y,\Lambda^{p,k})=0$, $r \geq 1$.
%Now, the corresponding long exact sequences of cohomology groups yield
%\begin{multline*}
%H^k(Y,\Omega^p) \cong H^{k-1}(Y,Z^{p,1}) \cong H^{k-2}(Y,Z^{p,2}) \cong \dots \\ H^1(Y,Z^{p,k-1}) \cong H^0(Y,Z^{p,k})/\bar{\partial}H^0(Y,\Lambda^{p,k-1})=H^{p,k}(Y),
%\end{multline*}
%as needed.

Since $\Omega^p$ is the sheaf of germs of sections of a holomorphic vector bundle on $Y$, it is coherent (see subsection \ref{cartansubmsect}). Then the previous corollary and Theorem \ref{thmB_} imply

\begin{corollary}
\label{dolcor2}
Suppose that $X_0$ is a Stein manifold, $Y \subset c_{\mathfrak a}X$ is a complex submanifold. Then
$$
H^{p,k}(Y)=0 \quad \text{for all}\quad p \geq 0, \quad k \geq 1
$$
(i.e., any $\bar{\partial}$-closed form in $\Lambda^{p,k}(Y)$ is $\bar{\partial}$-exact).
\end{corollary}

Similarly one can define the de Rham cohomology groups of $Y$ and obtain an analogue of the classical de Rham isomorphism (see the proof of Proposition \ref{zeroapcor}).

\subsection{Characterization of $c_{\mathfrak a}X$ as the maximal ideal space of $\mathcal O_{\mathfrak a}(X)$}
\label{sectionmaxid}

Now we relate the fibrewise compactification $c_{\mathfrak a}X$ of covering $p: X\rightarrow X_0$ to the maximal ideal space $M_X$ of algebra $\mathcal O_{\mathfrak a}(X)$, i.e., the space of non-zero continuous homomorphisms $\mathcal O_{\mathfrak a}(X) \rightarrow \mathbb C$ endowed with weak* topology (of $\mathcal O_{\mathfrak a}(X)^*$).

\begin{theorem}
\label{corona}
Suppose that algebra $\mathfrak a$ is self-adjoint, and $X_0$ is a Stein manifold. Then $M_{X}$ is homeomorphic to $c_{\mathfrak a}X$. 
\end{theorem}

Since $\iota(X)$ is dense in $c_{\mathfrak a}X$, 
and the natural mapping of $X$ into $M_{X}$, sending each point of $X$ to its point evaluation homomorphism, coincides with $\iota$ under the homeomorphism of Theorem \ref{corona}, we obtain the following corona-type theorem.

\begin{corollary}
Let $\mathfrak a$ be self-adjoint, $X_0$ be a Stein manifold.  
Then $\iota(X)$ is dense in $M_{X}$.  
\end{corollary}

%
%\begin{corollary}
%Let $\mathfrak a$ be self-adjoint, $X_0$ be Stein a manifold, $A$ be an endomorphism of algebra $\mathcal O_{\mathfrak a}(X)$. Then there exists a map $a \in \mathcal O(X,X)$ such that $Af=a^*f$ for all $f \in \mathcal O_{\mathfrak a}(X)$.
%\end{corollary}

%\section{Proofs}

%\addtocontents{toc}{\protect\setcounter{tocdepth}{0}}

%\setcounter {section} {0}

%\def\thesection{6.\arabic{section}}

%
%First, we prove the results that do not require Theorem \ref{cechtriv} (``Cartan-type theorems A and B''): Theorem \ref{equivapthm} (on presentation of Bohr's holomorphic almost periodic functions as holomorphic $\mathfrak a$-functions for an appropriate choice of $\mathfrak a$), Propositions \ref{iso1prop} and \ref{basicpropthm} (on equivalent presentations of holomorphic $\mathfrak a$-functions), Theorems \ref{hypthm} (on uniqueness sets) and \ref{hartogsthm} (Hartogs-type theorem), Theorem \ref{approxthm} (on approximation). 
%
%We prove the rest of results after: Theorem \ref{extapthm} (on extension from complex $\mathfrak a$-submanifolds), Theorems \ref{zeroapthm} and \ref{chernthm} (on $\mathfrak a$-divisors), Theorem \ref{corona} (on characterization of $c_{\mathfrak a}X$ as the maximal ideal space of $\mathcal O_{\mathfrak a}(X)$),  Proposition \ref{frechetprop} and Theorem \ref{cechtriv} (on coherent sheaves on $c_{\mathfrak a}X$).

\section{Proofs: Preliminaries}

\SkipTocEntry\subsection{\v{C}ech cohomology}
\label{proofnotation}

For a topological space $X$ and a sheaf of abelian groups $\mathcal S$ on $X$ by $\Gamma(X,\mathcal S)$ we denote the abelian group of continuous sections of $\mathcal S$ over $X$.

Let $\mathcal U$ be an open cover of $X$.
By $C^i(\mathcal U,\mathcal S)$ we denote the space of \v{C}ech $i$-cochains with values in $\mathcal S$, by $\delta:C^i(\mathcal U,\mathcal S) \rightarrow C^{i+1}(\mathcal U,\mathcal R)$ the \v{C}ech coboundary operator, by
$Z^i(\mathcal U,\mathcal S):=\{\sigma \in C^i(\mathcal U,\mathcal S): \delta \sigma=0\}$
the space of $i$-cocycles, and by $B^i(\mathcal U,\mathcal S):=\{\sigma \in  Z^i(\mathcal U,\mathcal S): \sigma=\delta( \eta), \eta \in \mathcal C^{i-1}(\mathcal U,\mathcal S)\}$
the space of $i$-coboundaries (see, e.g., \cite{GR} for details).
The \v{C}ech cohomology groups $H^i(\mathcal U,S)$, $i \geq 0$, are defined by $$H^i(\mathcal U,S):=Z^i(\mathcal U,\mathcal S)/B^i(\mathcal U,\mathcal S), \quad i \geq 1,$$ and $H^0(\mathcal U,\mathcal S):=\Gamma(\mathcal U,\mathcal S)$.

\SkipTocEntry\subsection{$\bar{\partial}$-equation}
\label{bformsect}

Let $B$ be a complex Banach space, $D_0 \subset X_0$ be a strictly pseudoconvex domain. We fix a system of local coordinates on $D_0$. Let $\{W_{0,i}\}_{i \geq 1}$ be the cover of $D_0$ by the coordinate patches.
By $\Lambda_b^{(0,q)}(D_0,B)$, $q \geq 0$, we denote the space of bounded continuous $B$-valued $(0,q)$-forms on $D_0$ endowed with norm
\begin{equation}
\label{approxsup}
\|\omega\|_{D_0}=\|\omega\|_{D_0,B}^{(0,q)}:=\sup_{x \in W_{0,i}, i \geq 1, \alpha}\|\omega_{\alpha,i} (x)\|_B, 
\end{equation}
where $\omega_{\alpha,i}$ ($\alpha$ is a multiindex) are coefficients of forms $\omega|_{W_{0,i}} \in  \Lambda_b^{(0,q)}(W_{0,i},B)$ in local coordinates on $W_{0,i}$. 

The next lemma follows easily from results in \cite{HL} (proved for $B=\mathbb C$) because all integral representations and estimates are preserved when passing to the case of Banach-valued forms.

\begin{lemma}
\label{hl1}
There exists a bounded linear operator $$R_{D_0,B}\in \mathcal L\left(\Lambda_b^{(0,q)}(D_0,B), \Lambda_b^{(0,q-1)}(D_0,B)\right), \quad q \geq 1,$$ such that if $\omega \in \Lambda_b^{(0,q)}(D_0,B)$ is $C^\infty$ and $\bar{\partial}$-closed on $D_0$, then $\bar{\partial} R_{D_0,B}\,\omega=\omega$.
\end{lemma}

\SkipTocEntry\subsection{Inverse function theorem with continuous dependence on parameter}
\label{iftsect}

\begin{theorem}
\label{ift}

Let $K$ be a topological space, $B_1, B_2 \Subset \mathbb C^n$ be open balls centered at the origin. Fix a point $(x_0,\eta_0) \in B_1 \times K$. Suppose that a continuous map $G:B_1 \times K \rightarrow B_2$ satisfies
\begin{itemize}
\item[(1)] $G(\cdot,\eta):B_1 \rightarrow B_2$ is holomorphic for every $\eta \in K$, 

\item[(2)] Jacobian matrix $D_xG(x_0,\eta_0)$ is non-degenerate. 
\end{itemize}
Then there exist an open subset $W \subset B_2 \times K$ and a continuous map $H:W \rightarrow B_1$ such that 
\begin{itemize}
\item[(a)]
$\bigl(G(x_0,\eta_0),\eta_0\bigr) \in W$, 
\item[(b)]
$H(\cdot,\eta)$ is holomorphic on $W\cap \bigl(B_2\times\{\eta\}\bigr)$ for all $\eta$ for which this set is non-empty, 
\item[(c)]
$G\bigl(H(z,\eta),\eta\bigr)=z$ for all $(z,\eta)\in W$.
\end{itemize}

\end{theorem}

Theorem \ref{ift} follows easily from the contraction principle with continuous dependence on parameter, see e.g.,~\cite[Ch.XVI]{KA}.

\medskip

As an application of Theorem \ref{ift} we prove Proposition \ref{localstructap} on the local structure of complex submanifolds of $c_{\mathfrak a}X$. 

\begin{proof}[Proof of Proposition \ref{localstructap}] Let $Y \subset c_{\mathfrak a}X$ be a complex submanifold and $y_0 \in Y$.
In notation of Definition \ref{defap}, there exists a neighbourhood $U:=\hat{\Pi}(U_0,L) \subset c_{\mathfrak a}X$ of $y_0$, where 
$U_0 \subset X_0$ is a simply connected coordinate chart and $L \subset \hat{G}_{\mathfrak a}$ is open, such that
$Y \cap U=\{y \in U: h_1(y)=\dots=h_k(y)=0\}$
with $h_i \in \mathcal O(U)$ ($1 \leq i \leq k$) satisfying non-degeneracy condition (2) of the definition.

Since sets $U$ and $U_0 \times L$ are biholomorphic (see~subsection \ref{charts}), in what follows we identify them. Next,
since functions $h_i$ satisfy the non-degeneracy condition of Definition \ref{defap}, we may choose coordinates $x_1,\dots, x_n$ on $U_0$ so that the Jacobian matrix $D_xG(x_0,\eta_0)$, $y:=(x_0,\eta_0)$, of the map
\begin{equation*}
G(x,\eta):= \bigl(h_{1}(x,\eta),\dots,h_{k}(x,\eta),x_{k+1},\dots,x_n\bigr), \quad x:=(x_1,\dots,x_n),\ (x,\eta) \in U_0 \times L,
\end{equation*}
is non-degenerate. Also, we may assume without loss of generality that $U_0=B_1$ and $G(B_1,\eta) \subset B_2$ for all $\eta \in L$, where $B_i \Subset \mathbb C^n$, $i=1,2$, are open balls centered at the origin. Hence, we can apply Theorem \ref{ift}. In its notation, shrinking $W$, if necessary, we may assume that $W=V_0 \times K$ for some open $V_0 \subset B_2$, $K \subset L$ which we take as the required sets in the formulation of Proposition \ref{localstructap}. We also take $\Phi(z,\eta):=\bigl(H(z,\eta),\eta\bigr)$, $(z,\eta) \in V_0 \times K$, and $V:=\Phi(V_0 \times K) \subset U$. By definition, $\Phi \in \mathcal O(V_0 \times K,V)$ (see~subsection \ref{charts}).
Further, since $\bigl( G \circ \Phi(z),\eta \bigr)=(z,\eta)$ for all $(z,\eta) \in V_0 \times K$, 
$$(h_i \circ  \Phi)(z_1,\dots,z_n,\eta)=z_i, \quad (z,\eta) \in V_0 \times K, \quad z=(z_1,\dots,z_n), \quad 1 \leq i \leq k.$$
Therefore, $\Phi^{-1}(V \cap Y)=Z_0 \times K$, where $Z_0:=\{(0,\dots,0,z_{k+1},\dots,z_n) \in V_0: (z_1,\dots,z_n) \in V_0\}$ is a complex submanifold of codimension $k$. 

By our construction we also have $\Phi\bigl(V_0\times (K\cap j(G))\bigr)=V\cap\iota(X)$.

The proof of the proposition is complete.
\end{proof}

\section{Proof of Theorem \ref{equivapthm}}

Fix some $\Omega' \Subset \Omega$ and denote $T':=\mathbb R^n+i\bar{\Omega}' \subset \mathbb C^n$. We endow $T'$ with the Euclidean metric induced from $\mathbb C^n$.

We will need the following definition.

\begin{definition}
A function $f\in C(T')$ is called \textit{continuous almost periodic} if the family of its translates $\{T' \ni z \mapsto f(z+t)\}_{t \in \mathbb R^n}$ is relatively compact in $C_b(\bar{T}')$ (the space of bounded continuous functions on $T'$ endowed with $\sup$-norm).
\end{definition}

\begin{proposition}[see, e.g.,~\cite{Bes}]
\label{besprop}
Any continuous almost periodic function on $T'$ is bounded and uniformly continuous.
\end{proposition}

By $APC(T')$ we denote the Banach algebra of continuous almost periodic functions on $T'$ endowed with $\sup$-norm.

We set 
\begin{equation*}
p(z):=\bigl(e^{i z_1}, \dots, e^{i z_n}\bigr), \quad z=(z_1,\dots,z_n) \in T',\quad \text{and}\quad T'_0:=p(T').
\end{equation*}
Then Banach algebra $C_{AP}(T')$, $AP:=AP(\mathbb Z^n)$, $\mathbb Z^n \cong p^{-1}(x_0)$ ($x_0 \in X_0$), associated to covering $p: T'\rightarrow T'_0$ and endowed with $\sup$-norm, is well defined (see~Introduction).

To prove the theorem it suffices to show that $APC(T')=C_{AP}(T')$. (Because the space of \textit{holomorphic almost periodic functions} on $T$ consists of all functions in $\mathcal O(T)$ whose restrictions to each tube subdomain $T' \subset T$ are in $APC(T')$, and $\mathcal O_{AP}(T):=\mathcal O(T) \cap \{f \in C(T): f|_{T'} \in C_{AP}(T') \text{ for each } T' \subset T \}$.)

\medskip

First, let $f \in APC(T')$, i.e., for any sequence $\{t_k\}\subset \mathbb R^n$ there exists a subsequence of $\{T' \ni z \mapsto f(z+t_{k_l})\}$ that converges uniformly on $T'$. 
%Then $f$ is uniformly continuous on $T^s$
In particular, it follows that for each fixed $z_0 \in T'$ and a sequence $\{d_k\} \subset \mathbb Z^n$ the family of translates $\{\mathbb Z^n \ni g \mapsto  f(z_0+g+d_k)\}$ has a convergent subsequence which implies (since $C_{AP}(T')$ is a metric space) that it is relatively compact in the topology of uniform convergence on $\mathbb Z^n$. This means that the function $\mathbb Z^n \ni g \mapsto  f(z_0+g)$ belongs to $AP(\mathbb Z^n)$. Also, by Proposition \ref{besprop} function $f$ is bounded and uniformly continuous on $T'$. Hence, by definition, $f \in C_{AP}(T')$.

\medskip

Now, let $f \in C_{AP}(T')$. We must show that $f \in APC(T')$. To this end we fix some sequence $\{t_k\} \subset \mathbb R^n$.
Let $\mu:\mathbb R^n \rightarrow \mathbb R^n/\mathbb Z^n$ be the natural projection. Since $\mathbb R^n/\mathbb Z^n$ is compact, $\{\mu(t_k)\}$ has a convergent subsequence. We may assume without loss of generality that $\{\mu(t_k)\}$ itself converges and has limit $0$. Hence, there exists a sequence $\{d_k\} \subset \mathbb Z^n$ such that $|t_k-d_k| \rightarrow 0$ as $k \rightarrow \infty$. Since $f$ is uniformly continuous on $T'$, functions $$h_k(z):=|f(z+t_k)-f(z+d_k)| \rightarrow 0 \quad \text{ uniformly on }T' \text{ as } k \rightarrow \infty.$$ Hence, it suffices to show that sequence $\{T' \ni z \mapsto f(z+d_k)\}$ has a convergent subsequence.

Let $C:=\{z=(z_1,\dots,z_n) \in T': 0 \leq \Real(z_i) \leq 1, 1 \leq i \leq n \}$. Since $f \in C_{AP}(T')$, for each fixed $w \in C$ the family of translates $\{\mathbb Z^n \ni g \mapsto f(w+g+d_k)\}$ is relatively compact in the topology of uniform convergence on $\mathbb Z^n$. 
Let $S\subset C$ be a countable dense subset. Using Cantor's diagonal argument we find a subsequence $\{d_{k_l}\}$ of $\{d_k\}$ such that for each $w\in S$ the family of translates $\{\mathbb Z^n \ni g \mapsto f(w+g+d_{k_l})\}$ converges in the topology of uniform convergence on $\mathbb Z^n$. 

Now, since $f$ is uniformly continuous on $T'$, for every $\varepsilon>0$ there exists $\delta>0$ such that for all $w_1, w_2 \in C$ satisfying $|w_1-w_2|<\delta$ and all  $h \in \mathbb Z^n$,
$$
|f(w_1+h)-f(w_2+h)|<\frac{\varepsilon}{3}.
$$
Since $C$ is compact, it can be covered by finitely many $\delta$-neighbourhoods of points, say, $w_1,\dots, w_p$, in $S$. Then we can find 
$N\in\mathbb N$ so that for all $l,m>N$, $w_j$, $1\le j\le p$, and $g \in \mathbb Z^n$,
$$
|f(w_j+g+d_{k_l})-f(w_j+g+d_{k_m})|<\frac{\varepsilon}{3}.
$$
The last two inequalities together with the triangle inequality imply that for all $l,m>N$, $z \in C$ and $g \in \mathbb Z^n$,
$$
|f(z+g+d_{k_l})-f(z+g+d_{k_m})|<\varepsilon .
$$
Since $\{z+g:z \in C,\, g \in \mathbb Z^n\}=T'$, the latter implies that
for all $l,m>N$ and $z\in T'$,
$$
|f(z+d_{k_l})-f(z+d_{k_m})|<\varepsilon .
$$
Thus $\{T' \ni z \mapsto f(z+d_{k_l})\}$ is a Cauchy sequence in the topology of uniform convergence on $T'$, i.e., it converges uniformly on $T'$.

The proof is complete.

\section{Proofs of Theorems \ref{countthm}, \ref{tubularnbd} and \ref{extapthm}}
\label{countthmsect}

\SkipTocEntry\subsection{Proof of Theorem \ref{countthm}}
Our proof is based on Theorem \ref{thmA} and the equivalence of notions of a complex $\mathfrak a$-submanifold of $X$ and a complex submanifold of $c_{\mathfrak a}X$ (see subsection \ref{submanifoldsect} for the corresponding definitions and results). 

Thus, it suffices to prove that given a complex submanifold $Y \subset c_{\mathfrak a}X$ of codimension $k$ there exists an at most countable collection of functions $f_i \in \mathcal O(c_{\mathfrak a}X)$, $i \in I$, such that 

(\textit{i}) $Y=\{y \in c_{\mathfrak a}X: f_i(y)=0 \text{ for all } i \in I\}$, and

(\textit{ii}) for each $y_0 \in Y$ there exist a neighbourhood $W=\hat{\Pi}(W_0,L)$ (see~(\ref{pi2}) for notation) and functions $f_{i_1},\dots,f_{i_k}$ such that $Y \cap W=\{y \in U: f_{i_1}(y)=\dots=f_{i_k}(y)=0\}$
and the rank of map $z \mapsto \bigl(f_1(z,\omega),\dots,f_k(z,\omega) \bigr)$, $(z,\omega) \in W$, is maximal at each point of $Y \cap W$.

By Proposition \ref{idcohlem} the ideal sheaf $I_Y$ of $Y$ is coherent, hence by Theorem \ref{thmA}, there exists an at most countable collection of sections $f_i \in \Gamma(c_{\mathfrak a}X,I_Y)$ ($\subset \mathcal O_{c_{\mathfrak a}X}$), $i \in I$, that generate $I_Y$ at each point of $c_{\mathfrak a}X$. (This collection is at most countable because any open cover of $c_{\mathfrak a}X$ admits an at most countable refinement as fibres of the bundle $\bar{p}: c_{\mathfrak a}X\rightarrow X_0$ are compact and any open cover of complex manifold $X_0$ admits an at most countable refinement.) 
Therefore condition (\textit{i}) is valid for this collection of functions. In addition, for every point $y_0 \in Y$ there exist a neighbourhood  $U=\hat{\Pi}(U_0,K)$ of $y_0$, sections $f_{i_1},\dots,f_{i_m}$ and functions $u_{jl} \in \mathcal O(c_{\mathfrak a}X)$,
$1\le j\le k$, $1\le l\le m$, such that
\begin{equation}
\label{matridcount}
h_j=u_{j1}f_{i_1}+\dots+u_{jm}f_{i_m}, \quad 1 \leq j \leq k,
\end{equation}
where $h_j$ are generators of $I_Y|_{U}$ from Definition \ref{defap}
(modulo a biholomorphic transformation of Proposition \ref{localstructap} we may identify $h_j$ with $z_j$, the $j$-th coordinate of $z\in\mathbb C^n$). 

Equation \eqref{matridcount} implies that
$Y \cap U=\{y \in U: {f}_{i_1}(y)=\dots={f}_{i_m}(y)=0\}$.

Next, let $\nabla h_j$, $\nabla {f}_{i_l}$ denote the vector-valued functions $\nabla_z h_j(z,\omega)$, $\nabla_z {f}_{i_l}(z,\omega)$, $(z,\omega) \in U$. Then
\begin{equation*}
\nabla h_j=u_{j1}\nabla{f}_{i_1}+\dots+u_{jm}\nabla{f}_{i_m} \quad \text{ on } Y \cap U, \quad 1 \leq j \leq k.
\end{equation*}
Since $(\nabla h_j)_{j=1}^k$ has rank $k$ on $U$, we obtain that $k \leq m$, and $(u_{jl})_{1 \leq j \leq k, 1 \leq l \leq m}$, $(\nabla {f}_{i_l})_{l=1}^m$ have rank $k$ at each point of $U$. Thus there exist two subfamilies of vector-valued functions $(u_{jl_1})_{j=1}^k,\dots,(u_{jl_k})_{j=1}^k$ and $\nabla \tilde{f}_{l_1},\dots,\nabla\tilde{f}_{l_k}$, $\tilde{f}_{l}:=f_{i_l}$, that are linearly independent at $y_0$. Now, we apply the holomorphic inverse function theorem (see Theorem \ref{ift}) to the matrix identity (\ref{matridcount}) to find a neighbourhood $W=\hat{\Pi}(W_0,L)\Subset U$ of $y_0$ such that functions $\tilde{f}_l|_W$, $l \neq l_i$, $1 \leq i \leq k$, belong to the ideal in $\mathcal O(W)$ generated by  $\tilde{f}_{l_1}|_W,\dots,\tilde{f}_{l_k}|_W$, and the rank of map $z \mapsto \bigl(\tilde{f}_{l_1}(z,\omega),\dots,\tilde{f}_{l_k}(z,\omega) \bigr)$, $(z,\omega) \in W$, is maximal at each point of $Y \cap W$. Clearly,
$Y \cap W=\{y \in U: \tilde{f}_{l_1}(y)=\dots=\tilde{f}_{l_k}(y)=0\}$.

This completes the proof of the theorem.

\SkipTocEntry\subsection{Proof of Theorem \ref{tubularnbd}}

The proof follows the lines of the proof of the classical tubular neighbourhood theorem (see, e.g.,~\cite{Forst2}).

We use notation and results of Section \ref{compsect}. Clearly, Theorem \ref{tubularnbd} is a corollary of

\begin{theorem}
\label{tubularnbd2}
Let $X_0$ be a Stein manifold, and $Y \subset c_{\mathfrak a}X$ a complex submanifold (see~subsection \ref{submanifoldsect}). Then there exists an open neighbourhood $\Omega \subset c_{\mathfrak a}X$ of $Y$ and maps $h_t \in \mathcal O(\Omega,\Omega)$ continuously depending on $t \in [0,1]$, such that $h_t|_{Y}=\Id_{Y}$ for all $t\in [0,1]$, $h_0=\Id_{\Omega}$ and $h_1(\Omega)=Y$.
\end{theorem}

\begin{proof}

In the proof of Theorem \ref{tubularnbd2} we use the following notation and definitions.

Let $U$ be an open subset of $c_{\mathfrak a}X$ or of a complex submanifold $Y \subset c_{\mathfrak a}X$. In the category of ringed spaces $(U,\mathcal O_U)$ (see subsections \ref{charts}, \ref{submanifoldsect}) 
we define in a standard way holomorphic vector bundles on $U$, their subbundles, the Whitney sum of bundles, holomorphic bundle morphisms, etc (see~\cite{Hirz}).

Now, we define the (holomorphic) tangent bundle $Tc_{\mathfrak a}X$ on $c_{\mathfrak a}X$ as the pullback $\bar{p}^*TX_0$ of the (holomorphic) tangent bundle $TX_0$ of $X_0$. We denote by $T_xc_{\mathfrak a}X$ the fibre of $Tc_{\mathfrak a}X$ at $x \in c_{\mathfrak a}X$.

Next, we define a Hermitian metric on $Tc_{\mathfrak a}X$ as the pullback by $\bar{p}$ of a (complete) Hermitian metric on $TX_0$. 

Let $Y \subset c_{\mathfrak a}X$ be a complex submanifold. Every point $x \in Y$ has a neighbourhood $U=\hat{\Pi}(U_0,K) \subset c_{\mathfrak a}X$, where $U_0 \subset X_0$, $K \subset \hat{G}_{\mathfrak a}$ are open, so that $Y \cap U$ is the set of common zeros of functions $h_1,\dots,h_k \in \mathcal O(U)$ such that the maximum of moduli of determinants of square submatrices of the Jacobian matrix of the map $z \mapsto (h_1(z,\omega),\dots,h_k(z,\omega))$, $(z,\omega) \in U$, is uniformly bounded away from zero (see Definition \ref{defap}). We define the tangent bundle $TY$ of $Y$ as the subbundle of $Tc_{\mathfrak a}X|_Y$ whose fibres are orthogonal to the local vector fields $(z,\omega) \mapsto D_{z}h_1(z,\omega),\dots,D_zh_k(z,\omega)$, $(z,\omega) \in Y \cap U$.
Namely, in local coordinates $(z,\omega) \in U$ the metric
has a form
$$
ds^2(z,\omega)=\sum_{l,j} g_{lj}(z)dz_l \otimes d\bar{z}_j, \quad g_{lj}(z):=\left(\frac{\partial}{\partial z_l},\frac{\partial}{\partial z_j}\right)_{(z,\omega)};
$$
hence, if vector fields $D_{z}h_{i}$ ($1 \leq i \leq k$) are given by
$$
D_{z}h_{i}(z,\omega)=\sum_l a_{li}(z,\omega)\frac{\partial }{\partial z_l},
$$
then $T_{(z,\omega)}Y$ consists of vectors $\sum_l b_l \frac{\partial }{\partial z_l}$ such that $\sum_{l,j} a_{li}(z,\omega)\bar{b}_j g_{lj}(z)=0$ for all $1 \leq i \leq k$.

It is easily seen that $\iota^* TY$ coincides with the holomorphic tangent bundle of the complex submanifold $\iota^{-1}(Y)\subset X$. Since $\iota(\iota^{-1}(Y))$ is dense in $Y$, the bundle $TY$ is uniquely defined by the latter condition (see~subsection \ref{derhamsect}).

Consequently, we obtain the notion of the (holomorphic) normal bundle $NY \subset Tc_{\mathfrak a}X|_Y$ of $Y$.

\medskip

The tangent bundle $TX_0$ is generated by finitely many holomorphic vector fields $V_{0,k}$, $1 \leq k \leq m$, which determine holomorphic local flows $\varphi_k:O_k \rightarrow X_0$, where $O_k$ is an open neighbourhood of $\{0\} \times X_0$ in $\mathbb C \times X_0$ such that the differential of map
$$
F_0(t,\cdot):=(\varphi_m(s_m,\cdot) \circ \dots \circ \varphi_1(s_1,\cdot)):X_0 \rightarrow X_0, \quad t=(s_1,\dots,s_m),
$$
at $t=0$ is non-degenerate. The map $F_0$ is defined and holomorphic in a neighbourhood $W_0$ of $\{0\} \times X_0$ in $\mathbb C^m \times X_0$.

We will need notation and results of Section 4 in \cite{BK8}. There, we have established that $c_{\mathfrak a}X$ (as a set) is the disjoint union of connected complex manifolds $X_H$ ($H \in \Upsilon$) each is a covering of $X_0$. Using the lifting property, for every $H \in \Upsilon$ we can lift $F_0$ to a unique map $\tilde{F}_H(t,\cdot):W_H \rightarrow X_H$ that is defined and holomorphic on the neighbourhood $W_H:=(\Id_{\mathbb C^m} \times p_H)^{-1}(W_0)$ of $\{0\} \times X_H$ in $\mathbb C^m \times X_H$, where $p_H:X_H \rightarrow X_0$ is the covering projection. It is not difficult to show that these maps constitute 
a holomorphic map $$\tilde{F}:W \rightarrow c_{\mathfrak a}X,$$
where $W:=(\Id_{\mathbb C^m} \times \bar{p})^{-1}(W_0)$ is a neighbourhood of $\{0\} \times c_{\mathfrak a}X$ in $\mathbb C^m \times c_{\mathfrak a}X$. (Alternatively, one can define the map $\tilde{F}$ using the covering homotopy theorem and the local structure of $c_{\mathfrak a}X$, see~subsection \ref{charts}. Note that in local coordinates lifted from $X_0$ the map $\tilde{F}$ looks exactly the same as $F_0$.)

Next, for a fixed $x \in c_{\mathfrak a}X$ we consider a linear map
$$
\theta_x=\partial_t|_{t=0} \tilde{F}(t,x):\mathbb C^m \rightarrow T_x c_{\mathfrak a}X.
$$
We denote by $\theta$ the corresponding holomorphic bundle morphism $\mathbb C^m \times c_{\mathfrak a}X \rightarrow T c_{\mathfrak a}X$.

Since vector fields $\bar{p}^*V_{0,j}$ span $Tc_{\mathfrak a}X$, the maps $\theta_x$ are surjective, for every $x \in c_{\mathfrak a}X$. Since $TY \subset Tc_{\mathfrak a}X|_Y$, we can define a holomorphic vector bundle over $Y$,
$$
E':=\theta^*(TY) \subset Y \times \mathbb C^m.
$$

\begin{lemma}
\label{splitlem}
There exists a holomorphic vector bundle $E$ over $Y$ such that
\begin{equation*}
E' \oplus E=Y \times \mathbb C^m.
\end{equation*}
\end{lemma}
\begin{proof}%[Proof of Lemma]
We have an exact 
sequence of holomorphic vector bundles over $Y$
\begin{equation}
\label{splits}
0 \rightarrow E' \rightarrow Y \times \mathbb C^m \overset{q}{\rightarrow} E'' \rightarrow 0
\end{equation}
(here $E''$ is the quotient bundle) which induces an exact sequence of \v{C}ech cohomology groups with values in the sheaves of germs of holomorphic sections of the corresponding holomorphic vector bundles
\begin{multline*}
0 \rightarrow \Gamma(Y,\Hom(E'',E')) \rightarrow \Gamma(Y,\Hom(E'',Y \times \mathbb C^m)) \rightarrow \Gamma(Y,\Hom(E'',E'')) \overset{\delta}{\rightarrow} \\ H^1(Y,\Hom(E'',E')) \rightarrow \cdots .
\end{multline*}
%(see, e.g., \cite[Ch.\,1,\,4.1d]{Hirz}).
Recall that sequence (\ref{splits}) splits if and only if $\delta(I)=0$, where $I:E'' \rightarrow E''$ is the identity homomorphism (see, e.g.,~\cite[Ch.\,1,\,4.1d-f]{Hirz}). Since the sheaf $\Hom(E'',E')$ is locally free, Theorem \ref{thmB_} implies that $H^1(Y,\Hom(E'',E'))=0$. Hence there is a holomorphic homomorphism $u:E'' \rightarrow Y \times \mathbb C^m$ such that $q\circ u={\rm Id}$, i.e., $Y \times \mathbb C^m=E' \oplus E$ with $E:=u(E'')$.
\end{proof}

It follows from this lemma that the restriction $\theta:E \rightarrow Tc_{\mathfrak a}X|_E$ is an injective holomorphic bundle morphism such that $Tc_{\mathfrak a}X|_Y=TY \oplus \theta(E)$. Therefore, we have

\begin{lemma}
\label{thlem}
 $\theta|_E$ determines a holomorphic isomorphism between bundles $E$ and $NY$.
\end{lemma}

Further, we define $$F:=\tilde{F} \circ (\theta|_E)^{-1}:NY \rightarrow c_{\mathfrak a}X$$
and show that there exists a neighbourhood $V$ of the zero section of $NY$ that is mapped by $F$ biholomorphically to a neighbourhood of $Y$ in $c_{\mathfrak a}X$. Using this and replacing $V$ by a smaller neighbourhood of the zero section of $NY$ with convex fibres over $Y$, we can define the required maps $h_t$ by dilations along images of the fibres of this smaller neighbourhood under $F$. This would complete the proof of the theorem.

\medskip

We prove that $F$ is a biholomorphism near the zero section of $NY$ in two steps.

(1) First, we show that $F$ is a local biholomorphism, i.e., every point of the zero section of $NY$ has a neighbourhood $V$ such that the restriction $F|_{V}$ determines a biholomorphism between $V$ and $F(V) \subset c_{\mathfrak a}X$.

Indeed, by Proposition \ref{localstructap} for every $x \in Y$ one can find its neighbourhood $U_x \subset c_{\mathfrak a}X$ biholomorphic to $U_0 \times K$, where $U_0 \subset X_0$, $K \subset \hat{G}_{\mathfrak a}$ are open, such that (a) $Y \cap U$ is biholomorphic to $Y_0 \times K$ for $Y_0 \subset U_0$ a complex submanifold of $U_0$; (b)
$NY|_{Y \cap U} \cong NY_0 \times K$; (c)
$F: NY_0 \times K \rightarrow U_0 \times K$ is determined by a collection of maps $NY_0 \rightarrow U_0$ continuously depending on variable in $K$ such that the maximum of moduli of determinants of square submatrices of their Jacobian matrices are uniformly bounded away from zero. 

The required result now follows from the inverse function theorem with continuous dependence on parameter (Theorem \ref{ift}).

(2) Now, we show that there is a neighbourhood $V$ of the zero section of $NY$ such that 
$F|_{V}$ is an injection; since $F$ is holomorphic, this would imply the required.

We have defined $F$ in such a way that it maps the fibres of $E$ that lie over the points of $Y \cap X_H$ into $X_H$, for every $H \in \Upsilon$ (see the definition of $\tilde{F}$ above).

Let $\rho_H$ be the path metric determined by the pullback to $X_H$ of the (complete) Hermitian metric on $X_0$ (that we fixed previously). We define a pseudo-metric $\rho$ on $c_{\mathfrak a}X$ by
$$
\rho(x_1,x_2):=\rho_H(x_1,x_2)<\infty \text{ if } x_1,x_2 \in X_H, \quad \rho(x_1,x_2):=\infty \text{ otherwise}.
$$
Let $\|\cdot\|_x$ denote the norm on the fibres of bundle $NY$ determined by the restriction to $NY$ of the Hermitian metric on $Tc_{\mathfrak a}X$, defined above. 
For $y \in Y$ by $v_y$ we denote an element of $N_yY$.
For $x \in Y$ we set
$$
V_\delta(x):=\{v_y \in NY: \rho(x,y)<\delta, \|v_y\|_y<\delta\}.
$$
Using the construction of part (1) based on the inverse function theorem with continuous dependence on parameter, one can easily show that there is a positive function $b \in C(Y)$ such that
\begin{equation}
\label{metest}
\rho(x,F(v_x)) \leq b(x) \|v_x\|_x
%\in c_{\mathfrak a}X
\end{equation}
for all $v_x$ in a neighbourhood of the zero section of $NY$.
Also, using the assertion of part (1), one can show that there is a positive function $r \in C(Y)$ such that $F|_{V_{r(x)}(x)}$ is a biholomorphism for all $x \in Y$. Now, we set
$$
V:=\left\{v_y \in NY: \|v_y\|_y<\frac{r(y)}{2 \max\{1,b(y)\}}\right\}.
$$
This is an open neighbourhood of the zero section of $NY$. Let us show that $F|_V$ is injective. Indeed, assume that $v_x$, $v_y \in V$ and $F(v_x)=F(v_y)$. Without loss of generality we may assume that $r(y) \leq r(x)$. %We have $x$, $y \in Y \cap X_H$ for some $H \in \Upsilon$ (see above), and hence $\rho(x,y)<\infty$. 
Thus, using the triangle inequality and (\ref{metest}) we obtain:
\begin{multline*}
\rho(x,y) \leq \rho(x,F(v_x))+\rho(F(v_x),F(v_y))+\rho(y,F(v_y))=\\ \rho(x,F(v_x))+\rho(y,F(v_y)) \leq \frac{1}{2} r(x)+\frac{1}{2} r(y) \leq r(x).
\end{multline*}
It follows that $v_x$, $v_y \in V_{r(x)}(x)$. Since $F|_{V_{r(x)}(x)}$ is a biholomorphism, we arrive to a contradiction with the assumption $F(v_x)=F(v_y)$. Therefore, $F|_V$ is injective.

The proof of the theorem is complete.
\end{proof}

\begin{remark}
In the classical tubular neighbourhood theorem the neighbourhood of a closed submanifold is chosen to be a Stein open submanifold (see, e.g.~\cite{Forst2}). 
The following question naturally arises: is it possible to choose $\Omega$ in Theorem \ref{tubularnbd2} to be a Stein open submanifold of $c_{\mathfrak a}X$ (see the definition in subsection \ref{cartansubmsect})?
\end{remark}

\SkipTocEntry\subsection{Proof of Theorem \ref{extapthm}}
\label{extsect2}

In view of Propositions \ref{basicpropthm}, \ref{equivprop2} and \ref{closprop}, Theorem \ref{extapthm} follows from

\begin{theorem}
\label{extapthm2}
Let $X_0$ be a Stein manifold, $Y \subset c_{\mathfrak a}X$ be a complex submanifold, $f \in \mathcal O(Y)$. Then there exists a function $F \in \mathcal O(c_{\mathfrak a}X)$ such that $F|_Y=f$.
\end{theorem}

\begin{proof}%[Proof of Theorem \ref{extapthm2}] 
We will need

\begin{lemma}
\label{localextap}
Let $f \in \mathcal O(Y)$.
For every point $y_0 \in Y$ there exist a neighbourhood $V \subset c_{\mathfrak a}X$ of $y_0$ and a function $F_V \in \mathcal O(V)$ such that $F_V|_{V \cap Y}=f|_{V \cap Y}$.
\end{lemma}
\begin{proof}%[Proof of Lemma]
By Proposition \ref{localstructap} there exist an open neighbourhood $V \subset c_{\mathfrak a}X$ of $y_0$, open subsets $V_0 \subset \mathbb C^n$, $K \subset \hat{G}_{\mathfrak a}$, and a biholomorphic map $\Phi \in \mathcal O(V_0 \times K,V)$ such that $$\Phi^{-1}(V \cap Y)=Z_0 \times K, \quad \text{ where }
Z_0=\{(0,\dots,0,z_{k+1},\dots,z_n): (z_1,\dots,z_n) \in V_0\}.$$
Let $\tilde{f}:=\Phi^*f \in \mathcal O(Z_0 \times K)$. We define
%\begin{equation*}
$$\tilde{F}_V(z_1,\dots,z_n,\omega):=\tilde{f}(z_{k+1},\dots,z_n,\omega), \quad (z_1,\dots,z_n,\omega) \in V_0 \times K,$$
%\end{equation*}
and $F_V:=(\Phi^{-1})^*\tilde{F}_V$.
\end{proof}

Now, by Lemma \ref{localextap} there exist an open cover $\mathcal U=\{U_j\}$ of $c_{\mathfrak a}X$ and functions $f_j \in \mathcal O(U_j)$ such that $f_j|_{Y \cap U_j}=f|_{Y \cap U_j}$ if $Y \cap U_j \neq \varnothing$; if $Y \cap U_j=\varnothing,$ we define $f_j:=0$. Then $\{g_{ij}:=f_i-f_j \text{ on } U_i \cap U_j \neq \varnothing\}$ is a $1$-cocylce with values in $I_Y$. 
By Proposition \ref{idcohlem} sheaf $I_Y$ is coherent, so by Theorem \ref{thmB} $H^1(c_{\mathfrak a}X,I_Y)=0$. Thus, $\{g_{ij}\}|_{\mathcal V}$ represents $0$ in  $H^1(\mathcal V,I_Y)$ for a refinement $\mathcal V$ of $\mathcal U$. To avoid abuse of notation we may assume without loss of generality that $\mathcal V=\mathcal U$.
Therefore, we can find holomorphic functions $h_j \in \Gamma(U_j,I_Y)$ such that $g_{ij}=h_i-h_j$ on $U_i \cap U_j \neq \varnothing$. Now, we define $F|_{U_j}:=f_j-h_j$ for all $j$.
\end{proof}

\section{Proofs of Theorems \ref{divrelprop}, \ref{zeroapthm}, \ref{chernthm} and Proposition  \ref{zeroapcor}}
\label{needs}

\SkipTocEntry\subsection{} In the proofs we use the following definitions.

(1) Let $\{U_\alpha\}$ be an open cover of $Z\subset X$ and $L$ and $L'$ be line bundles on $Z$ in one of the categories introduced in subsection~\ref{sectdivisors} defined on $\{U_\alpha\}$ by cocycles $d_{\alpha\beta}$ and $d'_{\alpha\beta}$, respectively. 
Recall that an isomorphism between $L$ and $L'$ is given by nowhere zero functions $h_\alpha$ on $U_\alpha$ (of the same category as  $d_{\alpha\beta}$, $d_{\alpha\beta}'$) such that $d'_{\alpha\beta}=h_\alpha d_{\alpha\beta} h_{\beta}^{-1}$ on $U_\alpha \cap U_\beta$ for all $\alpha$, $\beta$.

(2) In the proofs below we work with the \v{C}ech cohomology groups of sheaves on $X$ or complex $\mathfrak a$-submanifolds of $X$ associated to  presheaves of functions defined on subsets of $X$ open in topology $\mathcal T_{\mathfrak a}$ or their intersections with the submanifolds. By definition (see, e.g.,~\cite{Gun3}), these groups are inverse limits of the \v{C}ech cohomology groups defined on open covers of $X$ or of its complex $\mathfrak a$-submanifolds of class ($\mathcal T_{\mathfrak a}$) (see~Definitions \ref{tfinedef} and \ref{tfinedef2}).

(3) Recall that $\mathcal O_{\mathfrak a}(V)=\iota^*\mathcal O(U)$, where $V=\iota^{-1}(U) \in \mathcal T_{\mathfrak a}$, $U \subset c_{\mathfrak a}X$ are open. In particular, $\mathcal O_{\mathfrak a}=\iota^*\mathcal O$, where $\mathcal O$ is the structure sheaf of $c_{\mathfrak a}X$, and $\iota:X \rightarrow c_{\mathfrak a}X$ is the canonical map (see~Section \ref{mainsect}). Since $\iota(X)$ is dense in $c_{\mathfrak a}X$,
spaces $\mathcal O_{\mathfrak a}(V)$ and $\mathcal O(U)$ are isomorphic. 
It follows from the definition of cohomology groups that $H^p(X,\mathcal O_{\mathfrak a})=H^p(c_{\mathfrak a}X,\mathcal O)$, $p\in\mathbb N$.
A similar argument yields $H^p(Z,\mathcal O_{\mathfrak a})=H^p(Y,\mathcal O)$, $p\in\mathbb N$, where $Y \subset c_{\mathfrak a}X$ is a complex submanifold and $Z:=\iota^{-1}(Y) \subset X$.

\SkipTocEntry\subsection{Proof of Theorem \ref{divrelprop}}

Let us prove the first assertion.

Let $E=\{f_\alpha \in \mathcal O_{\mathfrak a}(U_\alpha)\}$ so that $L_E$ is determined by the cocycle $\{d_{\alpha\beta}:=f_\alpha f_\beta^{-1} \in \mathcal O_{\mathfrak a}(U_\alpha \cap U_\beta)\}$. By Definition \ref{semidef} 
there exist nowhere zero functions $h_\alpha \in \mathcal O_{\ell_\infty}(U_\alpha)$ with $|h_\alpha| \in C_{\mathfrak a}(U_\alpha)$ such that $d_{\alpha\beta}=h_\alpha^{-1} h_\beta$ for all $\alpha$, $\beta$, see subsection \ref{needs}.1 (1). We define $f:=f_\alpha h_\alpha$ on $U_\alpha$. This is a function in $\mathcal O(Z)$ such that $|f|_{U_\alpha}|=|f_\alpha|| h_\alpha| \in C_{\mathfrak a}(U_\alpha)$ for each $\alpha$.
One can easily show using a partition of unity on the complex submanifold $Y \subset c_{\mathfrak a}X$ such that $Z=\iota^{-1}(Y)$ (see subsection \ref{submanifoldsect}) that the latter implies $|f| \in C_{\mathfrak a}(Z)$. By our construction, divisor $E_f \in \Div(X)$ is $\ell_\infty$-equivalent to $E$.

\medskip

Conversely, suppose that $\mathfrak a$ is such that $\hat{G}_{\mathfrak a}$ is a compact topological group and $j(G)\subset\hat{G}_{\mathfrak a}$ is a dense subgroup, and let $E=\{f_\alpha \in \mathcal O_{\mathfrak a}(U_\alpha)\}\in {\rm Div}_{\mathfrak a}(X)$. 
By our assumption there exist nowhere zero functions $h_\alpha \in \mathcal O_{\ell_\infty}(U_\alpha)$ with $h_\alpha^{-1} \in \mathcal O_{\ell_\infty}(U_\alpha)$ such that
$f|_{U_\alpha}=h_\alpha f_\alpha$ for all $\alpha$ (see~Definition \ref{diveqdef}). It is clear that the family $\{h_\alpha\}$ determines an $\ell_\infty$-isomorphism of the line $\mathfrak a$-bundle $L_E:=\{(U_\alpha \cap U_\beta,d_{\alpha\beta}:=f_\alpha f_\beta^{-1})\}$ of $E$ onto the trivial bundle $\{(U_\alpha \cap U_\beta,1)\}$ (see subsection \ref{needs}.1 (1)); to conclude that $L_E$ is $\mathfrak a$-semi-trivial it remains to show that $|h_\alpha|, |h_\alpha|^{-1} \in C_{\mathfrak a}(U_\alpha)$ for all $\alpha$ (see~Definition \ref{semidef}). 

By Proposition \ref{basicpropthm}(2) there exist open subsets $V_\alpha \subset c_{\mathfrak a}X$ and functions $\hat{f}_\alpha \in \mathcal O(V_\alpha)$ such that $U_\alpha=\iota^{-1}(V_\alpha)$ and $f_\alpha=\iota^*\hat{f}_\alpha$ for all $\alpha$; also, there exists a function $F \in C(c_{\mathfrak a}X)$ such that $|f|=\iota^*F$. We will show that there exist positive functions $g_\alpha\in C(V_\alpha)$ such that 
$\iota^*g_{\alpha}=|h_\alpha|$. Then by Proposition \ref{basicpropthm}(1) $|h_\alpha|, |h_\alpha|^{-1} \in C_{\mathfrak a}(U_\alpha)$. 
%Clearly, $|h_\alpha|=\iota^*(F|_{V_\alpha})/\iota^*(|\hat{f}_\alpha|)$ and $|h_\alpha|^{-1}=\iota^*(|\hat{f}_\alpha|/F|_{V_\alpha})$. Thus, if we show that $F|_{V_\alpha}/|\hat{f}_\alpha|,|\hat{f}_\alpha|/F|_{V_\alpha} \in C(V_\alpha)$, then by Proposition \ref{basicpropthm}(1) we get the required inclusion $|h_\alpha|, |h_\alpha|^{-1} \in C_{AP}(U_\alpha)$. 

Let us fix $\alpha$. First, note that since the required inclusion is a local property, we may assume without loss of generality that $V_\alpha$ is biholomorphic to $V_0 \times K$, where $V_0\subset X_0$ is an open coordinate chart and $K \subset \hat{G}_{\mathfrak a}$ is open, $K \in \mathfrak Q$ (see subsection \ref{holfuncsect}). In particular, we can identify $V_\alpha$ with  $V_0\times K$.

The proof of the required inclusion consists of three parts.

(1) Let us show that $\hat{f}_\alpha(\cdot,\eta) \not\equiv 0$ for every $\eta \in K$. 

By definition, the family of functions $\hat{f}:=\{\hat{f}_{\alpha}\}$ determines a not identically zero holomorphic section of a holomorphic line bundle $\hat{L}_E$ on $c_{\mathfrak a}X$ (see~subsection \ref{sectdivisors}).
Based on results in \cite{BK8} we have:

$c_{\mathfrak a}X=\sqcup_{H \in \Upsilon}\, \iota_H(X_H)$, where $X_H=X$, $\iota_H:X_H \rightarrow c_{\mathfrak a}X$ is holomorphic (see~subsection \ref{charts}) and $\iota_H(X_H)$ is dense in $c_{\mathfrak a}X$ for every $H \in \Upsilon$ (cf.~subsection 4.1 and Example 4.2 in \cite{BK8}). 

In particular, since $\iota_H(X_H)$ is dense in $c_{\mathfrak a}X$, assuming that section $\hat{f} \equiv 0$ on $\iota_H(X_H)$ for some $H \in \Upsilon$ we obtain $\hat{f} \equiv 0$ on $c_{\mathfrak a}X$, a contradiction.

Suppose, on the contrary, that $\hat{f}_\alpha(\cdot,\eta) \equiv 0$ for some $\eta \in K$. Then 
%Thus, if assuming on the contrary $f_\alpha(\cdot,\eta) \equiv 0$ we obtain that $\hat{f}|_{\iota_H(X_H)}$ is identically zero on some $\iota_H(X_H)$, then we get $f_\alpha(\cdot,\eta) \not\equiv 0$, as needed. 
there exists a unique $H \in \Upsilon$ such that $V_0 \times \{\eta\}~(\cong \hat{\Pi}(V_0,\{\eta\})) \subset \iota_H(X_H)$. 
We set $\hat{f}_{\eta}:=\iota_H^*\hat{f}$. This is a holomorphic section of holomorphic line bundle $\iota_H^*\hat{L}_E$ on $X_H$. By our assumption $\hat{f}_\eta$ is zero on an open subset of $X_H~(=X)$. Since $\hat{f}_\eta$ is holomorphic and $X_H$ is connected, $\hat{f}_\eta \equiv 0$; hence $\hat{f}|_{\iota_H(X_H)} \equiv 0$, a contradiction, i.e., $\hat{f}_\alpha(\cdot,\eta) \not\equiv 0$.

(2) Next, we show that $F|_{V_\alpha}(\cdot,\eta) \not\equiv 0$ for every $\eta \in K$. 

Assume on the contrary that there exists $\eta_0 \in K$ such that $F|_{V_\alpha}(\cdot,\eta) \equiv 0$. By part (1) $\hat{f}_\alpha(\cdot,\eta_0) \not \equiv 0$, so we can choose an open $V_0' \Subset V_0$ such that $|\hat{f}_\alpha(\cdot,\eta_0)| \geq c>0$ on $V_0'$.
Now, under the identification $V_\alpha$ with $V_0\times K$ the set $U_\alpha=\iota^{-1}(V_\alpha)$ is identified with $U_\alpha=V_0 \times L$, where $L:=j^{-1}(K) \subset G$, and so $\iota|_{U_\alpha}=\Id_{V_0} \times j|_{L}$. 
Since $F|_{V_\alpha}$ is continuous and $j(L)$ is dense in $K$ (see~subsection \ref{constrsect}), there exists a net $\{g_\gamma\} \subset L$ such that the net $\{j(g_\gamma)\}\subset K$ converges to $\eta_0$. By continuity $\hat{f}_\alpha(\cdot,j(g_\gamma))$ converges to $\hat{f}_\alpha(\cdot,\eta_0)$ uniformly on $V_0'$ (so we may assume without loss of generality that $|\hat{f}_\alpha(\cdot,j(g_\gamma))| \geq \frac{c}{2}>0$ for all $\gamma$), while $F|_{V_\alpha}(\cdot,j(g_\gamma))$ converges to $0$ uniformly on $V_0'$. Since
$\bigl|f|_{U_\alpha}\bigr|=|h_\alpha| |f_\alpha|$, where $|f|=\iota^*F$, $f_\alpha=\iota^*\hat{f}$, the latter implies that $|h_\alpha(\cdot,g_\gamma)| \rightarrow 0$ uniformly over $V_0'$. We will show that this leads to a contradiction with our assumption. 

Indeed, due to results of subsection 4.1 of \cite{BK8} there exists an equivariant with respect to right actions of $G$ continuous proper map $\kappa:\hat{G}_{\ell_\infty} \rightarrow \hat{G}_{\mathfrak a}$. Set $K':=\kappa^{-1}(K)$. Passing to the corresponding subnets, if necessary, we may assume without loss of generality that there
exists a net $\{\xi_\gamma\} \subset K'$ having limit $\xi_0 \in K'$  such that $\kappa(\xi_0)=\eta_0$ and $\kappa(\xi_\gamma)=j(g_\gamma)$ for all $\gamma$. Further, since by our assumption $h_\alpha \in \mathcal O_{\ell_\infty}(U_\alpha)$, by Proposition \ref{basicpropthm}(2) there exists a function $\tilde{h}_\alpha \in \mathcal O(V_0 \times K')$ such that $(\Id_{V_0} \times j_{\ell_\infty})^*\tilde{h}_\alpha=h_\alpha$ (see~subsection \ref{constrsect} for notation). Now, since $|h_\alpha(\cdot,g_\gamma)| \rightarrow 0$ uniformly on $V_0'$, we obtain that $|\tilde{h}_\alpha(\cdot,\xi_\gamma)| \rightarrow 0$ uniformly on $V_0'$; so by continuity $\tilde{h}_\alpha(\cdot,\xi_0) \equiv 0$ on $V_0'$. However, by Proposition \ref{basicpropthm}(2) function $h^{-1}_\alpha \in \mathcal O_{\ell_\infty}(U_\alpha)$ admits a continuous extension to $V_0 \times K'$ such that its product with $\tilde{h}_\alpha$ is identically $1$ (because $h_\alpha h_\alpha^{-1} \equiv 1$ on $U_\alpha$ and $(\Id_{V_0} \times j_{\ell_\infty})(U_\alpha)$ is dense in $V_0 \times K'$). This contradicts the identity $\tilde{h}_\alpha(\cdot,\xi_0) \equiv 0$ on $V_0'$ and completes the proof of step (2).

(3) Finally, we show that there exists a positive function $g_\alpha\in C(V_\alpha)$, $V_\alpha=V_0 \times K$, such that 
$\iota^*g_{\alpha}=|h_\alpha|$.
%$F|_{V_\alpha}/|\hat{f}_\alpha|$ and $|\hat{f}_\alpha|/F|_{V_\alpha}$ are nowhere zero continuous functions on $V_\alpha=V_0 \times K$.

Let $\mathcal Z\subset V_\alpha$ be the union of zero loci of functions $F|_{V_\alpha}$ and $|\hat{f}_{\alpha}|$. By parts (1) and (2) we obtain that the open set $\mathcal Z^c:=V_\alpha\setminus \mathcal Z$ is dense in $V_\alpha$ and, moreover, $F|_{V_\alpha}/|\hat{f}_\alpha|$ and $|\hat{f}_\alpha|/F|_{V_\alpha}$ are continuous on $\mathcal Z^c$. We set $\tilde\kappa:=(\Id_{V_0} \times \kappa): V_0\times \hat{G}_{\ell_\infty} \rightarrow V_0\times \hat{G}_{\mathfrak a}$. By the definition pullbacks by $\tilde \kappa$ of $\bigl(F/|\hat{f}_\alpha|\bigr)|_{\mathcal Z^c}$ and $\bigl(|\hat{f}_\alpha|/F\bigr)|_{\mathcal Z^c}$ to $\tilde\kappa^{-1}(\mathcal Z^c)\subset V_0\times K'$ coincide with $|\tilde h_\alpha|$ and $|\tilde h_\alpha|^{-1}$ there (see~part (2)).
This, the fact that the open set $\tilde\kappa ^{-1}(\mathcal Z^c)$ is dense in $V_0\times K'$ and the definition of $\tilde\kappa$ imply that
$|\tilde h_\alpha|$ is constant on fibres of $\tilde\kappa$. Since $\tilde\kappa$ is a proper continuous map and $V_0\times K$, $V_0\times K'$ are 
locally compact Hausdorff spaces, the latter implies that there exists a positive function $g_\alpha\in C(V_0\times K)$ such that $\tilde\kappa^* g_\alpha=|\tilde h_\alpha|$. By the definition $g_\alpha=F|_{V_\alpha}/|\hat{f}_\alpha|$ on $\mathcal Z^c$. This yields $\iota^*g_{\alpha}=|h_\alpha|$, as required.

\SkipTocEntry\subsection{Proof of Theorem \ref{zeroapthm}}

Suppose that conditions (1), (2) are satisfied. Let us show that the holomorphic line $\mathfrak a$-bundle $L$ is $\mathfrak a$-semi-trivial.

Suppose that $L$ is defined by a holomorphic $1$-cocycle $\{c_{\alpha\beta}\}$ on an open cover $\{U_\alpha\}$ of $Z$ of class $(\mathcal T_{\mathfrak a})$.  In what follows, we may need to pass several times to refinements of class ($\mathcal T_{\mathfrak a}$) of the cover $\{U_\alpha\}$. To avoid abuse of notation we may assume without loss of generality that $\{U_\alpha\}$  is acyclic with respect to the corresponding sheaves so that according to the classical Leray lemma we can work with cover $\{U_\alpha\}$ itself only.

By (1) we can find functions $c_\alpha \in \mathcal O_{\mathfrak a}(U_\alpha)$ such that $c^{-1}_\alpha \in \mathcal O_{\mathfrak a}(U_\alpha)$ and $c_\alpha^{-1} c_{\alpha\beta} c_\beta=d_{\alpha\beta}$ is locally constant on $U_\alpha \cap U_\beta$ for all $\alpha,\beta$; hence, $\{d_{\alpha\beta}\}$ determines an equivalent discrete ${\mathfrak a}$-bundle $L'$ on $Z$. 
%We replace $E$ with equivalent $\mathfrak a$-divisor $E_0=\{(U_\alpha,c_\alpha f_\alpha)\}$ so that $L_{E_0}=L'$. 
Now, we have polar representation
$$
d_{\alpha\beta}=|d_{\alpha\beta}|e^{il_{\alpha\beta}}  \quad \text{ for all }\quad \alpha,\beta.
$$
Then $\{|d_{\alpha\beta}|\}\in  Z^1(\{U_\alpha\},\mathbb R_+)$, $\{e^{il_{\alpha\beta}}\}  \in Z^1(\{U_\alpha\},\mathbb U_1)$, where $\mathbb U_1$ is the 1-dimensional unitary group, are multiplicative locally constant cocycles.

Since $|d_{\alpha\beta}| \neq 0$ are locally constant and belong to $\mathcal O_{\mathfrak a}(U_\alpha\cap U_\beta)$, functions $\log|d_{\alpha\beta}|\in\mathcal O_{\mathfrak a}(U_\alpha\cap U_\beta)$ as well and form an additive holomorphic 1-cocyle on $\{U_\alpha\}$. We can resolve this cocycle by Theorem \ref{thmB_} (see definitions in subsection \ref{needs}.1), i.e.,~there exist functions
$g_\alpha \in \mathcal O_{\mathfrak a}(U_\alpha)$ such that $e^{g_\alpha}\cdot e^{-g_\beta}=|d_{\alpha\beta}|$ for all $\alpha,\beta$.

Further, by condition (2) bundle $L'$ is trivial in the category of discrete line bundles on $Z$. This implies existence of functions $e^{il_\alpha} \in \mathcal O(U_\alpha)$, where $l_\alpha$ are real-valued locally constant, such that $$e^{il_{\alpha\beta}}=e^{il_\alpha}\cdot e^{-il_\beta} \quad \text{ on }\quad U_\alpha \cap U_\beta.$$ 
Now, we define 
$$
\psi_\alpha:=e^{-g_\alpha}\cdot e^{-il_\alpha}\cdot c_\alpha \in \mathcal O_{\ell_\infty}(U_\alpha).
$$
Then $d_{\alpha\beta}=\psi_\alpha\psi_\beta^{-1}$, so the family of functions $\{\psi_\alpha\}$ determines an isomorphism in category $\mathcal L_{\ell_\infty}(Z)$ of $L$ onto the trivial line bundle (see subsection \ref{needs}.1 (1)).
Moreover, 
$|\psi_\alpha|=e^{-\Rea\, g_\alpha}|c_\alpha|$, $|\psi_\alpha|^{-1}=e^{\Rea\, g_\alpha}|c_\alpha^{-1}| \in C_{\mathfrak a}(U_\alpha)$ for all $\alpha$, as required.

\medskip

Conversely, suppose that the holomorphic line $\mathfrak a$-bundle $L$ is $\mathfrak a$-semi-trivial. Let us show that conditions (1) and (2) are satisfied. As before, we assume that $L$ is determined by a cocycle
$\{c_{\alpha\beta} \in \mathcal O_{\mathfrak a}(U_\alpha\cap U_\beta)\}$ on cover $\{U_\alpha\}$ of $Z$. 
By Definition \ref{semidef} there exist nowhere zero functions $\psi_\alpha \in \mathcal O(U_\alpha)$ with $|\psi_\alpha|,|\psi_\alpha|^{-1} \in C_{\mathfrak a}(U_\alpha)$ such that $\psi_\alpha c_{\alpha\beta}\psi_\beta^{-1} \equiv 1$ on $U_\alpha\cap U_\beta\ne\emptyset$ (see definition (1) in subsection \ref{needs}.1).

We will use notation and results of subsection \ref{derhamsect}.
Denote $$\Lambda_{\mathfrak a}^{p,k}(U_\alpha):=\iota^*\Lambda^{p,k}(V_\alpha), \quad Z_{\mathfrak a}^{p,k}(U_\alpha):=\iota^*Z^{p,k}(V_\alpha),$$ where $V_\alpha \subset Y$ is open and such that $U_\alpha=\iota^{-1}(V_\alpha)$, $Y \subset c_{\mathfrak a}X$ is the closure of $\iota(Z)$ (a complex submanifold of $c_{\mathfrak a}X$, see~Proposition \ref{closprop}), $\Lambda^{p,k}(Y)$ is the space of $(p,k)$-forms on $Y$ and $Z^{p,k}(Y)$ is the space of $\bar{\partial}$-closed form on $Y$. Also, denote $C^\infty_{\mathfrak a}(U_\alpha):=\iota^*C^\infty(V_\alpha)$. 

Let us show that $|\psi_\alpha|, |\psi_\alpha|^{-1} \in C^\infty_{\mathfrak a}(U_\alpha)$. We may assume without loss of generality that $U_\alpha=U_0\times j^{-1}(K)$, $V_\alpha=U_0\times K$, where $U_0 \subset \mathbb C^m$, $m:={\rm dim}_{\mathbb C} Z$, is an open ball and $K \subset \hat{G}_{\mathfrak a}$, $K \in \mathfrak Q$, see~(\ref{base1}), is open 
(see~Proposition \ref{localstructap} and subsection \ref{charts}). Then there exist $\tilde{\psi}_\alpha$, $\tilde{\psi}_\alpha^{-1}\in \mathcal O(U_0\times K')$, $K':=\kappa^{-1}(K)\subset\hat{G}_{\ell_\infty}$, such that $({\rm Id}_{U_0}\times j_{\ell_\infty})^*(\tilde{\psi}_\alpha)^{\pm 1}=(\psi_\alpha)^{\pm 1}$ (see part (2) in the proof of Theorem \ref{divrelprop} and subsection \ref{charts}) which can be viewed as holomorphic functions on $U_0$ taking values in the Fr\'{e}chet space $C(K')$. In particular, these are $C(K')$-valued $C^\infty$ functions. Further, since $|\psi_\alpha|,|\psi_\alpha|^{-1} \in C_{\mathfrak a}(U_\alpha)$, there exist nowhere zero functions $\hat{\psi}_\alpha$, $\hat{\psi}_\alpha^{-1}\in C(U_0\times K)$ whose pullbacks by ${\rm Id}_{U_0}\times\kappa$ to $U_0\times K'$ coincide with $\tilde{\psi}_\alpha$ and $\tilde{\psi}_\alpha^{-1}$, respectively. The last two facts imply easily that $\hat{\psi}_\alpha, \hat{\psi}_\alpha^{-1}\in C^\infty(U_0 \times K)$ (see~subsection \ref{derhamsect} for the definition). Pullbacks of $\hat{\psi}_\alpha, \hat{\psi}_\alpha^{-1}$ by $\iota:=\Id_{U_0} \times j$ are functions $|\psi_\alpha|, |\psi_\alpha|^{-1}$. Thus these functions are in $C^\infty_{\mathfrak a}(U_\alpha)$.

Now, since $\hat{\psi}_\alpha$, $\hat{\psi}_\alpha^{-1}$ are nowhere zero, $\log|\psi_\alpha| \in C^\infty_{\mathfrak a}(U_\alpha)$; hence, forms $\partial\bigl( \log |\psi_\alpha|\bigr)$ belong to $\Lambda_{\mathfrak a}^{1,0}(U_\alpha)$ 
%, {\bf we don't have definitions of these spaces on $Z$} 
and satisfy $\bar{\partial } \partial\bigl( \log |\psi_\alpha|\bigr)=0$, that is, $\partial \bigl(\log |\psi_\alpha|\bigr) \in Z_{\mathfrak a}^{1,0}(U_\alpha)$. 
Identifying $U_\alpha$ with $U_0\times j^{-1}(K)$ by a biholomorphism (see Proposition \ref{localstructap}) we obtain that
$\partial \bigl(\log |\psi_\alpha|\bigr)$ is the pullback by $\iota$ of the $d$-closed holomorphic 1-form $\partial\bigl(\log\hat{\psi}_\alpha\bigr)$ on $U_0$ with values in the Fr\'{e}chet space $C(K)$ (see subsection \ref{derhamsect} for notation). 
 %Furthermore, shrinking $U_\alpha$ (and $U_0$, $K$) if necessary, we may assume that the latter form takes values in the Banach space $C_b(K)$ of bounded continuous functions on $K$.
Integrating the latter form along rays in $U_0$ emanating from the center and taking the pullback of the obtained function by $\iota$ we obtain a function $u_\alpha \in \mathcal O_{\mathfrak a}(U_\alpha)$ such that $\partial u_\alpha=\partial\bigl(\log |\psi_\alpha|\bigr)$. Hence, $\log |\psi_\alpha|-u_\alpha=\bar{v}_\alpha$ for some 
$v_\alpha \in \mathcal O_{\mathfrak a}(U_\alpha)$. We define $b_\alpha:=u_\alpha+v_{\alpha} \in \mathcal O_{\mathfrak a}(U_\alpha)$.
Then $\log|\psi_\alpha|=\Real\, b_\alpha$. Now, set 
$$
d_{\alpha\beta}:=e^{b_\alpha} c_{\alpha\beta} e^{-b_\beta} \in \mathcal O_{\mathfrak a}(U_\alpha\cap U_\beta) \quad \text{ for all }\quad \alpha,\beta.
$$
Then $|d_{\alpha\beta}|=|\psi_\alpha c_{\alpha\beta}\psi_\beta^{-1}| \equiv 1$, i.e.,~$\{d_{\alpha\beta}\}$ is a locally constant 1-cocycle on the cover $\{U_\alpha\}$ of $Z$ with values in the unitary group $\mathbb U_1$. Therefore condition (1) is satisfied. 

Further, $\psi_{\alpha}e^{-b_\alpha} d_{\alpha\beta} \psi_{\beta}^{-1}e^{b_\beta} \equiv 1$ for all $\alpha$, $\beta$ and
$|\psi_{\alpha}e^{-b_\alpha}| \equiv 1$ on $U_\alpha$, that is, functions $\psi_{\alpha}e^{-b_\alpha}$ are locally constant for all $\alpha$. Hence the discrete line $\mathfrak a$-bundle $L':=\{(U_\alpha \cap U_\beta,d_{\alpha\beta})\}$ is trivial in the category of discrete line bundles on $Z$, i.e.,~condition
(2) is satisfied as well. 

The proof of the theorem is complete.

\SkipTocEntry\subsection{Proof of Proposition \ref{zeroapcor}}
\label{sub15.3}

%We will need results and definitions introduced in the beginning of the proof of Theorem \ref{zeroapthm}.

We use notation and results of subsection \ref{derhamsect}. Suppose that algebra $\mathfrak a$ is self-adjoint. Then $Z:=\iota^{-1}(Y)$ is a complex $\mathfrak a$-submanifold of $X$ (see~Definition \ref{manifolddef} and subsection \ref{submanifoldsect}). 
We set
$\Lambda_{\mathfrak a}^{p,k}(Z):=\iota^*\Lambda^{p,k}(Y)$, $Z_{\mathfrak a}^{p,k}(Z):=\iota^*Z^{p,k}(Y)$, and define
$$H^{p,k}_{\mathfrak a}(Z):=Z^{p,k}_{\mathfrak a}(Z)/\bar{\partial} \Lambda^{p,k-1}_{\mathfrak a}(Z), \quad p \geq 0, \quad k \geq 1, \qquad H_{\mathfrak a}^{p,0}(Z):=Z^{p,0}_{\mathfrak a}(Z).$$ 
These spaces of forms and cohomology groups are isomorphic to their counterparts on $Y$, so we have analogues of Proposition \ref{poincarelem} and Corollaries \ref{dolcor1}, \ref{dolcor2} on $Z$ (see definitions (2) and (3) of subsection \ref{needs}.1).

We will also need an analogue of the de Rham complex on $Y$. 

Let $Z^{m}(Y) \subset \Lambda^{m}(Y)$ denote the subspace of $d$-closed forms. Define $$H^{m}(Y):=Z^{m}(Y)/d \Lambda^{m-1}(Y), \quad p \geq 0, \quad m \geq 1, $$ 
$$
H^{0}(Y):=Z^{0}(Y).
$$
(``de Rham cohomology groups of $Y$'').
Now, set $\Lambda^m_{\mathfrak a}(Z):=\iota^*\Lambda^m(Y)$, $Z^m_{\mathfrak a}(Z):=\iota^*Z^m(Y)$, 
$$H^{m}_{\mathfrak a}(Z):=Z_{\mathfrak a}^{m-1}(Y)/d \Lambda_{\mathfrak a}^{m}(Z), \quad p \geq 0, \quad m \geq 1, \qquad H_{\mathfrak a}^{0}(Z):=Z_{\mathfrak a}^{0}(Z).$$  Then $H^{m}_{\mathfrak a}(Z)$ and $H^m(Y)$ are isomorphic.

Let us denote by $\mathcal O_{\mathfrak a}$ the sheaf associated to the presheaf of functions $\mathcal O_{\mathfrak a}(U)$, $U\subset Z$, $U \in \mathcal T_{\mathfrak a}$ (see~Definition \ref{tfinedef2}).  
Let $\mathbb Z_{\mathfrak a}$, $\mathbb R_{\mathfrak a} \subset \mathcal O_{\mathfrak a}$ denote subsheaves of
locally constant functions with values in groups $\mathbb Z$, $\mathbb R$, respectively.
Using an argument similar to that of the proof of Proposition \ref{poincarelem}, where instead of Lemma \ref{hl1} we use the Poincar\'{e} $d$-lemma for Banach-valued $d$-closed forms on a ball (see subsection \ref{poincare}), one obtains an analogue of the $d$-Poincar\'{e} lemma on $Y$ (i.e., a $d$-closed $C^\infty$ $m$-form, $m\ge 1$, on an open subset of $Y$ is locally $d$-exact). Then since sheaves $\Lambda^m$ of germs of $C^\infty$ $m$-forms on $Y$ are fine, see Lemma \ref{finesheaflem}, by a standard result about cohomology groups of sheaves admitting acyclic resolutions, see, e.g., \cite[Ch.B \S 1.3]{GR}, we obtain
\begin{equation}
\label{derhamiso}
H_{\mathfrak a}^{m}(Z) \cong H^m(Z,\mathbb R_{\mathfrak a}), \quad m \geq 0.
\end{equation}

Finally, by $\mathcal O^*_{\mathfrak a} \subset \mathcal O_{\mathfrak a}$ we denote a multiplicative subsheaf associated to the presheaf of functions $f\in\mathcal O_{\mathfrak a}(U)$, $U\subset Z$, $U \in \mathcal T_{\mathfrak a}$,
such that $f^{-1}\in \mathcal O_{\mathfrak a}(U)$ as well.

\begin{proof}[Proof of Proposition \ref{zeroapcor}]

First, we show that condition (1) of Theorem \ref{zeroapthm} is satisfied. 

We have an exact sequence of sheaves
$$
0 \rightarrow \mathbb Z_{\mathfrak a} \rightarrow \mathcal O_{\mathfrak a} \overset{e^{2\pi i \cdot }}{\rightarrow} \mathcal O^*_{\mathfrak a} \rightarrow 0
$$ 
which induces an exact sequence of cohomology groups
$$
\dots \rightarrow H^1(Z,\mathbb Z_{\mathfrak a}) \rightarrow H^1(Z,\mathcal O_{\mathfrak a}) \rightarrow H^1(Z,\mathcal O^*_{\mathfrak a}) \overset{\delta}{\rightarrow} H^2(Z,\mathbb Z_{\mathfrak a}) \rightarrow \cdots .
$$
By definition the class of holomorphic $\mathfrak a$-bundles isomorphic to the line $\mathfrak a$-bundle $L:=L_E$ of a divisor $E\in {\rm Div}_{\mathfrak a}(Z)$ determines an element of group $H^1(Z,\mathcal O^*_{\mathfrak a})$; its image under $\delta$ in $H^2(Z,\mathbb Z_{\mathfrak a})$ is denoted by $\delta(L)$ and is called the Chern class of $L$. On a suitable open cover $\{U_\alpha\}$ of $Z$ of class $(\mathcal T_{\mathfrak a})$ element $\delta(L)$ is defined by a locally constant 2-cocycle $\{m^L_{\alpha\beta\gamma}\} \in Z^2(\{U_\alpha\},\mathbb Z_{\mathfrak a})$ given by the formula (see, e.g.,~\cite{GH})
$$
m^L_{\alpha\beta\gamma}=\frac{1}{2\pi i} \left(\log c_{\alpha\beta}+\log c_{\beta\gamma}+\log c_{\gamma\alpha} \right) \quad \text{ on }\quad U_\alpha \cap U_\beta \cap U_\gamma,
$$
where $L$ is determined on $\{U_\alpha\}$ by 1-cocycle $\{c_{\alpha\beta} \in \mathcal O_{\mathfrak a}^*(U_\alpha \cap U_\beta)\}$.

Let $c(L)$ denote the image of $\delta(L)$ in $H^2(Z,\mathbb R_{\mathfrak a})$ under the natural homomorphism $H^2(Z,\mathbb Z_{\mathfrak a}) \rightarrow H^2(Z,\mathbb R_{\mathfrak a})$. We identify the last group with $H_{\mathfrak a}^2(Z)$, see \eqref{derhamiso}. Since ${\rm dim}_{\mathbb C}Z=1$, element $c(L)$ is determined by a $d$-closed $(1,1)$-form $\eta\in Z_{\mathfrak a}^2(Z)$.
%In fact, $c(L)$ is the image of an element $\eta(L) \in H^{1,1}_{\mathfrak a}(Z)$ (see, e.g.,~\cite{GH}) under the natural homomorphism $H^{1,1}_{\mathfrak a}(Z) \rightarrow H^2_{\mathfrak a}(Z) \cong H^2(Z,\mathbb R_{\mathfrak a})$ (cf.~(\ref{derhamiso})). 
%In what follows, we denote by $\eta:=\eta(L) \in Z^{1,1}_{\mathfrak a}(Z)$ a form that represents the corresponding cohomology class in $H^{1,1}_{\mathfrak a}(Z)$. 

\begin{lemma}
$\eta=d\lambda$ for some $\lambda \in \Lambda^1_{\mathfrak a}(Z)$.
\end{lemma}
\begin{proof}%[Proof of Lemma]
Since $Z$ is $1$-dimensional, $\bar{\partial} \eta=0$. Hence, by the analogue of Corollary \ref{dolcor2} on $Z$ we have $\eta=\bar{\partial} \lambda$ for some $\lambda \in \Lambda_{\mathfrak a}^{1,0}(Z)$. We have $\partial \lambda=0$, as $\Lambda_{\mathfrak a}^{2,0}(Z)=0$, so  $d\lambda=(\bar{\partial}+\partial)\lambda=\bar{\partial} \lambda=\eta$, as required.
\end{proof}

The lemma implies that $c(L)=0$.
Replacing cover $\{U_\alpha\}$ by its  refinement of class ($\mathcal T_{\mathfrak a}$), if necessary, we may assume without loss of generality that there exists a locally constant $1$-cochain $\{s_{\alpha\beta}\in C_{\mathfrak a}(U_\alpha \cap U_\beta, \mathbb R_{\mathfrak a})\}$ on $\{U_\alpha\}$ such that for all $\alpha,\beta,\gamma$
$$
m^L_{\alpha\beta\gamma}=s_{\alpha\beta}+s_{\beta\gamma}+s_{\gamma\alpha} \quad \text{on }\quad U_\alpha\cap U_\beta\cap U_\gamma.
$$
Then $\{\log c_{\alpha\beta}-2\pi i\cdot s_{\alpha\beta}\} \in Z^1(\{U_\alpha\},\mathcal O_{\mathfrak a})$. According to Theorem \ref{thmB} this cocycle represents 0 in $H^1(Z,\mathcal O_{\mathfrak a})$ (as $H^1(Z,\mathcal O_{\mathfrak a})=H^1(Y,\mathcal O)$, see~the discussion in subsection \ref{needs}.1). Again, passing to a refinement of cover $\{U_\alpha\}$ of class ($\mathcal T_{\mathfrak a}$), if necessary, we may assume without loss of generality that this cocycle can be resolved on $\{U_\alpha\}$, that is, there exist $h_\alpha \in \mathcal O_{\mathfrak a}(U_\alpha)$ such that
$$
\log c_{\alpha\beta}-2\pi i\cdot s_{\alpha\beta}=h_\alpha-h_\beta \quad \text{ on }\quad U_\alpha \cap U_\beta .
$$
%(Indeed, by definition, $\mathcal O_{\mathfrak a}=\iota^*\mathcal O$, where $\mathcal O$ is the structure sheaf of the complex submanifold $Y:=\overline{\iota(Z)} \subset c_{\mathfrak a}X$ (cf.~Section \ref{submanifoldsect}), and $\iota|_Z:Z \rightarrow Y$ is the canonical map (cf.~Section \ref{mainsect}). Since $\iota(Z)$ is dense in $Y$, sheaves $\mathcal O_{\mathfrak a}$ and $\mathcal O$ are isomorphic. This implies that $H^1(Z,\mathcal O_{\mathfrak a})=H^1(Y,\mathcal O)$, so we can apply Theorem \ref{thmB_}.)
We set $d_{\alpha\beta}:=e^{-h_\alpha} c_{\alpha\beta}e^{h_\beta}$ on $U_\alpha\cap U_\beta$. Then cocycle $\{d_{\alpha\beta}\}$ determines a discrete line $\mathfrak a$-bundle $L'$ isomorphic to $L$.
Therefore, condition (1) of Theorem \ref{zeroapthm} is satisfied.

\medskip

Now, we show that under the additional hypothesis $H^1(Z,\mathbb C)=0$ condition (2) of Theorem \ref{zeroapthm} is satisfied as well.  

First, note that $H^2(Z,\mathbb Z)=0$. Indeed, $Z$ is a complex submanifold of a Stein manifold $X$, and hence itself is a Stein manifold. Therefore, since $\dim_{\mathbb C}Z=1$, $Z$ is homotopically equivalent to a 1-dimensional CW-complex, which implies the required.

A discrete line bundle on $Z$ is determined (up to an isomorphism in the corresponding category) by an element of group $H^1(Z,\mathbb C^*)$, where $\mathbb C^*:=\mathbb C \setminus \{0\}$. Therefore, to show that the discrete line $\mathfrak a$-bundle $L'$ is trivial in the category of  discrete bundles on $Z$, it suffices to show that $H^1(Z,\mathbb C^*)=0$. 
In turn, the exact sequence of locally constant sheaves $0 \rightarrow \mathbb Z \rightarrow \mathbb C \overset{\exp}{\rightarrow} \mathbb C^* \rightarrow 0$ on $Z$ induces an exact sequence of cohomology groups
$$
\dots \rightarrow H^1(Z,\mathbb Z) \rightarrow H^1(Z,\mathbb C) \rightarrow H^1(Z,\mathbb C^*) \rightarrow H^2(Z,\mathbb Z) \rightarrow \cdots .
$$
Since $H^1(Z,\mathbb C)=H^2(Z,\mathbb Z)=0$, group $H^1(Z,\mathbb C^*)=0$, as required.
\end{proof}

\SkipTocEntry\subsection{Proof of Theorem \ref{chernthm}}
\label{divisors2sect}
Since $X_0$ is homotopy equivalent to open subset $Y_0\subset X_0$, $\pi_1(X_0)=\pi_1(Y_0)$ and the space $c_{\mathfrak a} X$ is homotopy equivalent to open subset $c_{\mathfrak a}Y\subset c_{\mathfrak a}X$, $Y:=p^{-1}(Y_0)\subset X$ (for $\mathfrak a=\ell_\infty$ the proof is given in \cite[Prop.~4.2]{Br8}; the proof in the general case repeats it word-for-word).

We retain notation of subsection \ref{sub15.3}. For the exact sequence of locally constant sheaves on $X$
\begin{equation*}
0 \rightarrow \mathbb Z_{\mathfrak a} \rightarrow \mathcal O_{\mathfrak a} \overset{e^{2\pi i \cdot }}{\rightarrow} \mathcal O_{\mathfrak a}^* \rightarrow 0
\end{equation*}
consider the induced exact sequences of cohomology groups
\[
\cdots \rightarrow  H^1(X,\mathbb Z_{\mathfrak a}) \rightarrow  H^1(X,\mathcal O_{\mathfrak a}) \rightarrow  H^1\bigl(X,\mathcal O^*_{\mathfrak a}) \overset{\delta}{\rightarrow } H^2(X,\mathbb Z_{\mathfrak a}) \rightarrow  \cdots . 
\] 
We have similar exact sequences over $Y=p^{-1}(Y_0)$ so that the embedding $Y\hookrightarrow X$ induces a commutative diagram of these sequences.
%let $\delta_{Y}:H^1\bigl(Y,\mathcal O^*_{\mathfrak a}) \to H^2(Y,\mathbb Z_{\mathfrak a})$ denote the corresponding morphism.

Since $X_0$ is a Stein manifold, by Theorem \ref{thmB}
$H^1(X,\mathcal O_{\mathfrak a})=H^1(c_{\mathfrak a}X,\mathcal O)=0$. Thus, $\delta$ is an injection. Also, since  $c_{\mathfrak a} X$ is homotopy equivalent to $c_{\mathfrak a}Y$,
by the homotopy invariance for cohomology of locally constant sheaves (see, e.g., \cite[Ch.~II.11]{Bre}), $H^k(Y,\mathbb Z_{\mathfrak a})=H^k(c_{\mathfrak a}Y,\mathbb Z) \cong H^{k}(c_{\mathfrak a}X,\mathbb Z)=H^k(X,\mathbb Z_{\mathfrak a})$, $k\ge 0$
(see~definitions of the corresponding cohomology groups in subsection \ref{needs}.1). 

Let $c_L \in H^1(X,\mathcal O^*_{\mathfrak a})$ be the cohomology class determined  by the line $\mathfrak a$-bundle $L=L_E$ of the $\mathfrak a$-divisor $E$. 
%Clearly, $c_L$ determines $E$ up to $\mathfrak a$-equivalence in $\Div_{\mathfrak a}(X)$. 
We show that $\delta(c_L)=0$; since $\delta$ is an injection, this would imply that $L$ is isomorphic to the trivial line $\mathfrak a$-bundle, and hence $E$ is $\mathfrak a$-equivalent to an $\mathfrak a$-principal divisor.

Indeed, the restriction $\delta(c_{L})|_Y \in H^2(Y,\mathbb Z_{\mathfrak a})$ of $\delta(c_L)$ to $Y$ is, by definition, the Chern class of the restriction $L|_Y$. Since $E|_Y$ is $\mathfrak a$-equivalent to an $\mathfrak a$-principal divisor on $Y$, the line $\mathfrak a$-bundle $L|_Y$ is isomorphic to the trivial line $\mathfrak a$-bundle in $\mathcal L_{\mathfrak a}(Y)$, so we have $\delta(c_{L})|_Y=0$. Since the restriction homomorphism $H^2(X,\mathbb Z_{\mathfrak a})\rightarrow H^2(Y,\mathbb Z_{\mathfrak a})$ is an isomorphism (see above), $\delta(c_L)=0$, as required.

\medskip

Let us prove the second assertion of the theorem. Assume that $\mathfrak a$ is such that $\hat{G}_{\mathfrak a}$ is a compact topological group and $j(G)\subset\hat{G}_{\mathfrak a}$ is a dense subgroup, and ${\rm supp}(E)\cap Y=\emptyset$. We retain notation and results of parts (1) and (2) of the proof of Theorem \ref{divrelprop}. By definition divisor $E$ determines a holomorphic line bundle $\hat{L}_E$ on $c_{\mathfrak a}X$ and a holomorphic section $s$ of $\hat{L}_E$ such that $s$ is not identically zero on each `slice` $\iota_H(X_H)\subset c_{\mathfrak a}X$ and $\iota^*\hat{L}_E=L_E$. If ${\rm supp}(s)\subset c_{\mathfrak a}X$ is zero loci of $s$, then $\iota^{-1}({\rm supp}(s))={\rm supp}(E)$. Let us show that ${\rm supp}(s)\cap c_{\mathfrak a}Y=\emptyset$. Indeed, assuming the contrary we find a point $x\in\iota_H(X_H)\cap c_{\mathfrak a}Y$ for some $H\in\Upsilon$ such that $s(x)=0$. Since $c_{\mathfrak a}Y\subset c_{\mathfrak a}X$ is open, there exists an open neighbourhood of $x$ which is contained in $c_{\mathfrak a}Y$. Without loss of generality we may identify this neighbourhood with $V_0\times K$, where $V_0\subset Y_0$ is an open coordinate chart and $K \subset \hat{G}_{\mathfrak a}$ is open, $K \in \mathfrak Q$ (see subsection \ref{holfuncsect}). Then $s(z,\eta)=0$ for some $(z,\eta)\in V_0\times K$. Let $S\subset K$ be a dense subset such that $j^{-1}(S)\subset G$, the deck transformation group of $X$ (see subsection \ref{constrsect} for notation). By definition, $\iota^{-1}(V_0\times S)\subset Y$. Also, $s(\cdot,\xi)\in \mathcal O(V_0)$ for all $\xi\in K$ and $s(\cdot,\eta)$ is not identically zero. Since
$s\in C(V_0\times K)$, by the Montel theorem there exists a sequence $\{s(\cdot,\xi_j)\}_{j\in\mathbb N}$, $\{\xi_j\}_{j\in\mathbb N}\subset S$, converging to $s(\cdot,\eta)$ uniformly on compact subsets of $V_0$. Then according to the Hurwitz theorem (on zeros of a sequence of univariate holomorphic functions uniformly converging to a nonidentically zero holomorphic function), there exists $(w,\xi_i)\in V_0\times S$ such that
$s(w,\xi_i)=0$. This implies that $\iota^{-1}((w,\xi_i))\in {\rm supp}(E)\cap Y$, a contradiction proving the required claim.

Thus we obtain that $s|_{c_{\mathfrak a}Y}$ is nowhere zero, i.e., $\hat{L}_{E}|_{c_{\mathfrak a}Y}$ is holomorphically trivial. In turn,
$\iota^*\bigl(\hat{L}_{E}|_{c_{\mathfrak a}Y}\bigr):=\bigl(L_E\bigr)|_Y=L_{E|_Y}$ is the trivial $\mathfrak a$-bundle on $Y$. Hence, the restriction of $E$ to $Y$
is $\mathfrak a$-equivalent to an $\mathfrak a$-principal divisor. The first part of the theorem then implies that $E$ is $\mathfrak a$-equivalent to an $\mathfrak a$-principal divisor on $X$.

The proof of the theorem is complete.

\section{Proofs of Theorem \ref{hypthm} and Proposition \ref{prop2.23}}
\label{hypsect}

\begin{proof}[Proof of Theorem \ref{hypthm}] We will use notation and results of subsection \ref{charts} and Example 4.4 in \cite{BK8}.

Using the axiom of choice we construct a (not necessarily continuous) right inverse $\lambda:\hat{G}_{\mathfrak a} \rightarrow \hat{G}_{\ell_\infty}$ to $\kappa$, i.e., $\kappa \circ \lambda=\Id$. Given a subset $K \subset G$,
by $\hat{K}_{\mathfrak a} \subset \hat{G}_{\mathfrak a}$ and $\hat{K}_{\ell_\infty} \subset \hat{G}_{\ell_\infty}$ we denote the closures of sets $j_{\mathfrak a}(K)$ and $j_{\ell_\infty}(K)$ in $\hat{G}_{\mathfrak a}$ and $\hat{G}_{\ell_\infty}$, respectively. 
%We denote by 
%$K^{\circ}_{\mathfrak a}$ and $K^{\circ}_{\ell^\infty}$ the interiors of $\hat{K}_{\mathfrak a}$ and $\hat{K}_{\ell_\infty}$, respectively.

For $\Pi(U_0,K):=\Pi_{\gamma_*}(U_0,K)$ we have a commutative diagram
\begin{equation}
\label{comdiag44}
\bfig
\node a1(0,0)[\Pi(U_0,K)]
\node a2(0,500)[\Pi(U_0,K)]
\node b1(1000,0)[\hat{\Pi}_{\ell_\infty}(U_0,\hat{K}_{\ell_\infty})]
\node b2(1000,500)[\hat{\Pi}_{\mathfrak a}(U_0,\hat{K}_{\mathfrak a})]
\node c2(2000,500)[\hat{\Pi}_{\ell_\infty}(U_0,\hat{K}_{\ell_\infty})]
\arrow[a1`a2;=]
\arrow[a1`b1;\Id \times j_{\ell_\infty}]
\arrow[b1`b2;\kappa]
\arrow[a2`b2;\Id \times j_{\mathfrak a}]
\arrow[b2`c2;\lambda]
%\arrow[b1`c2;=]
\efig
\end{equation}
All maps, except possibly for $\lambda$, are continuous.

%\subsection{Proof of Theorem \ref{hypthm}}

\medskip

We will need the following results.

\begin{lemma}
\label{lemA}
Under the hypotheses of the theorem there exists a unique function $\hat{f} \in \mathcal O\bigl(\hat{\Pi}_{\mathfrak a}(U_0,\hat{K}_{\mathfrak a}) \bigr)$ such that 
\begin{equation}
\label{repr5}
f|_{\Pi(U_0,K)}=(\Id \times j_{\mathfrak a})^* \hat{f}.
\end{equation}
\end{lemma}
\begin{proof}%[Proof of Lemma]
Since $f \in \mathcal O_{\ell_\infty}(X)$, there exists a function $\tilde{f} \in \mathcal O\bigl(\hat{\Pi}_{\ell_\infty}(U_0,\hat{K}_{\mathfrak a}) \bigr)$ such that $f|_{\Pi(U_0,K)}=(\Id \times j_{\ell_\infty})^* \tilde{f}.$ We set $\hat{f}:=(\Id \times \lambda)^*\tilde{f}:\hat{\Pi}_{\mathfrak a}(U_0,\hat{K}_{\mathfrak a}) \rightarrow \mathbb C.$ Clearly, (\ref{repr5}) is satisfied. Identifying $\hat{\Pi}_{\mathfrak a}(U_0,\hat{K}_{\mathfrak a})$ with $U_0 \times \hat{K}_{\mathfrak a}$ (see~(\ref{pi2})), we obtain that $\hat{f}(\cdot,\omega) \in \mathcal O(U_0)$ for all $\omega \in \hat{K}_{\mathfrak a}$. It remains to show that $\hat{f}$ is continuous.
Since $f\in C_{\mathfrak a}(Z)$, there exists a function $F \in C(\hat{\Pi}_{\mathfrak a}(Z_0,\hat{K}_{\mathfrak a}))$ such that $f|_{\Pi(Z_0,K)}=(\Id \times j_{\mathfrak a})^*F$. 
Also, since $(\Id \times j_{\mathfrak a})\bigl(\Pi(Z_0,K)\bigr)$ is dense in $\hat{\Pi}_{\mathfrak a}(Z_0,\hat{K}_{\mathfrak a})$ and diagram (\ref{comdiag44}) is commutative, 
\begin{equation}
\label{id700}
\hat{f}|_{\hat{\Pi}_{\mathfrak a}(Z_0,\hat{K}_{\mathfrak a})}=F.
\end{equation}
We identify $\hat{\Pi}_{\mathfrak a}(U_0,\hat{K}_{\mathfrak a})$ with $U_0 \times \hat{K}_{\mathfrak a}$, and $\hat{\Pi}_{\mathfrak a}(Z_0,\hat{K}_{\mathfrak a})$ with $Z_0 \times \hat{K}_{\mathfrak a}$. Suppose that $\hat{f}$ is discontinuous, i.e., there exists a  net $\{(z_\alpha,\omega_\alpha)\} \subset U_0\times\hat{K}_{\mathfrak a}$, $(z_\alpha,\omega_\alpha) \rightarrow (z,\omega) \in U_0\times\hat{K}_{\mathfrak a}$, such that $\lim_\alpha \hat{f}(z_\alpha,\omega_\alpha)$ exists but does not coincide with $\hat{f}(z,\omega)$.
Using the Montel theorem we find a subnet $\{\hat{f}(\cdot,\omega_{\alpha_\beta})\}$ of the net $\{\hat{f}(\cdot,\omega_\alpha)\}\subset \mathcal O(U_0)$ which converges to a function $g$. Since $\hat{f}|_{Z_0 \times \hat{K}_{\mathfrak a}}$ is continuous and $Z_0$ is a uniqueness set for functions in $\mathcal O(U_0)$, $g=\hat{f}(\cdot,\omega)$. But $g(z)=\lim_\alpha \hat{f}(z_\alpha,\omega_\alpha)$, a contradiction showing that
$\hat{f} \in C(\hat{\Pi}_{\mathfrak a}(Z_0,\hat{K}_{\mathfrak a}))$. Hence, $\hat{f} \in \mathcal O(\hat{\Pi}_{\mathfrak a}(Z_0,\hat{K}_{\mathfrak a}))$.
\end{proof}

%Let $\hat{L}_{\mathfrak a}$ denote the closure of $j_{\mathfrak a}(L)$ in $\hat{G}_{\mathfrak a}$. 

%By the assumption of the theorem, $\hat{L}_{\mathfrak a} \subset K^{\circ}_{\mathfrak a}$.

\begin{lemma}
\label{lemB}
We have $\cup_{i=1}^m \hat{L}_{\mathfrak a} \cdot j_{\mathfrak a}(g_i)=\hat{G}_{\mathfrak a}$.
\end{lemma}
(Recall that $\hat{L}_{\mathfrak a}$ is the closure of $j_{\mathfrak a}(L)$ in $\hat{G}_{\mathfrak a}$.)
\begin{proof}%[Proof of Lemma]
Indeed, $\cup_{i=1}^m \hat{L}_{\mathfrak a} \cdot j_{\mathfrak a}(g_i)$ is a closed subset of $\hat{G}_{\mathfrak a}$ containing $j_{\mathfrak a}(G)$, as by the assumption of the theorem $\cup_{i=1}^m  L \cdot g_i=G$. Since set $j_{\mathfrak a}(G)$ is dense in $\hat{G}_{\mathfrak a}$, we obtain that $\cup_{i=1}^m  \hat{L}_{\mathfrak a} \cdot j_{\mathfrak a}(g_i)=\hat{G}_{\mathfrak a}$.
\end{proof}

We now complete the proof of Theorem \ref{hypthm} by means of an analytic continuation-type argument.

%\medskip

Let us consider the open cover $\mathcal U=\{U_{0,\gamma}\}$ of $X_0$
and the corresponding system of trivializations $\psi_\gamma:p^{-1}(U_{0,\gamma} \times G) \rightarrow U_{0,\gamma} \times G$ of the covering $p:X_0\rightarrow X$ introduced in subsection \ref{sectbohr}. (Recall that
$\Pi_\gamma(U_{0,\gamma}, S):=\psi_\gamma^{-1}(U_0 \times S)$, $S \subset G$.) 
This system determines a system of trivializations $\bar{\psi}_\gamma: \hat{\Pi}_{\mathfrak a,\gamma}(U_{0,\gamma},L) \rightarrow U_{0,\gamma} \times L$, $L \subset \hat{G}_{\mathfrak a}$, of the fibrewise compactification $\bar{p}:c_{\mathfrak a}X\rightarrow X_0$, see~subsection \ref{charts}.
Passing to a refinement of $\mathcal U$, if necessary, we may and will assume without loss of generality that all nonempty sets $U_{0,\gamma}\cap U_{0,\delta}$ are connected and simply connected, and that
$U_0=U_{0,\gamma_*}$ for $\gamma_*$ from the statement of the theorem.

First, we will prove that

\medskip

(*)\quad {\em there exists a function $\hat{f}_{U_0}\in\mathcal O(\bar{p}^{-1}(U_0))$ such that $\iota^*\hat{f}_{U_0}=f|_{p^{-1}(U_0)}$.}

\medskip

Let us fix $1 \leq i \leq m$. 

\begin{lemma}
\label{cylinder}

There exist families $\{U_{0,\gamma_l}\}_{l=1}^{s(i)}\subset \mathcal U$ and $\{K_l\}_{l=1}^{s(i)} \subset G$ such that 

\begin{itemize}
\item[(1)]  $\gamma_1=\gamma_{s(i)}=\gamma_*$, $K_1:=K$ and $K_{s(i)}=K \cdot g_{i}$,

\item[(2)] $U_{0,\gamma_l} \cap U_{0,\gamma_{l+1}} \neq \varnothing$ for all $1\le l\le s(i)-1$, and

\item[(3)] $\Pi_{\gamma_l}(U_{0,\gamma_l} \cap U_{0,\gamma_{l+1}},K_l)=\Pi_{\gamma_{l+1}}(U_{0,\gamma_l} \cap U_{0,\gamma_{l+1}},K_{l+1})$ for all $1 \leq l \leq s(i)-1$.
\end{itemize}
\end{lemma}
\begin{proof}
Take $x_0 \in U_0$  and  define $y_0:=\psi_{\gamma_*}^{-1}(x_0,1)$.
Since covering $p:X \rightarrow X_0$ is regular, there exists
a continuous path joining $y_0$ and $g_{i}\cdot y_0$ obtained as the lift of a loop $\gamma_0: [0,1]\rightarrow X_0$ with basepoint $x_0$. Then there exist a partition $0=t_0<t_1<\dots < t_{s(i)}=1$ of $[0,1]$ and a family $\{U_{0,\gamma_l}\}_{l=1}^{s(i)}\subset \mathcal U$ such that \[
\gamma_0([0,1])\subset\cup_{i=1}^{s(i)}U_{0,\gamma_i};\ U_{0,\gamma_1}=U_{0,\gamma_{s(i)}}:=U_0~(=U_{0,\gamma_*});\ \gamma_0([t_i,t_{i+1}])\subset U_{0,\gamma_{i+1}}\ \forall i\ge 0.
\] 
Now, we define $K_1=K$ and
$K_{l+1}:=K_l \cdot c_{\gamma_l\gamma_{l+1}}$ for all $1 \leq l \leq s(i)-1$ (note that $U_{0,\gamma_l} \cap U_{0,\gamma_{l+1}} \neq \varnothing$ by our construction), where $\{c_{\delta\gamma}\}$ is the $1$-cocycle on $\mathcal U$ determining covering $p:X\rightarrow X_0$
(see subsection \ref{sectbohr}). Clearly, conditions (1)-(3) are satisfied.
\end{proof}

Further, using Lemma \ref{lemA} 
we can find a function  $\hat{f}_1 \in \mathcal O\bigl(\hat{\Pi}_{\mathfrak a,\gamma_1}(U_{0,\gamma_1},\hat{K}_{1\mathfrak a})\bigr)$ such that $\iota^*\hat{f}_1=f$ on $\Pi_{\gamma_1}(U_{0,\gamma_1},K_1)$.
%$\iota^{-1}\bigl(\hat{\Pi}_{\mathfrak a,\gamma_1}(U_{0,\gamma_1},\hat{K}_{\mathfrak a, 1})\bigr)$.
%we can extend $f|_{\Pi_{\gamma_*}(U_0,K)}$ to a (unique) function  $\hat{f}_1 \in \mathcal O\bigl(\hat{\Pi}_{\mathfrak a,\gamma_1}(U_{0,\gamma_1},K^{\circ}_{\mathfrak a, 1})\bigr)$. 
Since the open set $U_{0,\gamma_1} \cap U_{0,\gamma_2}\, ( \neq \varnothing)$ is a uniqueness set for functions in $\mathcal O(U_{0,\gamma_2})$, we can apply Lemma \ref{lemA}  to $f|_{\Pi_{\gamma_2}(U_{0,\gamma_2},K_2)}$ to find a function $\hat{f}_2 \in \mathcal O\bigl( \hat{\Pi}_{\mathfrak a,\gamma_2}(U_{0,\gamma_2},\hat{K}_{2\mathfrak a})\bigr)$ such that $\iota^*\hat{f}_2=f$ on $\Pi_{\gamma_2}(U_{0,\gamma_2},K_2)$. (Indeed, as the set $Z$ in the lemma we can take $\Pi_{\gamma_2}(V,K_2)$, where $V$ is a compact subset of 
$U_{0,\gamma_1} \cap U_{0,\gamma_2}$ with nonempty interior. Then $f=\iota^*\hat f_1$ on $Z$ and the continuous function $\hat f_1$ defined on compact subset $\hat{\Pi}_{\mathfrak a,\gamma_1}(U_{0,\gamma_1},\hat{K}_{1\mathfrak a})$ of $c_{\mathfrak a}X$ admits a continuous extension to $c_{\mathfrak a}X$ by the Tietze-Urysohn theorem. Thus, $f\in C_{\mathfrak a}(Z)$, as required in the lemma.)
We repeat this construction for $3 \leq l \leq s(i)$ to obtain functions $\hat{f}_l \in \mathcal O\bigl( \hat{\Pi}_{\mathfrak a,\gamma_{l}}(U_{0,\gamma_l}, \hat{K}_{l\mathfrak a})\bigr)$ such that $\iota^*\hat{f}_l=f$ on $\Pi_{\gamma_l}(U_{0,\gamma_l},K_l)$.
%$\hat{f}_{l-1}$ on $\hat{\Pi}_{\mathfrak a,\gamma_{l-1}}(U_{0,\gamma_{l-1}},K^\circ_{\mathfrak a, l-1}) \cap \hat{\Pi}_{\mathfrak a,l}(U_{0,l},K^\circ_{\mathfrak a,l})$ and such that 
%$\iota^*\hat{f}_l=f$ on $\iota^{-1}\bigl(\hat{\Pi}_{\mathfrak a,\gamma_l}(U_{0,\gamma_l},K^{\circ}_{\mathfrak a, l})\bigr)$. 
Using these arguments for all $1 \leq i \leq m$ we obtain functions $\hat{f}_{s(i)}\in \mathcal O\bigl(\hat{\Pi}_{\mathfrak a,\gamma_{*}}(U_{0,\gamma_*}, \hat{K}_{\mathfrak a} \cdot j_{\mathfrak a}(g_i))\bigr)$ such that $\iota^*\hat{f}_{s(i)}=f|_{\Pi_{\gamma_*}(U_{\gamma_*}, K \cdot g_i)}$. 

Let $K^{\circ}_{\mathfrak a}$ denote the interior of $\hat{K}_{\mathfrak a}$. Then $K^{\circ}_{\mathfrak a}\cdot j_{\mathfrak a}(g_i)$ is the interior of
$\hat{K}_{\mathfrak a} \cdot j_{\mathfrak a}(g_i)$ for all $i\ge 1$. We have $\hat{f}_{s(i)}=\hat{f}_{s(j)}$ on $\hat{\Pi}_{\mathfrak a,\gamma_{*}}(U_{0,\gamma_*}, K^\circ_{\mathfrak a} \cdot j_{\mathfrak a}(g_i))\cap \hat{\Pi}_{\mathfrak a,\gamma_{*}}(U_{0,\gamma_*}, K^\circ_{\mathfrak a} \cdot j_{\mathfrak a}(g_j))\ne\emptyset$ 
since by our construction these functions are continuous and coincide on dense subset 
$\iota\bigl(\Pi_{\gamma_*}(U_{0,\gamma_*},K\cdot g_i) \cap \Pi_{\gamma_*}(U_{0,\gamma_*},K\cdot g_j)\bigr)$ of the latter set.
Finally, by ~Lemma \ref{lemB} $\cup_{i=1}^m \hat{L}_{\mathfrak a} \cdot j_{\mathfrak a}(g_i)=\hat{G}_{\mathfrak a}$, 
and since $\hat{L}_{\mathfrak a} \subset K^{\circ}_{\mathfrak a}$ by the assumption of the theorem, we have $\cup_{i=1}^m K^\circ_{\mathfrak a} \cdot j_{\mathfrak a}(g_i)=\hat{G}_{\mathfrak a}$. This shows that $\bar{p}^{-1}(U_0)=\cup_{i=1}^m\hat{\Pi}_{\mathfrak a,\gamma_{*}}(U_{0,\gamma_*}, K^\circ_{\mathfrak a} \cdot j_{\mathfrak a}(g_i))$. Therefore $\hat{f}_{U_0}|_{\hat{\Pi}_{\mathfrak a,\gamma_{*}}(U_{0,\gamma_*}, K^\circ_{\mathfrak a} \cdot j_{\mathfrak a}(g_i))}:=\hat{f}_{s(i)}$, $1\le i\le m$, is a function in $\mathcal O(\bar{p}^{-1}(U_0))$ satisfying (*) as required.
By Proposition \ref{basicpropthm}(2) $f|_{p^{-1}(U_0)} \in \mathcal O_{\mathfrak a}(p^{-1}(U_0))$.

Let $W_0\subset X_0$ be the maximal connected open subset which consists of unions of elements of the cover $\mathcal U$ and such that $f|_{p^{-1}(W_0)} \in \mathcal O_{\mathfrak a}(p^{-1}(W_0))$. (Existence of $W_0$ follows from Zorn's lemma; also, $U_0\subset W_0$.) Let us show that $W_0=X_0$. Assuming the contrary, we find (because of connectedness of $X_0$) a subset $U_0'\in\mathcal U$ such that $W_0\cap U_0'\ne\emptyset$ and $W_0$ is a proper subset of the open connected set $W_0\cup U_0'$.
Now in conditions of Theorem \ref{hypthm} we replace $U_0$, $Z$ and $K$ by sets $U_0'$, $Z':=p^{-1}(Z_0')$, where $Z_0' \Subset U_0' \cap U_0$ is  compact with nonempty interior (hence, a uniqueness set for functions in $\mathcal O(U'_0)$), and $K':=G$, respectively. Since $f|_{p^{-1}(U_0)} \in \mathcal O_{\mathfrak a}(p^{-1}(U_0))$, we have $f|_{Z'} \in C_{\mathfrak a}(Z')$. Therefore claim (*) in this setting gives a function $\hat{f}_{U_0'}\in\mathcal O(\bar{p}^{-1}(U_0'))$ such that $\iota^*\hat{f}_{U_0'}=f|_{p^{-1}(U_0')}$, i.e., 
$f|_{p^{-1}(U'_0)} \in \mathcal O_{\mathfrak a}(p^{-1}(U'_0))$. Since $f|_{p^{-1}(W_0)} \in \mathcal O_{\mathfrak a}(p^{-1}(W_0))$, this implies that $f|_{p^{-1}(W_0\cup\, U_0')} \in \mathcal O_{\mathfrak a}(p^{-1}(W_0\cup U_0'))$ contradicting the maximality of $W_0$. Thus, 
$W_0=X_0$ and  $f \in \mathcal O_{\mathfrak a}(X)$. 

The proof of the theorem is complete.
\end{proof}

%\medskip

\begin{proof}[Proof of Proposition \ref{prop2.23}]
(a)$\Rightarrow$(b). Suppose that there exist $g_1,\dots, g_m\in G$ such that $\cup_{i=1}^m\, K\cdot g_i= G$. Let us show that the closure $\hat{K}_{\mathfrak a}$ of $j(K)$, $j:=j_{\mathfrak a}$, in $\hat{G}_{\mathfrak a}$ has a nonempty interior $\hat{K}^{\circ}_{\mathfrak a}$. Indeed, by Lemma \ref{lemB} $\cup_{i=1}^m\, \hat{K}_{\mathfrak a}\cdot j(g_i)=\hat{G}_{\mathfrak a}$. Assuming that $\hat{K}_{\mathfrak a}\neq \hat{G}_{\mathfrak a}$ (in this case the statement is trivial) we may choose $1\le k\le m-1$ such that $K'=\cup_{i=1}^{k}\, \hat{K}_{\mathfrak a}\cdot j(g_i)$ does not cover $\hat{G}_{\mathfrak a}$ but $\cup_{i=1}^{k+1}\, \hat{K}_{\mathfrak a}\cdot j(g_i)=\hat{G}_{\mathfrak a}$. Thus the complement of $K'$ is a nonempty open subset of $\hat{K}_{\mathfrak a}\cdot j(g_{k+1})$. This implies that $\hat{K}_{\mathfrak a}^\circ\ne\emptyset$.

(b)$\Rightarrow$(c). Let $U\subset\hat{G}_{\mathfrak a}$ be open. Since $j(G)$ is a dense subgroup of the compact topological group $\hat{G}_{\mathfrak a}$, the set $\cup_{g\in G}\,U\cdot j(g)$ coincides with $\hat{G}_{\mathfrak a}$. (For otherwise, there exists $v\in \hat{G}_{\mathfrak a}$ such that the closure in $\hat{G}_{\mathfrak a}$ of the set $\{v\cdot j(g)\}_{g\in G}$ is a proper subset of $\hat{G}_{\mathfrak a}$ which contradicts the density of $j(G)$ in $\hat{G}_{\mathfrak a}$.) Thus there exist $g_1,\dots, g_m\in G$ such that $\cup_{i=1}^m\,U\cdot j(g_i)=\hat{G}_{\mathfrak a}$. This implies that $\cup_{i=1}^m\,j^{-1}(U)\cdot g_i=G$. 

Now, suppose that $K\subset G$ is such that $\hat{K}^{\circ}_{\mathfrak a}\neq\emptyset$. Choose an open set $U\Subset \hat{K}^{\circ}_{\mathfrak a}$ and define $L:=j^{-1}(U)\subset G$. The previous argument shows that the pair $L\subset K$ satisfies conditions of Theorem \ref{hypthm}.

(c)$\Rightarrow$(a). Follows from the definitions.
\end{proof}

\section{Proofs of Proposition \ref{iso1prop} and Theorems \ref{interp}, \ref{hartogsthm} and \ref{approxthm}}

\label{isosect}

\SkipTocEntry\subsection{Proof of Proposition \ref{iso1prop}}

%The proofs of our results are based on the equivalences established in Propositions \ref{iso1prop} and \ref{basicpropthm}, so we prove these propositions first.
%
%
%
%\begin{proof}[Proof of Proposition \ref{iso1prop}]

It is easy to see that any function $f \in \mathcal O_{\mathfrak a}(X)$ is locally Lipschitz with respect to the semi-metric $d$ (see Introduction), i.e.,
\begin{equation}
\label{uc2}
|f(x_1,g)-f(x_2,g)| \leq C d\bigl((x_1,g),(x_2,g)\bigr):= Cd_0(x_1,x_2)
\end{equation}
for all $(x_1,g)$, $(x_2,g) \in W_0 \times G \cong p^{-1}(W_0)$, where $W_0 \Subset X_0$ is a simply connected coordinate chart. (Here $C$ depends on $d_0$ and $W_0$ only.) We set $f_{x_0}:=f|_{p^{-1}(x_0)} \in \mathfrak a$, $x_0\in X_0$, and define
\begin{equation*}
\tilde{f}(x_0):=f_{x_0}, \quad x_0 \in X_0.
\end{equation*}
Then $\tilde{f}$ is a section of bundle $C_{\mathfrak a}X_0$. Using \eqref{uc2} for any linear functional $\varphi\in \mathfrak a^*$ we have $\varphi(\tilde f(x)(g)):=\varphi(f(x,g))\in\mathcal O(W_0)$, $g\in G$, $x\in W_0\Subset X_0$, a simply connected coordinate chart, see \cite{Lin} or \cite{Br4} for similar arguments. Thus $\tilde f$ is a holomorphic section of $C_{\mathfrak a}X_0$. Reversing these arguments we obtain that any holomorphic section of $C_{\mathfrak a}X_0$ determines a holomorphic $\mathfrak a$-function on $X$.

\SkipTocEntry\subsection{Proof of Theorem \ref{interp}}

Let $B$ be a (complex) Banach space. We define
$$\mathcal A(D_0,B):=C(\bar{D}_0,B) \cap \mathcal O(D_0,B).$$
Consider a family of bounded linear operators $\mathcal L^B_z:B \rightarrow \mathcal A(D_0,B)$, $z \in D_0$, holomorphic in $z$ such that $\mathcal L_z^B(b)=b$ for every $b \in B$ and $\sup_{z \in D_0}\|\mathcal L^B_z\|=1$ defined by the formula $$\mathcal L^B_z(b)(x):=b \quad \text{for all } x \in D_0.$$
We use notation and results of subsection \ref{nonselfadj}.1. Namely, we identify functions in algebra $\mathfrak a_z$, $z\in D_0$, with sections over $z$ of the holomorphic Banach vector bundle $\tilde p:C_{\mathfrak a}X_0\rightarrow X_0$ associated to the principal fibre bundle $p:X\rightarrow X_0$ and having fibre $\mathfrak a$, and functions in $\mathcal O_{\mathfrak a}(D)$ with holomorphic sections of $\mathcal O(C_{\mathfrak a}X_0)|_{D_0}$. Recall that there is a holomorphic Banach vector bundle $E$ such that
$
C_{\mathfrak a}X_0 \oplus E=X_0 \times B
$
for some Banach space $B$. By $q:X_0 \times B \rightarrow C_{\mathfrak a}X_0$ and $i:C_{\mathfrak a}X_0 \rightarrow X_0 \times B$, $q \circ i=\Id$, we denote the corresponding bundle morphisms.
Now, for every $h \in \mathfrak a_z$ we define 
$$
L_z(h):=(q \circ \mathcal L_z^B \circ i)(h)\in \mathcal A_{\mathfrak a}(D).
$$
Clearly, the family $\{L_z\}_{z \in D_0}$ satisfies conditions (1), (2) of Theorem \ref{interp}.

\SkipTocEntry\subsection{Proof of Theorem \ref{hartogsthm}}

The arguments below are analogous to those in \cite{Br3}.

Using the construction of subsection \ref{nonselfadj}.1 we identify functions in $C_{\mathfrak a}(X)$ and $\mathcal O_{\mathfrak a}(X)$ with continuous and holomorphic sections of the holomorphic Banach vector bundle $\tilde p:C_{\mathfrak a}X_0\rightarrow X_0$ associated to the principal fibre bundle $p:X\rightarrow X_0$ and having fibre $\mathfrak a$. 
Further, there exist holomorphic Banach vector bundles $p_1:E_1 \rightarrow X_0$ and $p_2:E_2 \rightarrow X_0$ with fibres $B_1$ and $B_2$ such that $E_2=E_1 \oplus C_{\mathfrak a}X_0$ and $E_2$ is holomorphically trivial, i.e., $E_2 \cong X_0\times B_2$ (see, e.g.,~\cite{Obz}); so  continuous and holomorphic sections of $E_2$ can be identified with $B_2$-valued continuous and holomorphic functions on $X_0$.
By $q:E_2 \rightarrow C_{\mathfrak a}X_0$ and $i:C_{\mathfrak a}X_0 \rightarrow E_2$ we denote the corresponding quotient and embedding homomorphisms of the bundles so that 
$q \circ i=\Id$. 

As before we identify function $f$ satisfying the hypothesis of Theorem \ref{hartogsthm} with a continuous section of $C_{\mathfrak a}X_0$ over
$\partial D_0$. Then $h:=i(f) \in C(\partial D_0,B_2)$. 
Since $f\in C_{\mathfrak a}(\partial D)$ satisfies the tangential Cauchy-Riemann equations, 
$h$ satisfies the weak tangential Cauchy-Riemann equations on $\partial D_0$: $$\int_{\partial D_0} (\varphi\circ h) \bar{\partial }\omega =0$$
for any smooth form $\omega \in \Lambda^{n,n-2}(X_0)$ having compact support and any $\varphi\in B_2^*$.
Hence, applying the Hartogs-type theorem of \cite{HarvLaw} to functions $\varphi\circ h$ we obtain that there exists a function $H \in \mathcal O(D_0,B^{**}_2) \cap C(\bar{D}_0,B^{**}_2)$, where the second dual $B^{**}_2$ of $B_2$ is considered with weak* topology, such that $H|_{\partial D_0}=h$ (here $B_2$ is naturally identified with its isometric copy in $B_2^{**}$). 

Now, we use the integral representation result of \cite[Corollary 5.4]{Gle} asserting that
there exist a compact subset $S \subset \bar{D}_0 \setminus D_0$, a positive Radon measure $\mu$ on $S$
and a function $Q$ on $D_0 \times S$ such that 
(a) $Q(\cdot,y)$ is holomorphic for all $y\in S$; (b) $Q(x,\cdot)$ is $\mu$-integrable for all $x\in D_0$; (c) $x \mapsto \int_S|Q(x,y)|d\mu(y)$ is continuous; (d)
 for any function $w \in \mathcal O(D_0) \cap C(\bar{D}_0)$ 
$$
w(x)=\int_S Q(x,y)f(y)d\mu(y) \quad \text{ for all } x \in D_0.
$$
Using the Bochner integration we define
\begin{equation}
\label{hprime}
H'(x):=\int_M Q(x,y)h(y)d\mu(y), \quad x \in D_0.
\end{equation}
Then $H' \in C(D_0,B_2)$. Since the Bochner integral commutes with the action of bounded linear functionals,  $\varphi\circ H'=\varphi\circ H$ on $D_0$ for all $\varphi \in B_2^*$. Thus, $H'=H$ on $D_0$ and so $H \in \mathcal O(D_0,(B_2,w)) \cap C(\bar{D}_0,(B_2,w))$, where $(B_2,w)$ is $B_2$ equipped with weak topology, and $H \in \mathcal O(D_0,B_2)$.

Now, the required holomorphic extension of $f$ is given by $F:=q(H')$. Indeed, by our construction $F|_{D_0} \in \mathcal O(D_0,C_{\mathfrak a}X_0)$. By Proposition \ref{iso1prop} $F|_{D_0}$ can be viewed as a function in $\mathcal O_{\mathfrak a}(D)$. Further, since map $q$ is  continuous also if we equip fibres of the corresponding bundles with weak topologies, $F$ is a continuous section of $(C_{\mathfrak a}X_0,w)$ over  $\bar{D}_0$, i.e., of $C_{\mathfrak a}X_0$ with fibres endowed with weak topology. 
Using presentation (\ref{extrem3}) of $C_{\mathfrak a}X_0$ and evaluation functionals at points of $G$, we easily obtain from the weak continuity
of $F$ that considered as a function on $X$ it is continuous up to the boundary. Hence, $F \in \mathcal O_{\mathfrak a}(D)\cap C(\bar{D})$ and $F|_{\partial D}=f$, as required.

% where $\tilde{p}^{-1}(U_{0,\gamma})$ is identified with $U_{0,\gamma} \times \mathfrak a$. 
%Then for every $g \in G$ and each $\gamma$ we have $\delta_g(F|_{\tilde{p}^{-1}(U_{0,\gamma})}) \in C(U_{0,\gamma})$,
%where $\delta_g \in \mathfrak a^*$ is defined by $\delta_g(r):=r(g)$ ($r \in \mathfrak a$). In particular, $\delta_g(F|_{\tilde{p}^{-1}(U_{0,\gamma} \cap \bar{D}_0)}) \in C(U_{0,\gamma} \cap \bar{D}_0)$. Since $\{U_{0,\gamma} \cap \bar{D}_0\}$ forms an open cover of $\bar{D}_0$ and $g \in G$ was chosen arbitrarily, it follows that $F$ also determines a continuous function on $\bar{D}$. 
%Hence, $F \in \mathcal O_{\mathfrak a}(D)\cap C(\bar{D})$ and $F|_{\partial D}=f$, as required.

\SkipTocEntry\subsection{Proof of Theorems \ref{approxthm}}

Let $D_0$ be a relatively compact subdomain of $X_0$, $D:=p^{-1}(D_0)$. We set $\mathcal A_{\mathfrak a}(D):=\mathcal O_{\mathfrak a}(D) \cap C_{\mathfrak a}(\bar{D})$.
By $\mathcal A_\iota(D)$ we denote the space of holomorphic functions $f \in \mathcal A_{\mathfrak a}(D)$ such that 
for every $x_0 \in \bar{D}_0$ the function $g \mapsto f(g \cdot x)$ ($g \in G$, $x \in p^{-1}(x_0)$) is in $\mathfrak a_\iota$,
and by $\mathcal A_0(D)$ the $\mathbb C$-linear hull of spaces $\mathcal A_\iota(D)$, $\iota \in I$. 

Theorem \ref{approxthm} is a corollary of the following result.

\begin{theorem}
\label{approxthm0}
If $X_0$ is a Stein manifold and $D_0 \subset X_0$ is a strictly pseudoconvex domain, then $\mathcal A_{\trig}(D)$ is dense in $\mathcal A_{\mathfrak a}(D)$.
\end{theorem}

First, we deduce Theorem \ref{approxthm} from Theorem \ref{approxthm0} and then prove the latter.

By $C_{\mathfrak a_{\iota}}X_0$ ($\iota \in I$) we denote the holomorphic Banach vector bundle associated to the principal fibre bundle $p:X \rightarrow X_0$ and having fibre $\mathfrak a_\iota$ (see~(\ref{extrem3})). 
For a given open subset $D_0 \subset X_0$ by  $\mathcal O(D_0,C_{\mathfrak a_\iota}X_0)$ we denote the space of holomorphic sections of bundle $C_{\mathfrak a_{\iota}}X_0$ over $D_0$ endowed with the topology of uniform convergence on compact subsets of $D_0$ which makes it a Fr\'{e}chet space.
We have an isomorphism of Fr\'{e}chet spaces
\begin{equation}
\label{isoo}
\mathcal O_{\mathfrak a_\iota}(D) \overset{\cong}{\rightarrow} \mathcal O(D_0,C_{\mathfrak a_\iota}X_0)
\end{equation}
(the proof repeats literally that of Proposition \ref{iso1prop}).

Let $X_0$ be a Stein manifold, $Y_0 \Subset X_0$ be open such that $\bar{Y}_0$ is holomorphically convex, and $D_0 \subset X_0$ be an open neighbourhood of $\bar{Y}_0$. We set $Y:=p^{-1}(Y_0)$.

\begin{proposition}
\label{propapprox2}
Let $f \in \mathcal O_{\mathfrak a_\iota}(D)$.
For every $\varepsilon>0$ there exists $h \in \mathcal O_{\mathfrak a_\iota}(X)$ such that $\sup_{z \in Y}|f(z)-h(z)|<\varepsilon$.
\end{proposition}

\begin{proof}
We need the following approximation result established in \cite[Theorem~C]{Bu2}. 

Let $B$ be a complex Banach space and $\mathcal O(X_0,B)$ the space of $B$-valued holomorphic functions on $X_0$.

\begin{itemize}
\item[($\ast$)]
\textit{Let $\hat{f} \in \mathcal O(D_0,B)$. For every $\varepsilon>0$ there exists $\hat{h}\in \mathcal O(X_0,B)$ such that $\sup_{z \in Y_0}\|\hat{f}(z)-\hat{h}(z)\|_B<\varepsilon$.}
\end{itemize}
%\end{lemma} 

Further, since $X_0$ is a Stein manifold,
there exist holomorphic Banach vector bundles $p_1:E_1 \rightarrow X_0$ and $p_2:E_2 \rightarrow X_0$ with fibres $B_1$ and $B_2$ such that $E_2=E_1 \oplus C_{\mathfrak a_\iota}X_0$ and $E_2$ is holomorphically trivial, i.e., $E_2 \cong X_0\times B_2$ (cf. the proof of Theorem \ref{hartogsthm}). 
Thus, any holomorphic section of $E_2$ can be naturally identified with a $B_2$-valued holomorphic function on $X_0$. 
By $q:E_2 \rightarrow C_{\mathfrak a_\iota}X_0$ and $i:C_{\mathfrak a_\iota}X_0 \rightarrow E_2$ we denote the corresponding quotient and embedding homomorphisms of these bundles so that 
$q \circ i=\Id$.
Given a function $f \in \mathcal O_{\mathfrak a_\iota}(D)$ by  $\hat{f} \in \mathcal O(D_0,C_{\mathfrak a_\iota}X_0)$ we denote its image under isomorphism (\ref{isoo}).
Set $\tilde{f}:=i(\hat{f}) \in \mathcal O(D_0,B_2)$. By ($\ast$) for every $\tilde{\varepsilon}>0$ there exists a function $\tilde{h} \in \mathcal O(X_0,B_2)$ such that $\sup_{z \in Y_0}\|\tilde{f}(z)-\tilde{h}(z)\|_{B_2}<\tilde{\varepsilon}.$
We define $\hat{h}:=q(\tilde{h}) \in \mathcal O(X_0,C_{\mathfrak a_\iota}X_0)$ and by $h \in \mathcal O_{\mathfrak a_\iota}(X)$ denote the image of $\hat{h}$ under the inverse isomorphism  in (\ref{isoo}). The continuity of $i$ and $q$ and the compactness of $\bar{Y}_0$ now imply that $\sup_{z \in Y}|f(z)-h(z)|<C\tilde{\varepsilon}$
for some $C>0$ independent of $\hat{f}$ and $\tilde{\varepsilon}$.
\end{proof}

Using this proposition we complete the proof of Theorem \ref{approxthm} as follows. 

Let $f \in \mathcal O_{\mathfrak a}(X)$. 
It suffices to show that for a sequence $Y_{0,1}\Subset\cdots\Subset Y_{0,k}\Subset\cdots$ of open subsets of $X_0$ such that $\cup_{k}Y_{0,k}=X_0$ and for any $\varepsilon>0$ there exist functions $h_k \in \mathcal O_{0}(X)$ such that $\sup_{x \in Y_{k}}|f(x)-h_k(x)|<\frac{\varepsilon}{k}$, where $Y_k:=p^{-1}(Y_{0,k})$.
Since $X_0$ is a Stein manifold, we may assume without loss of generality that each $\bar{Y}_{0,k}$ is holomorphically convex. 
Then there is a strictly pseudoconvex open neighbourhood $D_{0,k} \Subset X_0$ of $\bar{Y}_{0,k}$, $k \geq 1$ (see, e.g., \cite{HL}). Since restriction $f|_{\bar{D}_k} \in \mathcal A_{\mathfrak a}(D_k)$, $D_k:=p^{-1}(D_{0,k})$, by Theorem \ref{approxthm0} there exist functions $h_k' \in \mathcal A_{0}(D_k)$ such that 
$\sup_{x \in D_k}|f(x)-h_k'(x)|<\frac{\varepsilon}{2k}$, $k\ge 1$. By the definition of space $\mathcal A_{0}(D_k)$, there exists $\iota_k \in I$ such that $h_k' \in \mathcal A_{\iota_k}(D_k)$. Now, by Proposition \ref{propapprox2} there exists $h_k \in \mathcal O_{\mathfrak a_{\iota_k}}(X)$ such that $\sup_{x \in Y_k}|h_k'(x)-h_k(x)|<\frac{\varepsilon}{2k}$. Thus, $\sup_{x \in Y_k}|f(x)-h_k(x)|<\frac{\varepsilon}{k}$. Since $\mathcal O_{\mathfrak a_{\iota_k}}(X) \subset \mathcal O_{0}(X)$, this implies the required result modulo Theorem \ref{approxthm0}.

\begin{proof}[Proof of Theorem \ref{approxthm0}]
By $\mathcal A(D_0,C_{\mathfrak a}X_0)$ and $\mathcal A(D_0,C_{\mathfrak a_\iota}X_0)$ we denote spaces of sections of bundles $C_{\mathfrak a}X_0|_{\bar{D}_0}$ and $C_{\mathfrak a_\iota}X_0|_{\bar{D}_0}$ continuous over $\bar{D}_0$ and holomorphic on $D_0$.
We equip $\mathcal A(D_0,C_{\mathfrak a}X_0)$ with norm $\|f\|:=\sup_{x \in \bar{D}_0}\|f(x)\|_{\mathfrak a}$ which makes it a Banach space. Then $\mathcal A(D_0,C_{\mathfrak a_\iota}X_0)$ is a closed subspace of $\mathcal A(D_0,C_{\mathfrak a}X_0)$.
Also, we define linear space $\mathcal A_0(D_0,C_{\mathfrak a}X_0):=\bigcup_{\iota \in I}\mathcal A(D_0,C_{\mathfrak a_\iota}X_0)$.
We have natural isomorphisms of vector spaces defined similarly to that of Proposition \ref{iso1prop}:
\begin{equation}
\label{iso2prop}
\mathcal A_{\mathfrak a_\iota}(D) \overset{\cong}{\rightarrow} \mathcal A(D_0,C_{\mathfrak a_\iota}X_0), \quad \mathcal A_0(D)  \overset{\cong}{\rightarrow} \mathcal A_0(D_0,C_{\mathfrak a}X_0)
\end{equation}
(the proof is analogous to the proof of Proposition \ref{iso1prop}).
In view of (\ref{iso2prop}), Theorem \ref{approxthm0} can be restated as follows: 

\medskip

\begin{enumerate}
\item[($\ast$)]\textit{Suppose that $X_0$ is a Stein manifold and $D_0 \Subset X_0$ is a strictly pseudoconvex subdomain. Then $\mathcal A_0(D_0,C_{\mathfrak a}X_0)$ is dense in $\mathcal A(D_0,C_{\mathfrak a}X_0)$.}
\end{enumerate}

\medskip

For the proof of this claim we need the following results.

As before, we define $\mathcal A(D_0,B):=\mathcal O(D_0,B) \cap C(\bar{D}_0,B)$ and endow this space 
with norm $\|f\|_{\bar{D}_0}:=\sup_{x \in \bar{D}_0}\|f(x)\|_B$.
The next result follows easily from a similar result in \cite{HL} (in case $B=\mathbb C$) since all integral formulas and estimates
used in its proof are preserved when passing to Banach-valued forms.

\begin{lemma}
\label{hl2}
Let $K \subset \mathcal A(D_0,B)$ be compact. 
For every $\varepsilon>0$ there exist an open neighbourhood $D_0' \Subset X_0$ of $\bar{D}_0$ and a bounded linear operator $A_{K,\varepsilon}=A_{D_0,K,\varepsilon} \in \mathcal L\bigl(\mathcal A(D_0,B), \mathcal A(D_0',B)\bigr)$ such that  $\|f-Af|_{\bar{D}_0}\|_{\bar{D}_0}<\varepsilon$ for each $f \in K$.
\end{lemma}

We prove assertion ($\ast$) in three steps.

\medskip

(\textit{1}) Let $f \in \mathcal A(D_0,C_{\mathfrak a}X_0)$.
Using the construction of subsection \ref{nonselfadj}.1 (cf. the proof of Theorem \ref{hartogsthm}) and Lemma \ref{hl2} we may assume without loss of generality that $f=f'|_{\bar{D}_0}$, where $f' \in \mathcal O(D_0',C_{\mathfrak a}X_0)$ and $D_0' \Subset X_0$ is an open neighbourhood of $\bar{D}_0$.

We have to show that for every $\varepsilon>0$ there exists a section $F \in \mathcal A_0(D_0,C_{\mathfrak a}X_0)$ such that $\sup_{x \in \bar{D}_0}\|f(x)-F(x)\|_{\mathfrak a}<\varepsilon$.% where $\|\cdot\|_{\mathfrak a}$ stands for the norm in algebra $\mathfrak a$.

\medskip

(\textit{2}) Let $\mathcal U=\{U_k\}_{k=1}^M$, where each $U_k \Subset D_0'$ is open biholomorphic to a polydisk in $\mathbb C^n$, and $D_0 \Subset \cup_{k=1}^M U_k$.
% so that this biholomorphism admits an extension to a homeomorphism between $\bar{U}_k$ and the closed polydisk, and $D_0 \Subset \cup_{k=1}^M U_k$.

\begin{lemma}
\label{lem2}
For every $\varepsilon>0$ there exist a subspace $\mathfrak a_{\iota_\varepsilon} \subset \mathfrak a$ {\rm (}$\iota_\varepsilon \in I${\rm)} and sections $F_{\varepsilon,k} \in \mathcal A\bigl(U_k,C_{\mathfrak a_{\iota_\varepsilon}}X_0 \bigr)$ such that
\begin{equation}
\label{feps}
\|f'(x)-F_{\varepsilon,k}(x)\|_{\mathfrak a}<\varepsilon \quad \text{ for all }\quad x \in U_k, \quad 1 \leq k \leq M.
\end{equation}
\end{lemma}

\begin{proof}
Since each $U_k$, $1 \leq k \leq M$, is simply connected, bundles $C_{\mathfrak a}X_0$, $C_{\mathfrak a_\iota}X_0$ ($\iota \in I$) admit holomorphic trivializations over $U_k$. Using these trivializations we identify sections of these bundles over $U_k$ with $\mathfrak a$-valued and $\mathfrak a_{\iota}$-valued functions on $U_k$.

By our assumption, for every $1 \leq k \leq M$ there exists a biholomorphism between $U_k$ and an open polydisk $\Delta \subset \mathbb C^n$ centered at $0$.
%, that admits extension to a homeomorphism between $\bar{U}_k$ and the closure $\bar{\Delta}$.
Without loss of generality we may assume that $f'_k:=f'|_{U_k}$ is defined over an open neighbourhood of $\bar{\Delta}$.
Then $f'_k$ can be identified by means of the corresponding holomorphic trivialization of bundle $C_{\mathfrak a}X_0$ with a holomorphic $\mathfrak a$-valued function defined on an open neighbourhood of $\bar{\Delta}$. 
 
For a given function $h \in \mathcal O\bigl(\Delta,\mathfrak a\bigr)$ by $T^N_0 h$ we denote its Taylor polynomial of degree $N$ at $x=0$.
Choose $N$ so large that 
\begin{equation*}
\|f'_{k}(x)-T^N_0 f'_k(x)\|_{\mathfrak a}<\frac{\varepsilon}{2} \quad \text{ for all }\quad x \in \Delta, \quad  1 \leq k \leq M,
\end{equation*}
where $T^N_0 f'_k(x):=\sum_{|\alpha| \leq N} a_{k,\alpha} x^\alpha,$ $a_{k,\alpha} \in \mathfrak a$, and $\alpha$ is a multi-index.
Since $\mathfrak a_0$ is dense in $\mathfrak a$, for every $\delta>0$ and all $1 \leq k \leq M$, $|\alpha| \leq N$, there exist $a_{k,\alpha}^\varepsilon \in \mathfrak a_0$ such that $\|a_{k,\alpha}-a_{k,\alpha}^\varepsilon\|_{\mathfrak a}<\delta.$
We choose $\delta>0$ to be sufficiently small
so that 
$$\sup_{x\in\Delta}\left\|\sum_{|\alpha| \leq N} a_{k,\alpha} x^\alpha-F_{\varepsilon,k}(x)\right\|_{\mathfrak a}<\frac{\varepsilon}{2},$$
where 
$F_{\varepsilon,k}(x):=\sum_{|\alpha| \leq N} a^\varepsilon_{k,\alpha} x^\alpha.$
Therefore,
$$\|f'_{k}(x)-F_{\varepsilon,k}(x)\|_{\mathfrak a}<\varepsilon \quad \text{ for all }\quad x \in \Delta, \quad 1 \leq k \leq M.$$
By definition, there exists $\iota_\varepsilon \in I$ such that $\mathfrak a_{\iota_\varepsilon}$ contains all $a^\varepsilon_{k,\alpha}$ ($1 \leq k \leq M$, $|\alpha|\leq N)$; hence 
$F_{\varepsilon,k} \in  \mathcal A\bigl(\Delta,\mathfrak a_{\iota_\varepsilon} \bigr)$.
%Now, we define $F_{\varepsilon,k}:=\tilde{F}_{\varepsilon,k} \circ \psi_k \in \mathcal A\bigl(U_k,\mathfrak a_{\iota_\varepsilon} \bigr)$. 
\end{proof}

(\textit{3}) We also need the following result.

\begin{lemma}
\label{lem3r}
In notation of Lemma \ref{lem2},
for every $\varepsilon>0$ there exists a section $F \in \mathcal A(D_0,C_{\mathfrak a_{\iota_\varepsilon}}X_0) \subset \mathcal A_0(D_0,C_{\mathfrak a}X_0)$ such that 
\begin{equation*}
\left\|F(x)-F_{\varepsilon,k}(x)\right \|_{\mathfrak a}<C\varepsilon \quad \text{ for all }\quad x \in U_k \cap \bar{D}_0, \quad 1 \leq k \leq M,
\end{equation*}
for some $C>0$ independent of section $f' \in \mathcal A_0(D'_0,C_{\mathfrak a}X_0)$ and $\varepsilon>0$.
\end{lemma}

\begin{proof}
There exists an open neighbourhood $D_0'' \Subset D_0'$ of $\bar{D}_0$ such that 
$D_0'' \Subset \cup_{k=1}^M U_k$. We may assume without loss of generality that $D_0''$ is strictly pseudoconvex. 
Let $B$ be  a complex space. By $\Lambda_b^{(0,q)}(D_0'',B)$, $q \geq 0$, we denote the space of bounded continuous $B$-valued $(0,q)$-forms on $D_0''$ endowed with norm $\|\cdot\|_{D_0'',B}^{(0,q)}$ defined by formula (\ref{approxsup}) with respect to local coordinates on the cover $\{U_k\}_{k=1}^M$ of $D_0''$.

Next, we define a holomorphic $1$-cocycle as follows. If $U_k \cap U_l \ne \varnothing$, then
$$g_{kl}:=F_{\varepsilon ,k}|_{U_k \cap U_l \cap D_0''}-F_{\varepsilon ,l}|_{U_k \cap U_l \cap D_0''} \in \mathcal A(U_k \cap U_l \cap D_0'',C_{\mathfrak a_{\iota_\varepsilon}}X_0),$$
and $g_{kl}:=0$ if $U_k \cap U_l \cap D_0''=\varnothing$.

Let $\{\rho_k\}_{k=1}^M \subset C^\infty(X_0)$ be a collection of nonnegative functions such that \penalty-10000 $\supp(\rho_k) \Subset U_k$, $1 \leq k \leq M$, and $\sum_{k=1}^m \rho_k \equiv 1$ on $\bar{D}_0''$. 

We set
%\begin{equation*}
$\tilde{g}_l:=\sum_{k=1}^M \rho_k g_{kl} \in C^\infty(U_l \cap D_0'')$
%\end{equation*}
so that $g_{kl}=\tilde{g}_k-\tilde{g}_l$ on $U_k \cap U_l \cap D_0''$. Then the family
$\{\bar{\partial} \tilde{g}_l\}$
%\end{equation*}
determines a $\bar{\partial}$-closed $(0,1)$-form $\omega$ on $D_0''$, $\omega:=\bar{\partial} \tilde{g}_l$ on $U_l \cap D_0''$, taking values in bundle $C_{\mathfrak a_{\iota_\varepsilon}}X_0$.

Recall that since $X_0$ is a Stein manifold, there exists a holomorphic Banach vector bundle $E$ such that
$
C_{\mathfrak a_{\iota_\varepsilon}}X_0 \oplus E=X_0 \times B
$
for some Banach space $B$. By $q:X_0 \times B \rightarrow C_{\mathfrak a_{\iota_\varepsilon}}X_0$ and $i:C_{\mathfrak a}X_0 \rightarrow X_0 \times B$, $q \circ i=\Id$, we denote the corresponding bundle morphisms.

Let $\tilde{\omega}:=i(\omega) \in \Lambda_b^{(0,1)}(D_0'',B)$.  Since $q$ is a holomorphic bundle morphism, $\tilde{\omega}$ is  $\bar{\partial}$-closed. Moreover, according to Lemma \ref{lem2} 
\begin{equation*}
\sup_{x \in U_k \cap U_l \cap D_0''}\|g_{kl}(x)\|_{\mathfrak a}<2\varepsilon\quad\text{for all}\quad k,l.
\end{equation*}
Therefore by the construction of $\omega$ and by continuity of $i$ and the fact that $D_0''\Subset X_0$ we obtain that for some $c>0$ independent of $\omega$,
$$
\|\tilde{\omega}\|_{D_0'',B}^{(0,1)} \leq c\varepsilon.
$$
Then by Lemma \ref{hl1} there exists a function $\tilde{\eta} \in \Lambda_b^{0,0}(D_0'',B)$  such that $\bar{\partial} \tilde{\eta}=\tilde{\omega}$ and $$\|\tilde{\eta}\|_{D_0'',B}^{(0,0)} \leq C_1\|\tilde{\omega}\|_{D_0'',B}^{(0,1)} \leq C_1c\varepsilon$$ for some $C_1>0$ independent of $\omega$. We set $\eta:=q(\tilde{\eta}) \in C^1(D_0'',C_{\mathfrak a_{\iota_\varepsilon}}X_0)$. Since $q$ is a holomorphic bundle morphism and $D_0''\Subset X_0$,  $$\bar{\partial}\eta=\omega \quad \text{ and } \quad
\sup_{x\in D_0''}\|\eta(x)\|_{\mathfrak a} \leq C_2C_1c\varepsilon$$ 
for some $C_2>0$ independent of $\omega$.

\noindent Since $D_0 \Subset D_0''$, the restriction $\eta|_{\bar{D}_0}$ is continuous on $\bar{D}_0$.
We define 
\begin{equation*}
F|_{U_k\cap\bar{D}_0}:=F_{\varepsilon,k}|_{U_k \cap \bar{D}_0}-\tilde{g}_k|_{U_k \cap \bar{D}_0}+\eta|_{U_k \cap \bar{D}_0}, \quad 1 \leq k \leq M.
\end{equation*}
It follows that $F \in \mathcal A(D_0,C_{\mathfrak a_{\iota_\varepsilon}}X_0)$ and
\begin{equation*}
\sup_{x \in \bar{D}_0}\|F-F_{\varepsilon,k}\|_{\mathfrak a} \leq 2M\varepsilon+C_2C_1c\varepsilon=: C\varepsilon,
\end{equation*}
as required. This completes the proof of the lemma.
\end{proof}
Assertion ($\ast$) now follows from Lemmas \ref{lem2} and \ref{lem3r}.
The proof of Theorem \ref{approxthm0} is complete.
\end{proof}

\section{Proofs of Propositions \ref{idcohlem}, \ref{closprop}, \ref{equivprop2}, \ref{poincarelem}, Theorems \ref{cartansubm}, \ref{corona} and Lemmas \ref{doesnotdependlem}, \ref{finesheaflem}}

\label{structproofsect}

\SkipTocEntry\subsection{Proof of Proposition \ref{idcohlem}}
According to Proposition \ref{localstructap}, it suffices to prove coherence of the ideal sheaf $I_{Z}\,(\subset \mathcal O_{V_0 \times K})$ of the complex submanifold $Z:=Z_0 \times K$ of $V_0 \times K$, where $V_0 \subset X_0$ and $K \subset \hat{G}_{\mathfrak a}$ are open, $K \in \mathfrak Q$ is an element of the basis of topology $ \mathfrak Q$ of $\hat{G}_{\mathfrak a}$ (see~subsection \ref{topsect}) and
$Z_0 \subset V_0$ is a complex submanifold. Here $\mathcal O_{V_0 \times K}$ denotes the structure sheaf of $V_0 \times K$ (see~subsection \ref{charts}); also, by ${\mathcal O}_{V_0}$ we denote the structure sheaf of $V_0$ and by ${I}_{Z_0} \subset {\mathcal O}_{V_0}$ the ideal sheaf of $Z_0 \subset V_0$. 

By Cartan's theorem (see, e.g.,~\cite{Gun3}) every point in $V_0$ has a neighbourhood over which ${I}_{Z_0}$ has a free resolution. Replacing $V_0$ by a smaller subset, if necessary, we may assume without loss of generality that such resolution is defined over $V_0$:
\begin{equation}
\label{extlemres}
0 \to {\mathcal O}_{V_0}^{m_N} \overset{{\varphi}_{N-1}}{\to} \dots \to 
{\mathcal O}_{V_0}^{m_1} \overset{{\varphi}_{0}}{\to} {I}_{Z_0} \rightarrow 0.
\end{equation}
Further, by the classical Cartan theorem B (see, e.g.,~\cite{Gun3}) every point in $V_0$ has a neighbourhood $U_0 \subset V_0$ biholomorphic to an open polydisk in $\mathbb C^n$ such that the sequence of sections induced by (\ref{extlemres})
\begin{equation}\label{eq8.31}
0 \to \Gamma(U_0,{\mathcal O}_{V_0}^{m_N}) \overset{\bar{\varphi}_{N-1}}{\to} \dots \to 
\Gamma(U_0,{\mathcal O}_{V_0}^{m_1}) \overset{\bar{\varphi}_{0}}{\to} \Gamma(U_0,{I}_{Z_0}) \rightarrow 0
\end{equation}
is exact. 

For an open subset $L \subset K$, $L \in \mathfrak Q$, by $C(L)$ we denote the Fr\'{e}chet space of complex continuous functions on $L$ endowed with the topology of uniform convergence on compact subsets ${N}_k \subset L$, $k\in\mathbb N$, that form an exhaustion of $L$, i.e., ${N}_k \subset {N}_{k+1}$ for all $k$ and $\cup_{k\in\mathbb N}\, {N}_k=L$ (such sets exist by Lemma 7.4(1) in \cite{BK8}). We endow the space $\mathcal O(U_0 \times L)$ defined in subsection \ref{charts} with the topology of uniform convergence on subsets ${W}_{0,k} \times {N}_k$, where $\{{W}_{0,k}\}_{k\in\mathbb N}$ is an exhaustion of $U_0$ by compact subsets, which makes it a Fr\'{e}chet space. 
Then we have 
\begin{equation}
\label{eq8.32}
\Gamma(U_0 \times L,\mathcal O_{V_0 \times K})=: \mathcal O(U_0 \times L)\cong C(L)\otimes\mathcal O(U_0):= C(L) \otimes \Gamma(U_0,{\mathcal O}_{V_0}),
\end{equation}
where $\otimes$ stands for the completion of the symmetric tensor product in the corresponding Fr\'{e}chet space. 

Next, we may assume without loss of generality that $V_0$ is an open polydisk in $\mathbb C^n$ and $Z_0$ is the intersection of a complex subspace of $\mathbb C^n$ with $V_0$. Then using the Taylor series expansion of a holomorphic function on $U_0 \times L$ vanishing on $Z_0 \times K$ (i.e.,~an element of $\Gamma(U_0 \times L,I_{Z_0 \times K})$), we easily obtain that
\begin{equation}
\label{sym1}
\Gamma(U_0 \times L,I_{Z_0 \times K}) \cong C(L) \otimes \Gamma(U_0,{I}_{V_0}).
\end{equation}
By Theorem B in \cite{Bu2} the operation $\otimes$ is an exact functor. Thus, from \eqref{eq8.31}, \eqref{eq8.32} and \eqref{sym1} we obtain that every point in $V_0 \times K$ has a neighbourhood of the form $U_0 \times L$ over which the sequence of sections
\begin{equation}
\label{sym2}
0 \to \Gamma(U_0 \times L,\mathcal O_{V_0 \times K}^{m_N}) \overset{\hat{\varphi}_{N-1}}{\to} \dots \to 
\Gamma(U_0 \times L,\mathcal O_{V_0 \times K}^{m_1}) \overset{ \hat{\varphi}_{0}}{\to} \Gamma(U_0 \times L,I_{Z_0 \times K}) \rightarrow 0
\end{equation}
is exact, where morphisms $\hat{\varphi}_i$ are defined on the corresponding symmetric tensor products by the formula $$\hat{\varphi_i}\left(\sum_{i=1}^l f_i \otimes g_i\right)=\sum_{i=1}^l f_i \otimes \bar{\varphi}_i(g_i), \quad f_i \in C(L), \quad g_i \in \Gamma(U_0,\mathcal O_{V_0}^{m_{i+1}}),$$
and then extended to $C(L) \otimes \Gamma(U_0,\mathcal O_{V_0}^{m_{i+1}})$ by continuity.
Hence, the sequence of sheaves generated by \eqref{sym1}
$$
0 \to \mathcal O_{V_0 \times K}^{m_N} \rightarrow \dots \rightarrow
\mathcal O_{V_0 \times K}^{m_1} \rightarrow I_{Z_0 \times K} \rightarrow 0
$$
is exact. This shows that the sheaf $I_Z$ is coherent.

\SkipTocEntry\subsection{Proof of Proposition \ref{closprop}}

Let us prove the first assertion. 

Let $Y \subset c_{\mathfrak a}X$ be the closure of $\iota(Z)$, where $Z\subset X$ is a complex $\mathfrak a$-submanifold, in $c_{\mathfrak a}X$. We fix a point $y \in Y$ and use notation of Definition \ref{manifolddef}. Since the open cover $\mathcal V$ in the definition of $Z$ is of class ($\mathcal T_{\mathfrak a}$) (and, hence, is the pullback by $\iota$ of an open cover of $c_{\mathfrak a}X$, see~Definition \ref{tfinedef}), there exist an open subset $V \in \mathcal V$,  $V=\iota^{-1}(U)$ for an open neighbourhood $U \subset c_{\mathfrak a}X$ of $y$, and functions $h_i \in \mathcal O_{\mathfrak a}(V)$, $1\le i\le k$, determining $Z\cap V$, i.e., satisfying conditions (1), (2) of Definition \ref{manifolddef}.
By Proposition \ref{basicpropthm}(1)
there exist (uniquely determined) functions $\hat{h}_i \in \mathcal O(U)$ such that $h_i=\iota^*\hat{h}_i$ for all $i$.
It follows from condition (2) of Definition \ref{manifolddef} and the fact that $\iota(V)$ is dense in $U$ that functions $\hat{h}_i$ satisfy condition (2) of Definition \ref{defap} at points of $U\cap Y$. Therefore, since $y \in Y$ is arbitrary, to complete the proof it suffices to show that $U \cap Y=\hat{Y}_U$, where $\hat{Y}_U \subset U$ denotes the common zero locus of functions $\hat{h}_i|_{U}$, $1\le i\le k$. 

Indeed, using the argument of the proof of Proposition \ref{localstructap} and shrinking $U$, if necessary, we obtain that there exists a biholomorphism $\Phi \in \mathcal O(U_0 \times K,U)$, where $U_0 \subset X_0$, $K \subset \hat{G}_{\mathfrak a}$ are open, and a closed submanifold $Z_0 \subset U_0$ such that
$\Phi^{-1}(\hat{Y}_U)=Z_0 \times K$ and $\Phi(U_0 \times (K \cap j(G)))=U \cap \iota(X)$.
In particular, since $h_i=\iota^*\hat{h}_i$ for all $i$, we have $\Phi(Z_0 \times (K \cap j(G)))=U \cap \iota(Z)$. Hence, since $Z_0 \times (K \cap j(G))$ is dense in $Z_0 \times K$ (see~subsection \ref{constrsect}), $U \cap \iota(Z)$ is dense in $\hat{Y}_U$, i.e.,~$U \cap Y=\hat{Y}_U$, as required. The proof of the first assertion is complete.

The second assertion follows easily from Definitions \ref{manifolddef}, \ref{defap} and Proposition \ref{basicpropthm}(2).

\SkipTocEntry\subsection{Proof of Proposition \ref{equivprop2}}

First, let $f$ be a holomorphic $\mathfrak a$-function on $Z:=\iota^{-1}(Y)$ in the sense of Definition \ref{holdef2}, i.e., there is a function $F \in C_{\mathfrak a}(X)$ such that $F|_Z=f$. 
By Proposition \ref{basicpropthm}(1) there exists a function $\hat{F} \in C(c_{\mathfrak a}X)$ such that $\iota^*\hat{F}=F$. We set $\hat{f}:=\hat{F}|_{Y}$. Since $\iota^*\hat{f}=f$, we obtain $\hat{f} \in \mathcal O(Y)$ (see~Definition \ref{defa}), as required.

Now, let $\hat{f} \in \mathcal O(Y)$. 
Since $c_{\mathfrak a}X$ is a normal space, by the Tietze-Urysohn extension theorem there exists a function $\hat{F} \in C(c_{\mathfrak a}X)$ such that 
$\hat{F}|_{Y}=\hat{f}$. By Proposition \ref{basicpropthm}(2) $F:=\iota^*\hat{F}$ belongs to $C_{\mathfrak a}(X)$. Since $F|_{Z}=f$, function $f$ $(=\iota^*\hat{f})$ is a holomorphic $\mathfrak a$-function on $Z$ in the sense of Definition \ref{holdef2}.

\SkipTocEntry\subsection{Proof of Theorem \ref{cartansubm}}

We will use notation and results of subsection \ref{charts}.

By definition, every point in $Y$ has a neighbourhood $V \subset Y$ over which, for every $N \geq 1$, there exists a free resolution
\begin{equation}
\label{coh0_}
\mathcal O_Y^{m_{4N}}|_V \overset{\varphi_{4N-1}}{\to} \dots \overset{\varphi_2}{\to} \mathcal O_Y^{m_{2}}|_V \overset{\varphi_1}{\to} \mathcal O_Y^{m_{1}}|_V \overset{\varphi_0}{\to} \mathcal A|_V \to 0.
\end{equation}
We need to show that sheaf $\tilde{\mathcal A}$ is coherent on $c_{\mathfrak a}X$, i.e., that every point $x \in c_{\mathfrak a}X$ has a neighbourhood $U \subset c_{\mathfrak a}X$ over which sheaf $\tilde{\mathcal A}|_U$ has free resolutions of any finite length. If $x \in c_{\mathfrak a}X \setminus Y$, then we can choose $U$ such that $U \cap Y=\varnothing$; hence, $\tilde{\mathcal A}|_U=0$ trivially has free resolutions of any finite length. Now, let $x \in Y$. Shrinking $V$, if necessary, and applying Proposition \ref{localstructap} we can choose $U \ni x$ such that $V=Y \cap U$ and there exists a biholomorphism that maps $U$ onto $U_0 \times K$, where $U_0 \subset X_0$ is biholomorphic to an open polydisk in $\mathbb C^n$, $K \subset \hat{G}_{\mathfrak a}$ is open, and $V$ is mapped onto $V_0 \times K$, where $V_0 \subset U_0$ is a  complex submanifold. Thus, applying this biholomorphism we may assume that $$U=U_0 \times K, \quad V=V_0 \times K.$$

We will need the following

\begin{lemma}
\label{cartansubmlem}
The trivial extension $\widetilde{\mathcal O_Y}$ of $\mathcal O_Y$ has free resolutions of any finite length over $U$.
\end{lemma}
\begin{proof}%[Proof of Lemma]
By definition, $\widetilde{\mathcal O_Y}|_U$ is isomorphic to the quotient sheaf $\mathcal O_U/I_V$, where $\mathcal O_U:=\mathcal O|_U$ is the sheaf of germs of holomorphic functions on $U$, $I_V \subset \mathcal O_U$ is the ideal sheaf of $V \subset U$, i.e., we have an exact sequence
\begin{equation}
\label{cartansubm1}
0 \rightarrow I_V \rightarrow \mathcal O_U \rightarrow \widetilde{\mathcal O_Y}|_U \rightarrow 0.
\end{equation}
Following the argument of the proof of Proposition \ref{idcohlem}, we obtain that sheaf $I_V$ has free resolutions of any finite length over $U$. Then using a free resolution of $I_V$ over $U$ of length $N$, we extend (\ref{cartansubm1}) to a free resolution of $\widetilde{\mathcal O_Y}|_U$ of length $N+1$. Since $N$ was chosen arbitrarily, this completes the proof.
%Following the proof of Lemma \ref{idcohlem} we may assume that sheaf $I_V$ has free resolutions over $U$ of any finite length. Hence $H^1(U,I_V)=0$ by Proposition 5.3.1(2) in \cite{BK8}, and we have an exact sequence of sections
%$$
%0 \rightarrow \Gamma(U,I_V) \rightarrow \Gamma(U,\mathcal O_U) \rightarrow \Gamma(U,\tilde{\mathcal O_Y}) \rightarrow 0,
%$$
%which means that sequence (\ref{cartansubm1}) is \textit{completely exact} (cf.~Definition 5.3.22 in \cite{BK8}). Therefore, Lemma 5.3.24 (``Three lemma'') in \cite{BK8} applies: since $I_V$, $\mathcal O_U$ have free resolutions of any finite length over $U$, sheaf $\tilde{O_Y}|_U$ also has free resolutions of any finite length over $U$.
\end{proof}

Now we finish the proof of Theorem \ref{cartansubm}. Since sequence (\ref{coh0_}) is exact, the corresponding sequence of trivial extensions 
\begin{equation}
\label{coh0__}
\widetilde{\mathcal O}_Y^{m_{4N}}|_U \overset{\tilde{\varphi}_{4N-1}}{\to} \dots \overset{\tilde{\varphi}_2}{\to} \widetilde{\mathcal O}_Y^{m_{2}}|_U \overset{\tilde{\varphi}_1}{\to} \widetilde{\mathcal O}_Y^{m_{1}}|_U \overset{\tilde{\varphi}_0}{\to} \tilde{\mathcal A}|_U \to 0
\end{equation}
is also exact; here $\tilde{\varphi}_i:=e \circ \varphi_i \circ r_{U,V}$, where $r_{U,V}:\widetilde{\mathcal O}_Y|_U \rightarrow \mathcal O_Y|_{V}$ is the restriction homomorphism, and $e:\mathcal B \rightarrow \tilde{\mathcal B}$ is the canonical homomorphism that maps an analytic sheaf $\mathcal B$ on $V$ to its trivial extension $\tilde{\mathcal B}$ on $U$.

By Lemmas 7.15 and 7.17 in \cite{BK8} sequence (\ref{coh0_}) truncated to the $N$-th term is completely exact over $V$ (i.e., the corresponding sequence of sections is exact, see~Definition 7.22 in \cite{BK8}), therefore sequence (\ref{coh0__}) truncated to the $N$-th term is completely exact as well. Since by Lemma \ref{cartansubmlem} each sheaf $\widetilde{\mathcal O}_Y^{m_i}$ in (\ref{coh0__}) has free resolutions over $U$ of any finite length, Lemma 9.3 in \cite{BK8} implies that $\tilde{\mathcal A}|_U$ has free resolutions over $U$ of any finite length as well. This implies that sheaf $\tilde{\mathcal A}$ is coherent on $c_{\mathfrak a}X$.

\SkipTocEntry\subsection{Proof of Lemma \ref{doesnotdependlem}}
We use the following consequence of Theorem 4.6 in \cite{BK8}:

\begin{lemma}
\label{le21.1}
Let $V_{0,1}$, $V_{0,2} \subset \mathbb C^n$ be open and connected and $K_1$, $K_2 \subset \hat{G}_{\mathfrak a}$ be open. A map $F \in \mathcal O(V_{0,1} \times K_1,V_{0,2} \times K_2)$ admits presentation
$
F(z,\omega)=\bigl(f_\omega(z),h(\omega) \bigr)$, $(z,\omega) \in V_{0,1} \times K_1,
$
where $f_\omega \in \mathcal O(V_{0,1},V_{0,2})$ depend continuously on $\omega \in K_1$ and $h \in C(K_1,K_2)$.
\end{lemma}
\begin{proof}%[Proof of Lemma]
Let $\pi^1:V_{0,2} \times K_2 \rightarrow V_{0,2}$, $\pi^2:V_{0,2} \times K_2 \rightarrow K_2$ be the natural projections. By \cite[Thm.~4.6]{BK8} $(\pi^2 \circ F)(\cdot,\omega)\equiv const$ for all $\omega\in K_1$. Thus, we define $h\in C(K_1,K_2)$ as $h(\omega):=(\pi^2 \circ F)(\cdot,\omega)$, $\omega\in K_1$, and $f_\omega(z):=(\pi^1 \circ F)(z,\omega)$, $(z,\omega) \in V_{0,1} \times K_1$.
\end{proof}

Now, suppose $\varphi_i \in \mathcal O(V,V_{i} \times K_i)$, where $V_{i} \subset \mathbb C^n$ are open and connected, $K_i \subset \hat{G}_{\mathfrak a}$ are open ($i=1,2$), are coordinate maps of an open  subset $V\subset Y$.
Let $F:=\varphi_2 \circ \varphi_1^{-1}$. By the above lemma 
$F(z,\omega)=\bigl(f_\omega(z),h(\omega) \bigr)$ ($(z,\omega) \in V_{1} \times K_1$), where $f_\omega \in \mathcal O(V_{1},V_{2})$ depend  continuously on $\omega$ and $h \in C(K_1,K_2)$. Since $F$ is a biholomorphism, $h: K_1\rightarrow K_2$ is a homeomorphism. Then replacing $F$, if necessary,  by  the holomorphic map $G\circ F$, where $G(z,\omega):=\bigl(z,h^{-1}(w)\bigr)$, $(z,\omega)\in V_2\times K_2$, we may assume without loss of generality that $K_2=K_1=:K$ and $h=\Id$. 

In order to prove the lemma it suffices to show that for each $p \in C^\infty(V_{0,2} \times K)$ its pullback $F^*p \in C^\infty(V_{0,1} \times K)$. Indeed, we have
$(F^*p)(z,\omega)=p(f_\omega(z),\omega)$.
Since $V_{1} \ni z \mapsto f_{\cdot}(z)=F(z,\cdot)$ is a holomorphic and, hence, a $C^\infty$ function taking values in the Fr\'{e}chet space $C(K)$, the required result follows by the chain rule.

\SkipTocEntry\subsection{Proof of Lemma \ref{finesheaflem}}

(For similar arguments, see \cite{Br6}.)

\begin{lemma}
For an open cover $\mathcal U=\{U_\gamma\}$ of $Y$ there exists an open cover $\mathcal V=\{\hat{\Pi}(V_{0,\alpha},K_{\alpha,\beta}) : V_{0,\alpha} \subset X_0,\, K_{\alpha,\beta}\subset \hat{G}_{\mathfrak a}\ \text{are open}\}$ of $c_{\mathfrak a}X$ such that $\{V_{0,\alpha}\}$ is a locally finite open cover of $X_0$, for each $\alpha$ the number of distinct sets $K_{\alpha,\beta}$ is finite and $\cup_{\beta}\, \hat{\Pi}(V_{0,\alpha},K_{\alpha,\beta})=\bar{p}^{-1}(V_{0,\alpha})$, and 
$\{Y \cap \hat{\Pi}(V_{0,\alpha},K_{\alpha,\beta})\}$ is a refinement of $\mathcal U$.
\end{lemma}
\begin{proof}%[Proof of Lemma]
By the definition of the relative topology of $Y$, there exists a collection $\tilde{\mathcal U}=\{\tilde{U}_\gamma\}$ of open subsets of $c_{\mathfrak a}X$ such that $U_\gamma=\tilde{U}_\gamma \cap Y$ for all $\gamma$. Further, since $Y\subset c_{\mathfrak a}X$ is closed,  $\tilde{\mathcal U}\cup\{c_{\mathfrak a}X\setminus Y\}$ is an open cover of $c_{\mathfrak a}X$. By the definition of topology on $c_{\mathfrak a}X$ the latter cover admits a refinement by sets of the form $\hat{\Pi}(V_{0},K)$, where $V_{0} \subset X_0$, $K\subset \hat{G}_{\mathfrak a}$ are open. Since $c_{\mathfrak a}X$ has compact fibres, we may choose this refinement $\{\hat{\Pi}(V_{0,\alpha},K_{\alpha,\beta})\}$ so that $\{V_{0,\alpha}\}$ is a locally finite open cover of $X_0$ and for each $\alpha$ the number of distinct sets $K_{\alpha,\beta}$ is finite and $\cup_{\beta}\, \hat{\Pi}(V_{0,\alpha},K_{\alpha,\beta})=\bar{p}^{-1}(V_{0,\alpha})$. By our construction, $\{Y \cap \hat{\Pi}(V_{0,\alpha},K_{\alpha,\beta})\}$ is a refinement of $\mathcal U$.
\end{proof}

The open cover $\mathcal V$ introduced in the lemma admits a subordinate partition of unity $\{\nu_{\alpha,\beta}\}$,
$\nu_{\alpha,\beta}:=\bar{p}^*\rho_\alpha\cdot\pi_\alpha^*\mu_{\alpha,\beta}$,
 where $\{\rho_\alpha\} \subset C^\infty(V_{0,\alpha})$ is a partition of unity subordinate to cover $\{V_{0,\alpha}\}$ of $X_0$, $\{\mu_{\alpha,\beta}\} \subset C(\bar{p}^{-1}(x_\alpha))$, $x_\alpha\in V_{0,\alpha}$ is fixed,  is a partition of unity subordinate to cover $\{\hat{\Pi}(V_{0,\alpha},K_{\alpha,\beta})\cap\bar{p}^{-1}(x_\alpha)\}_\beta$ of  $\bar{p}^{-1}(x_\alpha)$ and $\pi_\alpha: \bar{p}^{-1}(V_{0,\alpha})\rightarrow \bar{p}^{-1}(x_\alpha)$ is the continuous projection defined in local coordinates $(x,\omega)$ on $\bar{p}^{-1}(V_{0,\alpha})\, (\cong V_{0,\alpha}\times \hat{G}_{\mathfrak a})$ as $\pi_\alpha(x,\omega):=(x_{\alpha},\omega)$.
By definition $\nu_{\alpha,\beta} \in C^\infty(\hat{\Pi}(V_{0,\alpha},K_{\alpha,\beta}))$; hence, the restriction of $\{\nu_{\alpha,\beta}\}$ to $Y$ is a $C^\infty$ partition of unity subordinate to $\mathcal U$ (see~subsection \ref{derhamsect}).

\SkipTocEntry\subsection{Proof of Proposition \ref{poincarelem}}
\label{poincare}

For a point $x\in Y$ consider its open neighbourhood $V$ for which there exists a biholomorphic map $\varphi:V\rightarrow Z_0 \times K$, where $Z_0 \subset \mathbb C^p$ is an open ball and $K \subset \hat{G}_{\mathfrak a}$ is open (see Proposition \ref{localstructap}). We choose an open neighbourhood $W\Subset V$ of $x$ so that $\varphi(W)=Z_0'\times K'$, where $Z_0'\Subset Z_0$ is an open ball of the same center as $Z_0$ and $K' \Subset K$ is an open subset. Then under the identification of $V$ with $Z_0\times K$ by $\varphi$ the restriction to $W$ of the space of $C^\infty$ $\bar{\partial}$-closed $(p,k+1)$-forms on $V$ is identified with a subspace of the space of 
$C^\infty$ $\bar{\partial}$-closed $(p,k+1)$-forms on $Z_0'$ with values in the Banach space $C_b(K')$ of bounded continuous functions on $K$ endowed with $\sup$-norm (see subsection \ref{derhamsect} for the corresponding definitions). According to Lemma \ref{hl1} such Banach-valued forms on $Z_0'$ are $\bar{\partial}$-exact. This completes the proof of the proposition.

\SkipTocEntry\subsection{Proof of Theorem \ref{corona}}

%\begin{lemma}
%\label{propderiv}
%Let $U_0 \subset \mathbb C^n$, $K \subset \hat{G}_{\mathfrak a}$ be open, $f \in \mathcal O(U_0 \times K)$ (cf.~Section \ref{notsect}).
%Then $$\frac{\partial f}{\partial z_k} \in \mathcal O(U_0 \times K), \quad 1 \leq k \leq n, \quad z=(z_1,\dots,z_n) \in U_0.$$
%\end{lemma}

%\begin{proof}[Proof of Lemma]
%By definition, $f(\cdot,\xi) \in \mathcal O(U_0)$ for every $\xi \in K$.
%Let $\Delta_1 \Subset \Delta_2 \Subset U_0$ be open polydisks. By Cauchy formula
%\begin{equation}
%\label{cauchy}
%\frac{\partial f}{\partial z_k}(z,\xi)=-\frac{1}{(2\pi i)^n}\int_{\partial \Delta_1} \frac{f(\zeta,\xi)}{(z_1-\zeta_1) \dots (z_k-\zeta_k)^2 \dots (z_n-\zeta_n)}d\zeta, \quad z \in \Delta_1,
%\end{equation}
%where $\partial \Delta_1$ stands for the boundary torus of $\Delta_1$. Since $f$ is continuous and $\Delta_1$ is relatively compact in $U_0$, by Montel theorem $f(\cdot,\xi_k) \rightarrow f(\cdot,\xi)$ uniformly on $\Delta_1$ if $\xi_k \rightarrow \xi$ in $K$. Therefore, by (\ref{cauchy}) $$\frac{\partial f}{\partial z_k}(\cdot,\xi_k) \rightarrow \frac{\partial f}{\partial z_k}(\cdot,\xi) \quad \text{ uniformly on } \Delta_2.$$ It follows that  $\frac{\partial f}{\partial z_k}  \in C(U_0 \times K)$ and $\frac{\partial f}{\partial z_k} (\cdot,\xi) \in \mathcal O(U_0)$ ($\xi \in K$), hence $\frac{\partial f}{\partial z_k} \in \mathcal O(U_0\times K)$.
%\end{proof}

\begin{lemma}
\label{divlem}
Let $U_0 \subset X_0$, $K \subset \hat{G}_{\mathfrak a}$ be open, $f \in \mathcal O(U_0 \times K)$ 
be such that $\nabla_z f(z,\eta) \neq 0$ for all $(z,\eta) \in Z_f:=\{(z,\eta) \in U_0 \times K: f(z,\eta)=0\}$. 

If $g \in \mathcal O(U_0 \times K)$ vanishes on $Z_f$, then $h:=g/f \in \mathcal O(U_0 \times K)$.
\end{lemma}

(The proof follows straightforwardly from Proposition \ref{localstructap}.)

%\begin{proof}[Proof of Lemma]
%Since this is a local statement, by Proposition \ref{localstructap} we may assume that $U_0$ is an open polydisk in $\mathbb C^n$, and $f(z_1,\dots,z_n,\xi)=z_1$ on $U_0 \times K$. 
%Now, for each $\eta \in K$ we have
%\begin{equation}
%\label{vanishcorpres}
%g(z_1,\dots,z_n,\eta)=z_1 \int_{0}^1 \frac{\partial g(tz_1,\dots,z_n,\xi)}{\partial z_1}dt, \quad (z_1,\dots,z_n) \in U_0.
%\end{equation}
%%Here $(z_1,\dots,z_n) \in U_0$ implies $(tz_1,\dots,z_n) \in U_0$, $0 \leq t \leq 1$, since $U_0$ is a polydisk.
%By Lemma \ref{propderiv} function $$(z_1,\dots,z_n,\eta) \mapsto \frac{\partial g(tz_1,\dots,z_n,\xi)}{\partial z_1}$$ belongs to $\mathcal O(U_0 \times K)$ for all $t \in [0,1]$. 
%By Montel theorem, if
%$\eta_k \mapsto \eta$ in $K$, then $$\partial g(z_1,\dots,z_n,\eta_k)/\partial z_1 \rightarrow \partial g(z_1,\dots,z_n,\eta)/\partial z_1$$ uniformly over a smaller polydisk $\Delta \Subset U_0$. Hence, $$\frac{\partial g(tz_1,\dots,z_n,\eta_k)}{\partial z_1} \rightarrow \frac{\partial g(tz_1,\dots,z_n,\eta)}{\partial z_1}$$ as $k \rightarrow \infty$ uniformly over $[0,1] \times \Delta$.
%Therefore, by shrinking $U_0$ if necessary, we obtain that  function $h(z,\eta)$ ($z \in U_0$, $\eta \in K$), defined to be the second multiple in (\ref{vanishcorpres}),
%is continuous in $\eta$ and, therefore belongs to $\mathcal O(U_0 \times K)$, as required.
%\end{proof}

\begin{proof}[Proof of Theorem]
By Proposition \ref{basicpropthm}(2) $M_X$ is homeomorphic to the maximal ideal space of $\mathcal O(c_{\mathfrak a}X)$. 
It follows, e.g., from Theorem \ref{extapthm}, that algebra $\mathcal O(c_{\mathfrak a}X)$ separates points of $c_{\mathfrak a}X$, therefore we have a continuous injection $c_{\mathfrak a}X \hookrightarrow M_X$ defined via point evaluation homomorphisms. 
Let us show that this map is surjective.

The transpose to the pullback homomorphism $\bar{p}^*:\mathcal O(X_0)\rightarrow  \mathcal O(c_{\mathfrak a}X)$ is a map $\bar{p}_*: M_X\rightarrow M_{X_0}$, where the latter is the maximal ideal space of algebra $\mathcal O(X_0)$. Since $X_0$ is Stein, $M_{X_0}$ can be naturally identified with $X_0$ (see, e.g., \cite{GR}) so that $\bar{p}_*|_{c_{\mathfrak a}X}=\bar{p}$. Hence, for $\varphi\in M_X$ there exists a point $x_0\in X_0$ such that $\bar{p}_*(\varphi)=\delta_{x_0}$, the evaluation homomorphism at point $x_0$.

Next, there exists a function $h \in \mathcal O(X_0)$ such that $X_0^{n-1}:=\{x \in X_0: h(x)=0\}$, $n:={\rm dim}_{\mathbb C}X$, is a non-singular complex hypersurface, $dh(z)\ne 0$ for each $z\in X_0^{n-1}$, and $x_0 \in X_0^{n-1}$, see \cite{Forst}. We set $X^{n-1}:=p^{-1}(X_0^{n-1})$ and $c_{\mathfrak a}X^{n-1}:=\bar{p}^{-1}(X_0^{n-1})$. 
Now, if $f \in \mathcal O(c_{\mathfrak a}X)$ is identically zero on $c_{\mathfrak a}X^{n-1}$, then $\varphi(f)=0$. Indeed, by Lemma \ref{divlem} function $\tilde{f}:=f/\bar{p}^*h \in \mathcal O(c_{\mathfrak a}X)$, hence, $$\varphi(f)=\varphi(\tilde{f})\varphi(\bar{p}^*h)=\varphi(\tilde{f})\delta_{x_0}(h)=0.$$
Thus, there exists a homomorphism $\varphi_1$ of the quotient algebra $\mathcal O(c_{\mathfrak a}X)/I_{c_{\mathfrak a}X^{n-1}}$, where $I_{c_{\mathfrak a}X^{n-1}}$ is the ideal of holomorphic functions in $\mathcal O(c_{\mathfrak a}X)$ vanishing on $c_{\mathfrak a}X^{n-1}$, such that $\varphi=\varphi_1\circ q_1$, where $q_1: \mathcal O(c_{\mathfrak a}X)\rightarrow \mathcal O(c_{\mathfrak a}X)/I_{c_{\mathfrak a}X^{n-1}}$ is the quotient homomorphism. According to Theorem \ref{extapthm}, we have a natural isomorphism $$\mathcal O(c_{\mathfrak a}X)/I_{c_{\mathfrak a}X^{n-1}} \cong \mathcal O(c_{\mathfrak a}X^{n-1})$$
defined by restrictions of functions in $\mathcal O(c_{\mathfrak a}X)$ to $c_{\mathfrak a}X^{n-1}$; hence $\varphi_1$ can be identified with an element of the maximal ideal space of algebra $\mathcal O(c_{\mathfrak a}X^{n-1})$.

Starting with $c_{\mathfrak a}X^{n-1}$ instead of $c_{\mathfrak a}X$ we proceed similarly to
define flags of complex submanifolds $X_0^k\subset X_0$, $X^k\subset X$, $c_{\mathfrak a}X^k\subset c_{\mathfrak a}X$ of codimension $n-k$ and homomorphisms $\varphi_{n-k}:  \mathcal O(c_{\mathfrak a}X^{k})\rightarrow\mathbb C$ such that
$\varphi_{n-k-1}=\varphi_{n-k}\circ q_{n-k}$ ($0 \leq k \leq n-1$), where $q_{n-k}:\mathcal O(c_{\mathfrak a}X^{k+1})\rightarrow \mathcal O(c_{\mathfrak a}X^{k+1})/I_{c_{\mathfrak a}X^{k}}= \mathcal O(c_{\mathfrak a}X^{k})$ are the quotient homomorphisms.

By the definition, $\varphi_n$ is an element of the maximal ideal space of algebra $\mathcal O(c_{\mathfrak a}X^0)$, where $X_0^0=\{x_0,x_1,\dots\}$ is a discrete set. Clearly, $\mathcal O(c_{\mathfrak a}X^0) \cong \sqcup_{i \geq 0} C(\bar{p}^{-1}(x_i))$. 
Moreover, if $f\in I_{x_0}\subset \mathcal O(c_{\mathfrak a}X^0)$, the ideal of functions vanishing on $\bar{p}^{-1}(x_0)$, then $f=f\cdot\bar{p}^*g_{x_0}$, where $g_{x_0}\in \mathcal O(X_0^0)$, $g_{x_0}(x_i)=1-\delta_{0i}$ (Kronecker delta), so that
\[
\varphi_n(f)=\varphi_n(f)\varphi_n(\bar{p}^*g_{x_0})=\varphi_n(f)\varphi(\bar{p}^*g)=\varphi_n(f)g_{x_0}(x_0)=0;
\]
here $g\in\mathcal O(X_0)$ is such that $g|_{X_0^0}=g_{x_0}$.

Thus, there exists a homomorphism $\varphi_{n+1}:\mathcal O(c_{\mathfrak a}X^0)/I_{x_0}=C(\bar{p}^{-1}(x_0))\rightarrow\mathbb C$ such that
$\varphi_n=\varphi_{n+1}\circ q_{n+1}$, where $q_{n+1}:\mathcal O(c_{\mathfrak a}X^0)\rightarrow C(\bar{p}^{-1}(x_0))$ is the quotient homomorphism. Since $\bar{p}^{-1}(x_0)$ is compact Hausdorff, the maximal ideal space of $C(\bar{p}^{-1}(x_0))$ is homeomorphic to $\bar{p}^{-1}(x_0)$.
In particular, $\varphi_{n+1}=\delta_{\omega}$ for some $\omega\in \bar{p}^{-1}(x_0)$.

Finally, we have $\varphi(f)=\varphi_{n+1}(f|_{\bar{p}^{-1}(x_0)})=f(\omega)=\delta_\omega(f)$, $f\in \mathcal O(c_{\mathfrak a}X)$ as required.
This shows that the natural map $c_{\mathfrak a}X\rightarrow M_X$ is a continuous bijection. It is easily seen that this map is a homeomorphism since $c_{\mathfrak a}X$ is locally compact and $X_0\rightarrow M_{X_0}$ is a homeomorphism. 

The proof of the theorem is complete.
\end{proof}

\end{document}